\documentclass[a4paper,reqno,12pt]{amsart}
\numberwithin{equation}{section}
\usepackage[utf8]{inputenc}
\usepackage[T1]{fontenc}
\usepackage{siunitx}
\usepackage{braket}
\usepackage{tikz}
\usetikzlibrary{tqft}
\usepackage{tikz-cd}
\usepackage{amssymb}
\usepackage{amsfonts}
\usepackage{amsmath}
\usepackage{graphicx,epsfig}
\usepackage[enableskew]{youngtab}
\usepackage{fourier}
\usepackage{comment}
\usepackage{todonotes}
\usepackage{mathtools}
\usepackage{hyperref}

\usepackage{cite}
\newcommand{\emptybox}{\mbox{ }}

\DeclareMathOperator{\Tr}{Tr}
\DeclareMathOperator{\End}{End}

\newtheorem{theorem}{Theorem}

\newtheorem{coro}[theorem]{Corollary}

\newtheorem{defi}[theorem]{Definition}
\newtheorem{exam}[theorem]{Example}

\newtheorem{lemma}[theorem]{Lemma}

\newtheorem{prop}[theorem]{Proposition}
\newtheorem{rema}[theorem]{Remark}

\numberwithin{equation}{section}
\numberwithin{theorem}{section}

\newcommand{\cH}{\mathcal{H}}

\newcommand{\cS}{\mathcal{S}}
\newcommand{\cT}{\mathcal{T}}

\newcommand{\cN}{\mathcal{N}}
\newcommand{\cP}{\mathsf{P}}
\newcommand{\cQ}{\mathsf{Q}}
\newcommand{\mb}[1]{\mathbb{#1}}
\newcommand{\mc}[1]{\mathcal{#1}}
\newcommand{\mf}[1]{\mathfrak{#1}}
\newcommand{\bs}[1]{\boldsymbol{#1}}
\newcommand{\op}[1]{\operatorname{#1}}

\newcommand{\g}{\mathfrak{gl}}
\newcommand{\h}{\mathfrak{h}}
\newcommand{\s}{\mathfrak{sl}}
\newcommand{\gh}{\mathfrak{\widehat{gl}}}
\newcommand{\sh}{\mathfrak{\widehat{sl}}}
\newcommand{\V}{\mathcal{V}}
\newcommand{\bV}{\mathcal{V}}
\newcommand{\R}{R}
\newcommand{\RR}{R'}
\newcommand{\rS}{\mathsf{S}}

\newcommand{\GS}{G(n,1,k)}

\newcommand{\mL}{L}
\newcommand{\alc}{\mathcal{A}^+_{k}(n)}
\newcommand{\balc}{\mathcal{B}_{k}(n)}
\newcommand{\salc}{\mathcal{A}^{++}_{k}(n)}
\newcommand{\z}{{\rm z}}
\newcommand{\e}{{\rm e}}
\newcommand{\rot}{\operatorname{rot}}
\newcommand{\veps}{\varepsilon}

\author{Christian Korff and David Palazzo}
\address{School of Mathematics and Statistics, University of Glasgow, Glasgow G12 8QQ, UK}
\email{christian.korff@glasgow.ac.uk}
\email{david.pala89@gmail.com}
\begin{document}
\title[Cylindric symmetric functions]{Cylindric symmetric functions and positivity}
\begin{abstract}
We introduce new families of cylindric symmetric functions as subcoalgebras in the ring of symmetric functions $\Lambda$ (viewed as a Hopf algebra) which have non-negative structure constants. Combinatorially these cylindric symmetric functions are defined as weighted sums over cylindric reverse plane partitions or - alternatively - in terms of sets of affine permutations. We relate their combinatorial definition to an algebraic construction in terms of the principal Heisenberg subalgebra of the affine Lie algebra $\sh_n$ and a specialised cyclotomic Hecke algebra. Using Schur-Weyl duality we show that the new cylindric symmetric functions arise as matrix elements of Lie algebra elements in the subspace of symmetric tensors of a particular level-0 module which can be identified with the small quantum cohomology ring of the $k$-fold product of projective space. The analogous construction in the subspace of alternating tensors gives the known set of cylindric Schur functions which are related to the small quantum cohomology ring of Grassmannians.  We prove that cylindric Schur functions form a subcoalgebra in $\Lambda$  whose structure constants are the 3-point genus 0 Gromov-Witten invariants. We show that the new families of cylindric functions obtained from the subspace of symmetric tensors also share the structure constants of a symmetric Frobenius algebra, which we define in terms of tensor multiplicities of the generalised symmetric group $G(n,1,k)$. 
\end{abstract}
\maketitle

\tableofcontents
\addtocontents{toc}{\protect\setcounter{tocdepth}{2}}
\vspace{-0.5cm}

\section{Introduction}
The ring of symmetric functions $\Lambda=\underleftarrow{\lim}\, \Lambda_k$ with $\Lambda_k=\mb{C}[x_1,\ldots,x_k]^{\rS_k}$ lies in the intersection of representation theory, algebraic combinatorics and geometry. In order to motivate our results and set the scene for our discussion we briefly recall a classic result for the cohomology of Grassmannians, which showcases the interplay between the mentioned areas based on symmetric functions.

\subsection{Schur functions and cohomology}
A distinguished $\mb{Z}$-basis of $\Lambda$ is given by Schur functions $\{s_\lambda~|~\lambda\in\cP^+\}$ with $\cP^+$ the set of integer partitions. In the context of Schur-Weyl duality the associated Schur polynomials, the projections of $s_\lambda$ onto $\Lambda_k$, play a prominent role as characters of irreducible polynomial representations of $GL(k)$. In particular, the product expansion $s_\lambda s_\mu=\sum_{\nu\in\cP^+} c_{\lambda\mu}^\nu s_\nu$ of two Schur functions yields the Littlewood-Richardson coefficients $c_{\lambda\mu}^\nu\in\mb{Z}_{\geq 0}$, which describe the tensor product multiplicities of the mentioned $GL(k)$-representations. There is a purely combinatorial rule how to compute these coefficients in terms of so-called Littlewood-Richardson tableaux, which are a particular subclass of reverse plane partitions; see e.g. \cite{fulton1997}. 

The positivity of Littlewood-Richardson coefficients can also be geometrically explained: let $\mc{I}_{n,k}$ denote the ideal generated by those $s_\lambda$, where the Young diagram of the partition $\lambda$ does not fit inside $(n-k)^k$ (the bounding box of height $k$ and width $n-k$). Then the quotient $\Lambda/\mc{I}_{n,k}$ is known to be isomorphic to the cohomology ring $H^*(\op{Gr}(k,n))$ of the Grassmannian $\op{Gr}(k,n)$, the variety of $k$-dimensional hyperplanes in $\mb{C}^n$. Under this isomorphism Schur functions are mapped to Schubert classes and, thus, the Littlewood-Richardson coefficients are the intersection numbers of Schubert varieties. 

Alternatively, one can obtain the same coefficients via so-called skew Schur functions $s_{\lambda/\mu}$, which are the characters of reducible $GL(k)$-re\-pre\-sen\-tations. Noting that $\Lambda$ carries the structure of a (positive self-dual) Hopf algebra \cite{zelevinsky1981representations}, the image of a Schur function under the coproduct $\Delta:\Lambda\to\Lambda\otimes\Lambda$ can be used to define skew Schur functions via
\begin{equation}\label{skewSchur}
\Delta(s_\lambda)=\sum_{\mu\in\cP^+} s_{\lambda/\mu}\otimes s_\mu,\qquad
s_{\lambda/\mu}=\sum_{\nu\in\cP^+} c_{\mu\nu}^\lambda s_\nu\;. 
\end{equation}
Here the second expansion is a direct consequence of the fact that $\Lambda$ viewed as a Hopf algebra is self-dual with respect to the Hall inner product; see Appendix A. Recall that $H^*(\op{Gr}(k,n))$ carries the structure of a symmetric Frobenius algebra with respect to the non-degenerate bilinear form induced via Poincar\'e duality (see e.g. \cite{abrams2000quantum}). In particular, $H^*(\op{Gr}(k,n))$ is also endowed with a coproduct. It follows from $s_{\lambda/\mu}=0$ if $\mu\not\subset\lambda$ and \eqref{skewSchur} that the (finite-dimensional) subspace in $\Lambda$ spanned by the Schur functions $\{s_\lambda~|~\lambda\subset (n-k)^k\}$ viewed as a {\em coalgebra} is isomorphic to $H^*(\op{Gr}(k,n))$. In particular, there is no quotient involved, the additional relations are directly encoded in the combinatorial definition of skew Schur functions as sums over skew tableaux. 

In this article, we generalise this result and identify what we call {\em positive subcoalgebras} of $\Lambda$: subspaces $M\subset\Lambda$ that possess a distinguished basis $\{f_{\lambda}\}\subset M$ satisfying $\Delta(f_\lambda)=\sum_{\mu,\nu}n^{\lambda}_{\mu\nu}f_\mu\otimes f_\nu$ with $n^{\lambda}_{\mu\nu}\in\mb{Z}_{\geq 0}$. So, in particular, $\Delta(M)\subset M\otimes M$ and $M$ is a coalgebra. One of the examples we will consider is the (infinite-dimensional) subspace of cylindric Schur functions whose positive structure constants $n^{\lambda}_{\mu\nu}$ are the Gromov-Witten invariants of the small quantum cohomology of Grassmannians.

\subsection{Quantum cohomology and cylindric Schur functions}
Based on the works of Gepner \cite{gepner1991fusion}, Intriligator \cite{intriligator1991fusion}, Vafa \cite{vafa1991topological} and Witten \cite{witten1993verlinde} on fusion rings, ordinary (intersection) cohomology was extended to (small) quantum cohomology. The latter also possesses an interpretation in enumerative geometry \cite{fulton1996notes}. The Grassmannians were among the first varieties whose quantum cohomology ring $qH^*(\op{Gr}(k,n))$ was explicitly computed \cite{agnihotri1995quantum, bertram1997quantum}. The latter can also be realised as a quotient of $\Lambda\otimes\mb{C}[q]$ \cite{siebert1997quantum,bertram1999quantum} and Postnikov introduced in \cite{postnikov2005affine} a generalisation of skew Schur polynomials, so-called {\em toric Schur polynomials}, which are Schur positive and whose expansion coefficients are the 3-point genus zero Gromov-Witten invariants $C_{\mu\nu}^{\lambda,d}$. The latter are the structure constants of $qH^*(\op{Gr}(k,n))$, where $d\in\mb{Z}_{\geq 0}$ is the degree of the rational curves intersecting three Schubert varieties in general position labelled by the partitions $\lambda,\mu,\nu\subset (n-k)^k$. 

Toric Schur polynomials are finite variable restrictions of cylindric skew Schur functions $s_{\lambda/d/\mu}\in\Lambda$ which have a purely combinatorial definition in terms of sums over cylindric tableaux, i.e. column strict cylindric reverse plane partitions. Cylindric plane partitions were first considered by Gessel and Krattenthaler in \cite{gessel1997cylindric}. There has been subsequent work \cite{mcnamara2006cylindric,lam2006affine,lee2019positivity} on cylindric skew Schur functions exploring their combinatorial and algebraic structure. In particular, Lam showed that they are a special case of affine Stanley symmetric functions \cite{lam2006affine}. While cylindric Schur functions are in general {\em not} Schur-positive, McNamara conjectured in \cite{mcnamara2006cylindric} that their skew versions have non-negative expansions in terms of cylindric non-skew Schur functions $s_{\lambda/d/\emptyset}$. A proof of this conjecture was recently presented in \cite{lee2019positivity}. 

In this article we shall give an alternative proof and, moreover, show that the functions $\{s_{\lambda/d/\emptyset}~|~\lambda\subset(n-k)^k,\;d\in\mb{Z}_{\geq 0}\}$ span a positive subcoalgebra of $\Lambda$ whose structure constants are the Gromov-Witten invariants of $qH^*(\op{Gr}(k,n))$; see Corollaries \ref{cor:mcnamara_conj} and \ref{cor:cylschurcoalg}.

\subsection{The Verlinde algebra and TQFT} 
The quantum cohomology ring $qH^*(\op{Gr}(k,n))$ has long been known to be isomorphic to the $\gh_n$-Verlinde algebra $\V_k(\gh_n)$ at level $k$  when $q=1$ \cite{witten1993verlinde}. Here we will also be interested in the fusion ring $\V_k(\sh_n)$ whose precise relationship with $qH^*(\op{Gr}(k,n))$ was investigated in \cite{korffstroppel2010}: despite being closely related, the two Verlinde algebras $\V_k(\sh_n)$ and $\V_k(\gh_n)$ exhibit different combinatorial descriptions in terms of bosons and fermions.

Verlinde algebras \cite{verlinde1988fusion} or {\em fusion rings} arise in the context of conformal field theory (CFT) and, thus, vertex operator algebras, where they describe the operator product expansion of two primary fields modulo some descendant fields. There is an entire class of rational CFTs, called Wess-Zumino-Witten models, which are constructed from the integrable highest weight representations of Kac-Moody algebras $\mf{\hat g}$ with the level $k$ fixing the value of the central element and the primary fields being in one-to-one correspondence with the highest weight vectors; see e.g. the textbook \cite{francesco2012conformal} and references therein. 

Geometrically the Verlinde algebras $\V_k(\mf{\hat g})$ have attracted interest because their structure constants $\mc{N}_{\lambda\mu}^{\nu}$, called {\em fusion coefficients} in the physics literature, equal the dimensions of moduli spaces of generalised $\theta$-functions, so-called conformal blocks \cite{beauville9conformal, faltings1994proof}. Here the partitions $\lambda,\mu,\nu$ label the primary fields or highest weight vectors. The celebrated Verlinde formula \cite{verlinde1988fusion} %
\begin{equation}\label{Verlinde0}
\mc{N}_{\lambda\mu}^{\nu}=\sum_{\sigma}\frac{\cS_{\lambda\sigma}\cS_{\mu\sigma}\cS_{\sigma\nu}^{-1}}{\cS_{0\sigma}}
\end{equation}
expresses these dimensions in terms of the modular $\cS$-matrix, a generator of the group $PSL(2,\mb{Z})$ describing the modular transformation properties of the characters of the integrable highest weight representations of $\mf{\hat g}$. The representation of $PSL(2,\mb{Z})$ is part of the data of a Verlinde algebra, in particular the $\cS$-matrix encodes the idempotents of the Verlinde algebra as it diagonalises the fusion matrices $(\mc{N}_\lambda)_{\mu\nu}=\mc{N}_{\lambda\mu}^{\nu}$. More recently, the work of Freed, Hopkins and Teleman has also linked the Verlinde algebras to twisted K-theory \cite{freed2011loop,freed2013loop}.

The Verlinde formula and the existence of the modular group representation are a `fingerprint' of a richer structure: a three-dimensional topological quantum field theory (TQFT) or modular tensor category, of which the Verlinde algebra is the Grothendieck ring. In fact, the Verlinde algebra itself can be seen as a TQFT, but a two-dimensional one. Based on work of Atiyah \cite{atiyah1988topological}, the class of 2D TQFT is known to be categorically equivalent to symmetric Frobenius algebras. 

Exploiting the construction from \cite{korffstroppel2010} using quantum integrable systems, a $q$-deformation of the $\sh_n$-Verlinde algebra was constructed in \cite{korff2013cylindric} which (1) carries the structure of a symmetric Frobenius algebra and (2) whose structure constants or fusion coefficients $N_{\lambda\mu}^{\nu}(q)\in\mb{Z}[q]$ are related to cylindric versions of Hall-Littlewood and $q$-Whittaker polynomials. Both types of polynomials occur (in the non-cylindric case) as specialisation of Macdonald polynomials \cite{macdonald1998symmetric}. Setting $q=0$ one recovers the non-deformed Verlinde algebra $\V_k(\sh_n)$ and cylindric Schur polynomials that are different from Postnikov's toric Schur polynomials as their expansion coefficients yield the fusion coefficients $\mc{N}_{\lambda\mu}^{\nu}=N_{\lambda\mu}^{\nu}(0)$ of $\V_k(\sh_n)$ rather than $\V_k(\gh_n)$. Geometrically, the $q$-deformed Verlinde algebra has been conjectured \cite[Section 8]{korff2013cylindric} to be related to the deformation of the Verlinde algebra discussed in   \cite{teleman2004k,teleman2009index}. 

In this article we investigate the combinatorial structure of the $q$-deformed Verlinde algebra  from \cite{korff2013cylindric} in the limit $q\to 1$: we construct two families of positive subcoalgebras of $\Lambda$ whose structure constants are given by $N_{\lambda\mu}^{\nu}(1)$; see Corollary \ref{cor:coalgebra}.
 
\subsection{The main results in this article}
We summarise the main steps in our construction of positive subcoalgebras in $\Lambda$. Recall that the $GL(k)$-characters of tensor products of symmetric $\rS^\mu V=\rS^{\mu_1}V\otimes\cdots\otimes \rS^{\mu_r}V$ and alternating powers $\bigwedge\nolimits^{\!\mu} V=\bigwedge\nolimits^{\!\mu_1} V\otimes\cdots\otimes \bigwedge\nolimits^{\!\mu_r} V$ with $V$ the natural or vector representation of $GL(k)$ are respectively the homogeneous $h_\mu$ and elementary symmetric polynomials $e_\mu$; see Appendix A.2 for their definitions. Similar as in the case of Schur functions we introduce {\em skew complete symmetric} and {\em skew elementary symmetric functions} in $\Lambda$ via the coproduct of their associated symmetric functions,
\begin{equation}\label{skewhe}
\Delta h_\lambda = \sum_{\mu\in\cP^+}h_{\lambda/\mu}\otimes h_\mu \ \qquad\text{ and }\qquad
\Delta e_\lambda = \sum_{\mu\in\cP^+}e_{\lambda/\mu}\otimes e_\mu  \;.
\end{equation}
The latter exhibit interesting combinatorics associated with weighted sums over reverse plane partitions (RPP); see our discussion in Appendix A.4. In light of the generalisation of skew Schur functions \eqref{skewSchur} to cylindric Schur functions in connection with quantum cohomology, one might ask if there exist analogous cylindric generalisations of the functions \eqref{skewhe} and if these define a positive infinite-dimensional subcoalgebra of $\Lambda$.

\subsubsection{Satake correspondence and quantum cohomology}
In order to motivate our approach we first discuss the case of quantum cohomology. It has been long known that the ring $qH^*(\op{Gr}(k,n))$  can be described in terms of the (much simpler) quantum cohomology ring $qH^*(\mb{P}^{n-1})$ of projective space $\mb{P}^{n-1}=\op{Gr}(1,n)$; see e.g. \cite{hori2000mirror,bertram2005two,kim2008quantum}. Here we shall follow the point-of-view put forward in \cite{golyshev2011quantum} concentrating on the simplest case of the Grassmannian only: it follows from the Satake isomorphism of Ginzburg \cite{ginzburg1995perverse} that when identifying the cohomology of $\op{Gr}(k,n)$ as a minuscule Schubert cell in the affine Grassmannian of $GL(n)$ the latter corresponds to the $k$-fold exterior power of the cohomology of $\mb{P}^{n-1}$ under the same identification. In particular, the Satake correspondence identifies $H^*(\op{Gr}(k,n))$ with the $\g_n$-module $\bigwedge^kV$ and the multiplication by the first Chern class corresponds to the action by the principal nilpotent element of $\g_n$. This picture has been extended to quantum cohomology in \cite{golyshev2011quantum} by replacing the principal nilpotent element with the cyclic element in $\g_n(\mb{C})$, which then describes the quantum Pieri rule in $qH^*(\op{Gr}(k,n))$ provided one sets $q=1$, i.e. one considers the Verlinde algebra $\V_k(\gh_n)$.

In our article we shall work instead with the loop algebra $\g_n[z,z^{-1}]$ and identify the quantum parameter with the loop variable via $q=(-1)^{k-1}z^{-1}$. This will allow us to identify the two distinguished bases of $qH^*(\op{Gr}(k,n))$ from a purely Lie-algebraic point of view. Fix a Cartan subalgebra $\h\subset\g_n$. Then under the Satake correspondence the associated (one-dimensional) weight spaces of $\h$ in $\V^-_k=\bigwedge^kV\otimes\mb{C}[z,z^{-1}]$ are mapped onto {\em Schubert classes}. The other, algebraically distinguished, basis of $qH^*(\op{Gr}(k,n))$ is given by the set of its {\em idempotents}. The latter only exist if we introduce the $n$th roots $t=z^{1/n}$ and under the Satake correspondence they are mapped to the weight spaces of the Cartan algebra $\h'$ in apposition to $\h$ \cite{kostant1959principal}, i.e. $\h'$ is the centraliser of the cyclic element. The basis transformation between idempotents and Schubert classes is described in terms of a $t$-deformed modular $\cS$-matrix which encodes an isomorphism $\g_n[z,z^{-1}]\cong\g_n^{\Omega}[t,t^{-1}]$ with the twisted loop algebra $\g_n^{\Omega}[t,t^{-1}]=\bigoplus_{m\in\mb{Z}}t^m\otimes\g_n^{(m)}$, where $\g_n^{(m)}\subset\g_n$ is the subalgebra of principal degree $m\!\!\mod n$. If $t=1$ and $k$ odd we recover the modular $\cS$-matrix of the Verlinde algebra $\V_k(\gh_n)$; see e.g. \cite{naculich2007level}. 

However, we believe the extension to the loop algebra important, not only for the reasons already outlined, but because it allows us to identify the multiplication operators with Schubert classes and, thus, $qH^*(\op{Gr}(k,n))$ itself, with the image of the {\em principal Heisenberg subalgebra} $\mf{\hat h}'_0\subset\gh_n$ \cite{kac1994infinite} in the endomorphisms over $\V^-_k$. Because the latter is a level-0 module we are dealing in this article only with the projection $\h'_0[z,z^{-1}]$ of $\mf{\hat h}'_0$ in the loop algebra $\g_n[z,z^{-1}]$. From a representation theoretic point of  of view it is then natural to consider also the image of the principal Heisenberg subalgebra in the endomorphisms of other subspaces, especially the symmetric tensor product $\V^+_k=\rS^kV\otimes\mb{C}[z,z^{-1}]$.

\subsubsection{Schur-Weyl duality and skew group rings}
A description of our construction in terms of Schur-Weyl duality is as follows: let $\R=\mb{C}[z,z^{-1}]$ be the ring of Laurent polynomials and consider the $\R$-module $\bV_k=\R\otimes V^{\otimes k}$. The latter carries a natural left $U(\g_n[z,z^{-1}])$-action and a right action of the following skew group ring: set $\cH_k=\R[x^{\pm 1}_1,\ldots,x^{\pm 1}_k]\otimes_{\R}\R[\rS_k]$ with $1\otimes_{\R}\R[\rS_k]\cong\R[\rS_k]$ being the group ring of the symmetric group in $k$-letters and impose the additional relations $\sigma_ix_i=x_{i+1}\sigma_i$,  $\sigma_ix_j=x_j\sigma_i$ for $|i-j|>1$ where $\{\sigma_i\}_{i=1}^{k-1}$ are the elementary transpositions in $\rS_k$. The action of the ring $\cH_k$ on $\bV_k$ is fixed by permuting factors and letting each $x_i$ act by the cyclic element of $\g_n[z,z^{-1}]$ in the $i$th factor and trivially everywhere else. This action is not faithful, but factors through the quotient $\cH_k(n)=\cH_k/\langle x_1^n-z\rangle$. Considering $\cH'_k(n)=\cH_k(n)\otimes_\R\RR$ with $\RR=\R[t]/(z-t^n)$  we obtain a (semi-simple) algebra which can be seen as some sort of specialisation or ``classial limit'' of a Ariki-Koike (or cyclotomic Hecke) algebra \cite{ariki1994hecke}.

In this construction the image of the centre $\mc{Z}(\cH_k)\cong\R[x^{\pm 1}_1,\ldots,x^{\pm 1}_k]^{\rS_k}$ in $\End_{\R}\bV_k$ is isomorphic to the ring $\R[x^{\pm 1}_1,\ldots,x^{\pm 1}_k]^{\rS_k}/\langle x_i^n-z\rangle$, which as a Frobenius algebra can be identified with the extension of $qH^*(\mb{P})$ over the ring of Laurent polynomials in $q$, where $\mb{P}=\mb{P}^{n-1}\times\cdots\times\mb{P}^{n-1}$ are $k$ copies of projective space and $q=z^{-1}$. Schur-Weyl duality then tells us that each class in that latter ring must correspond to an element in the image of $U=U(\g_n[z,z^{-1}])$. In fact, we show equality between the images of $\mc{Z}(\cH_k)$ and the principal Heisenberg subalgebra in $\End_{\R}\bV_k$. Restricting to the subspace $\V^-_k\subset \bV_k$ of alternating tensors, we recover the quantum cohomology of Grassmannians $qH^*(\op{Gr}(k,n))$ via the Satake correspondence: the Schur polynomials $s_\lambda(x_1,\ldots,x_k)\in \mc{Z}(\cH_k)$ are mapped to operators in $\End_\R \V^-_k$ which correspond to multiplication by Schubert classes in $qH^*(\op{Gr}(k,n))$. 

Theorem \ref{thm:CylSchurExpansion} then shows how cylindric Schur functions occur in our construction: let $u_i$, $i\in\mb{N}$ be a set of commuting indeterminates and consider in $\End_{\R}\bV_k$ the Cauchy identity
\[
\prod_{i\geq 0}\prod_{j=1}^k(1+u_ix_j)=\sum_{\lambda\in\cP^+}s_{\lambda}(x_1,\ldots,x_k)s_{\lambda'}(u_1,u_2,\ldots)\;.
\]
Taking matrix elements in the above identity with alternating tensors from $\V^-_k$ we obtain formal power series in the quantum deformation parameter $q=(-1)^{k+1}z^{-1}$ whose coefficients are cylindric Schur functions in the $u_i$. The proof of McNamara's conjecture is then an easy corollary; see Corollary \ref{cor:mcnamara_conj}.

\begin{figure}
\centering
\includegraphics[width=.99\textwidth]{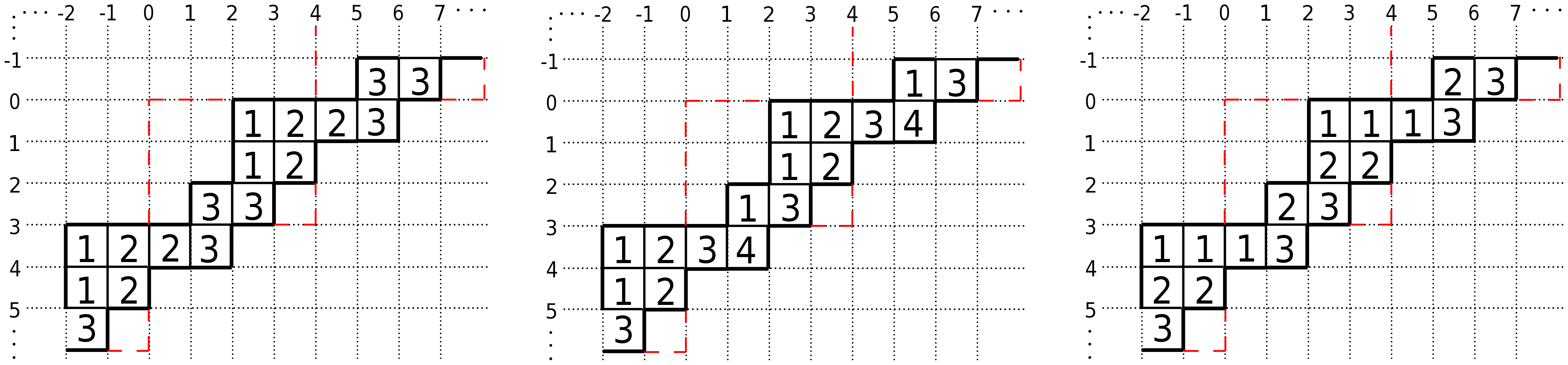} 
\caption{From left to right: a cylindric reverse plane partition (CRPP), a row strict  CRPP and an adjacent column CRPP when $n=4$ and $k=3$; see Section 5 for their definitions.}
\label{fig:CRPPexamples}
\end{figure}

\subsubsection{The subspace of symmetric tensors}
In complete analogy with the previous case we consider the image of $\mc{Z}(\cH_k)$ in the endomorphisms over the subspace of symmetric tensors $\V^+_k=\rS^kV\otimes\mb{C}[z,z^{-1}]\subset \bV_k$. The latter defines again a symmetric Frobenius algebra, which is the $q\to 1$ limit of the $q$-deformed Verlinde algebra discussed in \cite{korff2013cylindric}. Since the corresponding module $\V^+_k$ is no longer minuscule one does not expect that this Frobenius algebra describes the quantum cohomology of a {\em smooth} projective variety. In the context of Frobenius manifolds the symmetric tensor product has also been considered in \cite[Section 2.b]{kim2008quantum}. 

Albeit a direct geometric interpretation is currently missing, there is interesting combinatorics to discover: we consider the following alternative Cauchy identity in $\End_{\R}\bV_k$,
\[
\prod_{i\geq 0}\prod_{j=1}^k(1-u_ix_j)^{-1}=\sum_{\lambda\in\cP^+}m_{\lambda}(x_1,\ldots,x_k)h_{\lambda}(u_1,u_2,\ldots),
\]
where the $m_\lambda$ are the monomial symmetric functions, and now take matrix elements with symmetric tensors in $\V^+_k$. In Theorem \ref{thm:ExpansionCylindricH} we show that as coefficients in the resulting power series in the loop variable $z$, one obtains cylindric analogues $h_{\lambda/d/\mu}$ of the skew complete symmetric functions in \eqref{skewhe}. Similar to cylindric Schur functions, the latter have a completely independent combinatorial definition as weighted sums over cylindric reverse plane partitions (see Figure \ref{fig:CRPPexamples} for examples). Their non-skew versions have a particularly simple expansion in $\Lambda$ (c.f. Lemma \ref{lem:nscylh}),
\begin{equation}
h_{\lambda/d/\emptyset}=\sum_{\mu}\frac{|\rS_\lambda|}{|\rS_\mu|}\,h_\mu\;,
\end{equation}
where $\lambda$ is an element in the fundamental alcove of the $\g_k$ weight lattice under the level-$n$ action of the extended affine symmetric group $\hat \rS_k$, $d\in\mb{Z}$ a sort of winding number around the cylinder and the sum runs over all weights $\mu$ in the $\hat \rS_k$-orbit of $\lambda$. The expansion coefficients are given by the ratio of the cardinalities of the stabiliser subgroups $\rS_\lambda,\rS_\mu\subset \rS_k$ of the weights $\lambda,\mu$ and, despite appearances, are integers; see Lemma \ref{lem:Nreduction}. 

We show that the subspace spanned by these non-skew cylindric complete symmetric functions is a positive subcoalgebra of $\Lambda$, whose non-negative integer structure constants $N_{\mu\nu}^{\lambda}$ coincide with those of the Frobenius algebra $\V^+_k$ and which we express in terms of tensor multiplicities of the generalised symmetric group $G(n,1,k)$. 

Cylindric elementary symmetric functions $e_{\lambda/d/\mu}$ enter naturally by considering the image of the $h_{\lambda/d/\mu}$ under the antipode which is part of the Hopf algebra structure on $\Lambda$. Their combinatorial definition involves row strict cylindric reverse plane partitions; see Figure \ref{fig:CRPPexamples}. We unify all three families of cylindric symmetric functions (elementary, complete and Schur) by relating their combinatorial definitions in terms of cylindric reverse plane partitions to the same combinatorial realisation of the affine symmetric group $\tilde \rS_k$ in terms of `infinite permutations', bijections $w:\mb{Z}\to\mb{Z}$, considered by several authors \cite{lusztig1983some,bjorner1996affine,eriksson1998affine,shi2006kazhdan}. The new aspect in our work is that we link this combinatorial realisation of the {\em extended} affine symmetric group $\hat \rS_k$ to cylindric loops and reverse plane partitions by considering the level-$n$ action for $h_{\lambda/d/\mu}$ and $e_{\lambda/d/\mu}$, while for the cylindric Schur functions we consider the {\em shifted} level-$n$ action.

\section{The principal Heisenberg subalgebra}
Our main reference for the following discussion of the principal Heisenberg subalgebra is \cite[Chapter 14]{kac1994infinite}. While the latter can be introduced for arbitrary simple complex Lie algebras $\mathfrak{g}$, we shall here focus on the simplest case $\mathfrak{g}=\mathfrak{gl}_{n}(\mathbb{C})$, the general linear algebra of the Lie group $GL(V)$ with %
$V=\mathbb{C}v_1\oplus\cdots\oplus\mathbb{C}v_n\cong \mathbb{C}^{n}$. %
A basis of $\mathfrak{gl}_{n}(\mathbb{C})$ are the unit matrices $\{e_{ij}|1\leq i,j\leq n\}$ whose matrix elements are zero except in the $i$th row and $j$th column where the element is 1. The Lie bracket of these basis elements is found to be 
\begin{equation}
\lbrack e_{ij},e_{kl}]=\delta _{jk}e_{il}-\delta _{il}e_{kj}\,.
\label{eijbracket}
\end{equation}%
Note that by choosing a basis we have also fixed a Cartan subalgebra,  
\begin{equation}\label{Cartan}
\mathfrak{h}=\bigoplus_{i=1}^n\mathbb{C} e_{ii}\,.
\end{equation} 
For $i=1,\ldots,n-1$ set $e_{i}=e_{i,i+1}$, $f_{i}=e_{i+1,i}$ and $%
h_{i}=e_{ii}-e_{i+1,i+1}$. These matrices are the Chevalley generators of the subalgebra $%
\mathfrak{sl}_{n}\subset \mathfrak{gl}_{n}$ and we denote by $\mathfrak{h}_0%
\subset\s_n$ the Cartan subalgebra, which is spanned by the $h_{i}$.

\subsection{The affine Lie algebra}
We now turn our attention to the affine Lie algebra of $\g_n$. Let $\mathfrak{gl}_{n}[z,z^{-1}]=\mathbb{C}[z,z^{-1}]\otimes \mathfrak{gl}_{n}$ be the {\em loop algebra} with Lie bracket, $[f(z)x,g(z)y]:=f(z)g(z)[x,y]$, where $f,g\in\mathbb{C}[z,z^{-1}]$ are Laurent polynomials in some variable $z$ and 
$x,y\in\mathfrak{gl}_n$. Then the {\em affine Lie algebra} is the unique central extension of the loop algebra 
\begin{equation}
\mathfrak{\widehat{gl}}_n=\mathfrak{gl}_n[z,z^{-1}]\oplus\mathbb{C}\underline{k}
\end{equation}
together with the Lie bracket
\begin{equation}
[x(z)\oplus \alpha\underline{k},y(z)\oplus \beta\underline{k}]=[x(z),y(z)]\oplus \operatorname{Res}\langle x'(z)|y(z)\rangle\underline{k}\,,
\end{equation}
where $x(z)=f(z)x,y(z)=g(x)y\in\mathfrak{gl}_n[z,z^{-1}]$, $\alpha,\beta\in\mathbb{C}$ and 
\[
\operatorname{Res}\langle x'(z)|y(z)\rangle =\operatorname{Res}(f'(z)g(z)) \langle x|y\rangle\,,
\]
is the 2-cocycle fixed by the Killing form $\langle x|y\rangle=\operatorname{tr}(xy)$ and the
linear map $\operatorname{Res}:\mathbb{C}[z,z^{-1}]\rightarrow\mathbb{C}$ with 
$\operatorname{Res}(z^r)=\delta_{r,-1}$. 
 The set of Chevalley generators for the affine algebra 
$\mathfrak{\widehat{sl}}_n$ contains the additional elements
\begin{equation}\label{aff_efh}
e_n=z\otimes e_{n1},\qquad f_n=z^{-1}\otimes e_{1n},\qquad h_n=(e_{nn}-e_{11})\oplus\underline{k}\;.
\end{equation}

For our discussion we will only need the loop algebra $\mathfrak{sl}_n[z,z^{-1}]\subset\mathfrak{gl}_n[z,z^{-1}]$ but it is instructive to briefly recall the definition of the principal Heisenberg subalgebra in the affine Lie algebra $\sh_n\subset\mathfrak{\widehat{gl}}_n$ first.

For $r=1,\ldots ,n-1$ define the following elements in $\mathfrak{sl}_{n}[z,z^{-1}]$,
\begin{equation}
P _{r}=\sum_{j-i=r}e_{ij}+z\sum_{i-j=n-r}e_{ij}\; \label{P_r}
\end{equation}%
and set $ P_{r+n}=z P_{r}$ for all other $r\in \mathbb{Z}\backslash n\mathbb{Z}$. 
\begin{lemma} 
In the affine algebra $\mathfrak{\widehat{sl}}_n$ we have the Lie bracket relations
\begin{equation}
[P_m,P_{m^\prime}]=\delta_{m+m^\prime,0}\,\underline{k}\,,\qquad \forall m,m^\prime\in\mathbb{Z}\backslash
n\mathbb{Z}\,.
\end{equation}
\end{lemma}
The $P_r$ together with the central element $\underline{k}$ span the so-called {\em principal Heisenberg subalgebra} $\mathfrak{\hat h}'_0\subset \mathfrak{\widehat{sl}}_n$. 
\footnote{Note that this principal Heisenberg subalgebra is different from the {\em homogeneous Heisenberg subalgebra} 
$\mathfrak{\hat h}_0=\mathfrak{h}_0[z,z^{-1}]\oplus\mathbb{C}\underline{k}$ 
which is generated by $h_i(m)=z^mh_i$ with $[h_i(m),h_j(m')]=\delta_{m+m',0}\langle h_i|h_j\rangle\underline{k}$ and 
$m,m'\in\mathbb{Z}$.}
Consider the exact sequence
\begin{equation}\label{pi}
0\rightarrow\mathbb{C}\underline{k}\rightarrow \mathfrak{\widehat{gl}}_n \overset{\pi}{\twoheadrightarrow} \mathfrak{gl}_n[z,z^{-1}]\rightarrow 0\,.
\end{equation}
Then $\mathfrak{\hat h}'_0=\pi^{-1}(\mathfrak{z}(P_1))$ is the pre-image of the centraliser $\mathfrak{z}(P_1)=\{x(z)\in\mathfrak{sl}_n[z,z^{-1}]~|~[P_1,x(z)]=0\}$ under the projection $\pi$. We denote by 
\begin{equation}\label{preHeisenberg}
\mathfrak{h}_0'[z,z^{-1}]=\pi(\mathfrak{\hat h}'_0)\subset \mathfrak{sl}_{n}[z,z^{-1}]\,.
\end{equation}
the projection of the principal Heisenberg subalgebra onto the loop algebra in \eqref{pi}. It follows that $[P_{r}, P_{r^{\prime }}]=0$ in the loop algebra $\mathfrak{sl}[z,z^{-1}]$ and, thus, $\mathfrak{h}'_0[z,z^{-1}]$ is an (infinite-dimensional) commutative Lie subalgebra in $\mathfrak{sl}_n[z,z^{-1}]$. It is the latter which is in the focus of our discussion for the remainder of this article.

Note that both `halves' of the principal subalgebra, the negative and positive indexed elements $P_r$ are related by the following anti-linear $*$-involution on $\mathfrak{gl}_{n}[z,z^{-1}]$, 
\begin{equation}\label{star}
z^*=\bar z=z^{-1}\quad\text{ and }\quad x^{\ast }=\bar{x}^{T},\quad\forall x\in\mathfrak{gl}_n,
\end{equation}
where $\bar x$ is the matrix obtained by complex conjugation and $x^T$ denotes the transpose. That is, $x^*$ is the hermitian conjugate of $x$. One immediately verifies from the definition \eqref{P_r} that $P_r^*=P_{-r}$. 

Denote by $\Gamma:\g_n\to\g_n$ the (finite) Dynkin diagram automorphism induced by exchanging the $i$th Dynkin node with the $(n-i)$th node for $i=1,\ldots,n-1$, $\Gamma(e_{i,j})=e_{n-i,n-j}$. The latter is clearly an involution. Consider now the affine Dynkin diagram and the map $\hat\Gamma:\g_n\to\g_n$ of order $n$ that corresponds to a cyclic permutation of all nodes, $\hat\Gamma(e_{ij})=e_{i+1,j+1}$ where indices are understood modulo $n$. Note that the latter also induces a (finite) Lie algebra automorphism and that the relation 
$\Gamma\circ\hat\Gamma^{-1}=\hat\Gamma\circ\Gamma$ holds. We extend both automorphisms (viewed as $\mb{C}[z,z^{-1}]$-linear maps)  to the loop algebra $\g_n[z,z^{-1}]$ by setting $\Gamma(f(z)x)=f(z)\Gamma(x)$ and $\hat\Gamma(f(z)x)=f(z)\hat\Gamma(x)$ for all $x\in\g_n$ and $f\in\mb{C}[z,z^{-1}]$.
\begin{lemma}\label{lem:Pi}
Let $\Pi=\hat\Gamma\circ\Gamma$, then $\Pi(P_r)=\overline{P}_{-r}$ for all $r\in \mathbb{Z}\backslash
n\mathbb{Z}$.
\end{lemma}
\begin{proof}
A straightforward computation using the definition \eqref{P_r}.
\end{proof}

\subsection{Cartan subalgebras in apposition}
We now consider the projection of the principal Heisenberg subalgebra onto the finite Lie algebra $\mathfrak{gl}_n$. Let $\pi_1:\mathfrak{gl}_n[z,z^{-1}]\rightarrow\mathfrak{gl}_n$ be the projection obtained by specialising to $z=1$, that is $\pi_1:f(z)\otimes g\mapsto f(1) g$. Set 
$h'_r=\pi_1(P_r)$ for $r=1,\ldots,n-1$. Then $h'_1$ is known as {\em cyclic element}
and the centraliser
\begin{equation}
\mathfrak{h}'=\{g\in \mathfrak{gl}_{n}~|~[h'_1 ,g]=0\}
\label{Cartan'}
\end{equation}%
is another Cartan subalgebra, called {\em in apposition} to the Cartan algebra $\mathfrak{h}$ defined in \eqref{Cartan}; see \cite{kostant1959principal} for details.

We recall the construction of the root vectors with respect to the Cartan algebra $\mathfrak{h}'$. Let $\zeta=\exp(2\pi\iota/n)$ and define for$\ i,j=1,\ldots ,n$ the matrices%
\begin{equation}
e'_{ij}=\mathcal{S}^{-1}e_{ij}\mathcal{S}=\frac{1}{n}%
\sum_{a,b=1}^{n}\zeta ^{-ia+jb}e_{ab},\quad \mathcal{S}=(\zeta ^{ab}/\sqrt{n}%
)_{1\leq a,b\leq n}\;. \label{modS}
\end{equation}%
Obviously, $\mathcal{S}$ induces a Lie algebra automorphism, i.e. the
matrices $e'_{ij}$ also satisfy (\ref{eijbracket}). 
\begin{lemma}
The matrices $\{e'_{ij}\}$ give the root space decomposition with respect to the Cartan algebra $\mathfrak{h}'$. That is, we have
\begin{equation}
[h'_r,e'_{ij}]=(\zeta^{-ir}-\zeta^{-jr})e'_{ij},\quad r=1,\ldots,n-1,\quad i,j=1,\ldots,n
\end{equation}
Moreover,
\begin{equation}\label{diagonalh}
\mathcal{S}h'_r \mathcal{S}^{-1}=\sum_{j=1}^{n}\zeta ^{-jr}e_{jj}\;.
\end{equation}
and the map $e_{ij}\mapsto e'_{ij}$ is a Lie algebra automorphism.
\end{lemma}
\begin{proof} A straightforward computation using the Lie bracket relations
\begin{equation}\label{rotation}
\lbrack h'_r ,e_{i,j}]=e_{i-r,j}-e_{i,j+r}\,,
\end{equation}
where indices are understood modulo $n$.
\end{proof}
Note that the matrices $\{h'_1,h'_2,\ldots,h'_{n-1}\}$ only span the $\mathfrak{sl}_n$ Cartan subalgebra $\mathfrak{h}'_0=(\mathfrak{h}'\cap\mathfrak{sl}_n)\subset\mathfrak{h}'$. A basis of $\mathfrak{h}'$ is given by 
the matrices $e'_{ii}=\mathcal{S}^{-1} e_{ii}\mathcal{S}$.

We now take a closer look at the matrix 
\begin{equation}\label{modS1}
\mathcal{S}=\frac{1}{\sqrt{n}}\sum_{i,j=1}^n \zeta^{ij}e_{ij}
\end{equation}
which diagonalises the Cartan subalgebra $\mathfrak{h}'$. If we introduce in addition the matrix 
\begin{equation}\label{modT1}
\mathcal{T}=\zeta^{-\frac{n(n-1)}{24}}\sum_{i=1}^n \zeta^{\frac{i(n-i)}{2}}e_{ii}
\end{equation}
then we obtain a representation of the group $PSL_2(\mathbb{Z})$. Let $\mathcal{C}=(\delta_{i+j,0\mod n})_{1\leq i,j\leq n}$ and $\mc{\hat C}=(\delta_{i,j+1\mod n})_{1\leq i,j\leq n}$ be the matrices which implement the Dynkin diagram automorphisms $\Gamma$ and $\hat\Gamma$, %
\begin{equation}\label{C1}
\mathcal{C}e_{ij}\mathcal{C}=\Gamma(e_{ij})=e_{n-i,n-j}\quad\text{and}\quad
\mathcal{\hat C}e_{ij}\mathcal{\hat C}^{-1}=\hat\Gamma(e_{ij})=e_{i+1,j+1}\,,
\end{equation}
where indices are understood modulo $n$. Then $\mc{C}h'_r=h'_{n-r}\mc{C}$ and $\mc{\hat C}h'_r=h'_{r}\mc{\hat C}$ for $r=1,\ldots,n-1$. Moreover, we have the following proposition.
\begin{prop}[Kac-Peterson \cite{kavc1984infinite}]\label{STC}
The $\mathcal{S}$ and $\mathcal{T}$-matrix satisfy 
\begin{equation}\label{mod_relations}
\mathcal{S}^2=(\mathcal{S}\mathcal{T})^3=\mathcal{C}
\end{equation}
and
\begin{equation}\label{mod_relations2}
\mathcal{CS}=\mathcal{\bar S}=\mathcal{S}^{\ast}=\mathcal{S}^{-1},\qquad
\mathcal{CT}=\mathcal{TC},\qquad \mathcal{T}^{\ast}=\mathcal{T}^{-1}\,.
\end{equation}
Moreover, $(\mc{\hat C}\mc{S})_{ab}=\mc{S}_{(a-1)b}=\zeta^{-b} \mc{S}_{ab}$ and $(\mc{S}\mc{\hat C})_{ab}=\mc{S}_{a(b+1)}=\zeta^{a} \mc{S}_{ab}$ with indices taken modulo $n$.
\end{prop}
This is the the familiar representation of the $\mathcal{S}$ and $\mathcal{T}$-matrix from $\widehat{su}(n)_k$-WZNW conformal field theory for $k=1$; see \cite{kavc1984infinite} for the case of general $k$. We therefore omit the proof.

\subsection{The twisted loop algebra}
We recall from \cite{kac1994infinite} the twisted realisation of the loop algebra with respect to the {\em principal gradation}. Denote by $\rho^\vee\in\mathfrak{h}$ the dual Weyl vector, which is the sum over the fundamental co-weights with respect to the original Cartan subalgebra $\mathfrak{h}$. In the case of $\mathfrak{gl}_n$ it reads explicitly, 
\begin{equation}\label{dual_Weyl}
\rho^\vee=\sum_{i=1}^n\frac{n-2i+1}{2}\, e_{ii} \;.
\end{equation}
One easily verifies that 
$[\rho^\vee,e_{ij}]=(j-i)e_{ij}$ and, hence, the group element $\omega=\exp[2\pi \imath\rho^\vee/n]$ %
induces a $\mathbb{Z}_n$-gradation on the Lie algebra 
$\mathfrak{gl}_n$ via the automorphism $\Omega(g)=\omega g\omega^{-1}$,
\[
\mathfrak{gl}_n
=\bigoplus_{r=0}^{n-1}\mathfrak{gl}_n^{(r)},\qquad \mathfrak{gl}_n^{(r)}=\{g~|~\Omega (g)=
e^{\frac{2\pi \imath}{n} r} g\}\;.
\]
In particular, we have that $\mathfrak{h}=\mathfrak{gl}^{(0)}_n$ and $\Omega(h'_r)=\exp(2\pi\imath r/n)h'_r$. Thus, $\Omega(\mathfrak{h}\rq{})=\mathfrak{h}\rq{}$, and Kostant has shown in \cite{kostant1959principal} that the automorphism $\Omega$ induces a Coxeter transformation in the Weyl group $S_n\cong N(T')/T'$, where $T'\subset GL(V)$ is the maximal torus corresponding to the Cartan algebra in apposition $\mathfrak{h}'$ and $N(T')$ its normaliser.

For $r\in\mathbb{Z}$ define $\bar r\in\{0,1,2,\ldots,n-1\}$ by $r=\bar r+m n$ with $m\in\mathbb{Z}$. Then the twisted loop algebra is defined as 
\begin{equation}\label{twistedg}
\mathfrak{gl}^\Omega_n[t,t^{-1}]=\bigoplus_{r\in\mathbb{Z}}t^r\otimes\mathfrak{gl}_n^{(\bar r)}\,,
\end{equation}
with Lie bracket $[f(t)x,g(t)y]:=f(t)g(t)[x,y]$ where $f,g\in\mathbb{C}[t,t^{-1}]$ and $x,y\in\mathfrak{gl}_n$. The following result is an immediate consequence of the analogous isomorphism for the corresponding affine Lie algebras in  \cite[Chapter 14]{kavc1984infinite} and we therefore omit the proof.
\begin{prop}\label{prop:twisted} 
The linear map $\phi:\mathfrak{gl}_n[z,z^{-1}]\rightarrow \mathfrak{gl}^\Omega_n[t,t^{-1}]$ defined by 
\[
z^m\otimes e_{ij}\mapsto\phi(z^m\otimes e_{ij})=t^{j-i+nm}\otimes e_{ij}
\]
is a Lie algebra isomorphism. In particular, we have for $r\in\mb{N}\backslash n\mb{N}$ that
\begin{equation}
\phi(P_r)=t^{r} h'_{\bar r} \qquad\text{and}\qquad\phi( P_{-r})=t^{-r} h'_{n-\bar r}\,,
\end{equation}
where $h'_m\in\mathfrak{h}'$ with $m=1,2,\ldots,n-1$ is the image of $P_m$ under the projection $\pi_1:\mathfrak{gl}_n[z,z^{-1}]\twoheadrightarrow
\mathfrak{gl}_n$ obtained by setting $z=1$. 
\end{prop}
Setting formally $z=t^n$ the loop algebra isomorphism $\phi:\mathfrak{gl}_n[z,z^{-1}]\rightarrow \mathfrak{gl}^\Omega_n[t,t^{-1}]$ corresponds to the following similarity transformation with the diagonal matrix
\begin{equation}
\mathcal{D}(t)=\sum_{i=1}^n t^{-i}e_{ii}\;.
\end{equation}
That is, we have the straightforward identities $\mathcal{D}(t)e_{ij}\mathcal{D}(t^{-1})=t^{j-i}e_{ij}$ from which the result is immediate. This allows us to combine the $\mathfrak{gl}_n$ Lie algebra automorphism induced by the $\mathcal{S}$-matrix with the loop algebra isomorphism $\phi$ via introducing the following deformed $\mathcal{S}$-matrix,
\begin{equation}\label{modSt}
\mathcal{S}(t)=\mathcal{S}\mathcal{D}(t)=\frac{1}{\sqrt{n}}\sum_{i,j=1}^n t^{-j}\zeta^{ij}e_{ij}
\end{equation}
while leaving the $\mathcal{T}$-matrix unchanged.
\begin{lemma}\label{lem:twistedHeisenberg}
Let $z=t^n$. Then %
$\mathcal{S}(t)P_r\mathcal{S}^{-1}(t)=t^{r}\sum_{j=1}^n \zeta^{-jr}e_{jj}$.
\end{lemma}

\begin{proof}
A somewhat tedious but straightforward computation using the definition \eqref{modSt}. We therefore omit the details.
\end{proof}
The last lemma shows that the similarity transformation with the deformed $\cS$-matrix \eqref{modSt} describes the Lie algebra isomorphism from Proposition \ref{prop:twisted} together with a change of basis in which the principal subalgebra is diagonal. 

\subsection{Universal enveloping algebras and PBW basis} 
Consider the universal enveloping algebra 
$U(\mathfrak{\widehat{gl}}_n)$. According to the Poincare-Birkhoff-Witt (PBW) Theorem the latter has the triangular decomposition
\[
U(\mathfrak{\widehat{gl}}_n)\cong U(\mathfrak{\hat n}_+)\oplus 
U(\mathfrak{\hat h})\oplus U(\mathfrak{\hat n}_-)\,,
\]
with $\mathfrak{\hat n}_+$ and $\mathfrak{\hat n}_-$ being the Lie subalgebras generated by $\{e_i\}_{i=1}^{n}$ and $\{f_i\}_{i=1}^{n}$, respectively. Here the affine Cartan subalgebra is given by $\mathfrak{\hat h}=\mathfrak{h}[z,z^{-1}]\oplus\mathbb{C}\underline{k}$, but in what follows we are only interested in level 0 representations and, hence, focus on the loop algebra $\mathfrak{h}[z,z^{-1}]$ and the universal enveloping algebra
$U(\mathfrak{gl}_n[z,z^{-1}])\cong U(\mathfrak{\hat n}_+)\oplus 
U(\mathfrak{h}[z,z^{-1}])\oplus U(\mathfrak{\hat n}_-)$ instead. To unburden the notation we will henceforth write $U$ for $U(\mathfrak{gl}_n[z,z^{-1}])$ and $U_{\pm}$ for the Borel subalgebras $U(\mathfrak{h}[z,z^{-1}])\oplus U(\mathfrak{\hat n}_\pm)$. Similarly, we abbreviate the universal enveloping algebra of the twisted loop algebra, $U(\mathfrak{gl}^\Omega_n[t,t^{-1}])$, by $U^\Omega$.

\begin{lemma}\label{lem:phiHopf}
The Lie algebra isomorphism $\phi:\mathfrak{gl}_n[z,z^{-1}]\rightarrow \mathfrak{gl}^\Omega_n[t,t^{-1}]$ extends to a Hopf algebra isomorphism $\phi:U\rightarrow U^\Omega$.
\end{lemma}
\begin{proof}
Recall that $\mf{g}=\g_n[z,z^{-1}]$ naturally embeds into its tensor algebra $T(\mf{g})$. Identifying $\mf{g}$ with its image under the projection $T(\mf{g})\twoheadrightarrow U$ it generates $U$. The same is true for $U^{\Omega}$. As the definition of coproduct $\Delta:U\to U\otimes U$, $\Delta(x)=x\otimes 1+1\otimes x$, antipode $\gamma:U\to U$, $\gamma(x)=-x$ and co-unit $\veps(x)=0$ for all $x\neq 1$, are the same for both $U$ and $U^\Omega$ the assertion follows from $\phi$ being a Lie algebra isomorphism.
\end{proof}
Let $\cP^+_{\neq n}$ be the 
set of partitions $\lambda$ where each part $\lambda_i\neq \ell n$, $\ell\in\mathbb{N}$ and set 
$ P_{\pm\lambda}= P_{\pm\lambda_1} P_{\pm\lambda_2}\cdots$ Then $\{ P_{\pm\lambda}~|~\lambda\in\cP^+_{\neq n}\}\subset U(\mathfrak{\hat n}_\pm)$ and according to the PBW theorem $\{P_{\lambda} P_{-\mu}~|~\lambda,\mu\in\cP^+_{\neq n}\}$ forms a basis of $U(\mathfrak{h}'_0[z,z^{-1}])$.

\subsection{Power sums and the ring of symmetric functions}
Recall the definition of the ring of symmetric functions $\Lambda=\mb{C}[p_1,p_2,\ldots]$, which is freely generated by the power sums $p_r$, and can be equipped with the structure of a Hopf algebra; see Appendix A for details and references. 
\begin{lemma} \label{lem:Lambda2U}
The maps $\Phi_{\pm}:\Lambda\to U_{\pm}$ fixed by 
\begin{equation}\label{allPs}
p_{mn}\mapsto P_{\pm mn}=z^{\pm m}\sum_{i=1}^ne_{ii}\quad\text{and}\quad 
p_r\mapsto P_{\pm r},\;r\neq mn,\;m,r\in\mb{N}
\end{equation}
are Hopf algebra homomorphisms. 
\end{lemma}
We shall henceforth set $P_{\pm mn}=z^{\pm m}\sum_{i=1}^ne_{ii}$ and $P_0=\sum_{i=1}^ne_{ii}$, such that the relation $P_{r+n}=zP_r$ holds for all $r\in\mb{Z}$. Denote by $U'$ the subalgebra in $U$ generated by $\{P_r\}_{r\in\mb{Z}}$. 
\begin{proof}
This is immediate upon noting that the Hopf algebra relations for power sums in $\Lambda$ match the relations of the standard Hopf algebra structure on $U$; compare with the proof of Lemma \ref{lem:phiHopf} and Appendix A.
\end{proof}
Using the maps $\Phi_{\pm}$ we now introduce the analogue of various symmetric functions (see Appendix A for their definitions and references) as elements in the upper (lower) Borel algebras and denote these by capital letters. For example, we define for any partition $\lambda\in\cP^+$ in $U_+$ the elements
\begin{equation}\label{defS}
S_\lambda=\sum_\mu \frac{\chi_\lambda(\mu)}{\z_\mu}P_\mu,\qquad \z_\lambda = 
\prod_{i \ge 1} i^{m_i(\lambda)} m_i(\lambda)!\,,
\end{equation} 
where $\chi_\lambda(\mu)$ is the character value of an element of cycle type $\mu$ in the symmetric group. Thus, the element $S_\lambda$ is the image of the Schur function $s_\lambda\in\Lambda$ under $\Phi_+$. If the partition $\lambda$ in the definition of $S_\lambda$ consists of a single vertical or horizontal strip of length $r>0$ we obtain the following elements in $U_+$,
\begin{equation}\label{defE&H}
E_r = \sum_{\lambda \vdash r} \epsilon_\lambda \z_\lambda^{-1} P_\lambda 
\qquad\text{and}\qquad
H_r =\sum_{\lambda \vdash r} \z_\lambda^{-1} P_\lambda 
\end{equation}
where $\epsilon_\lambda = (-1)^{|\lambda|-\ell(\lambda)}$. %
These generators are solutions to Newton\rq{}s formulae and can be identified with the elementary $e_r$ and complete symmetric functions $h_r$ in the ring of symmetric functions. Below we will also make use of the `generating functions' 
\begin{equation}\label{E}
E(u)=\sum_{r\geq 0} u^r E_r=\exp\left(-\sum_{r\geq 1}\frac{(-u)^r}{r}\,P_r\right)
\end{equation}
and
\begin{equation}\label{H}
H(u)=\sum_{r\geq 0} u^r H_r=\exp\left(\sum_{r\geq 1}\frac{u^r}{r}\,P_r\right)\,.
\end{equation}
Their matrix elements should be understood as formal power series in the indeterminate $u$. Besides the image of Schur functions we will also need the image of monomial symmetric functions $m_\lambda$ under $\Phi_+$; see Appendix A. We recall from $\Lambda$ the following `straightening rules' for elements $S_{v}$ which are not indexed by partitions but by some $v=(\dots ,a,b,\dots )\in\mathbb{Z}_{\geq 0}^\ell$,
\begin{equation}
S_{(\dots ,a,b,\dots )}=-S_{(\dots ,b-1,a+1,\dots
)}\qquad \text{and}\qquad S_{(\dots ,a,a+1,\dots )}=0
\label{straightenSchur}
\end{equation}
In $\Lambda$ these relations are a direct consequence of the known relations for determinants. The need for these straightening rules arises when applying Macdonald's raising operator \cite[Chapter I.1]{macdonald1998symmetric}. Recall that the raising operator $R_{ij}$ acts on partitions as $R_{ij}\lambda =(\lambda _{1},\dots
,\lambda _{i}+1,\dots ,\lambda _{j}-1,\dots )$. Given $S_\lambda$ we set $R_{ij}S_\lambda=
S_{R_{ij}\lambda}$ and extend this action linearly. 

We are now in the position to introduce $M_\lambda=\Phi_+(m_\lambda)$. Given a partition $\lambda\in\cP^+$ define
\begin{equation}\label{M}
M_{\lambda }=\prod_{\lambda _{i}>\lambda _{j}}(1-R_{ji})%
S_{\lambda }\,.
\end{equation}%
Compare with \cite[Example III.3]{macdonald1998symmetric} when $t=1$. The definition \eqref{M} is best explained on an example.
\begin{exam}
Consider the partition $\lambda=(2,1)$. Then we have: 
\[
M_{(2,1)}=(1-R_{21})\prod_{j>2}(1-R_{j1})(1-R_{j2})S_{(2,1)}\;.
\]%
From the above product we need to extract all operators $R=\prod R_{ji}$ for which $S_{R\lambda}$ is nonzero. Employing the straightening rules \eqref{straightenSchur}, we find that the only non-zero terms  are given by $M_{(2,1)}=S_{(2,1)}-2S_{(1,1,1)}$.
\end{exam}
Making the substitution $P_r\rightarrow P_{-r}$ we denote by $\{E_{-r},H_{-r}\}_{r\in\mathbb{N}}$ the images of the elementary and symmetric functions under $\Phi_-$. Using the anti-linear involution \eqref{star} and the fact that the power sums are a $\mb{Q}$-basis it follows at once that $\Phi_+^*=\Phi_-$. Similarly, taking the combined Dynkin diagram automorphism $\Pi$ from \ref{lem:Pi} that $\Phi_-=\overline{\Pi\circ \Phi}_+$.
\begin{lemma}\label{lem:adjoint}
Let $f\in\Lambda$ then $F^*=\Phi_+(f)^*=\overline{\Pi(F)}=\Phi_-(f)$, %
where $*$ is the anti-linear involution \eqref{star} and $\Pi$ the combined Dynkin diagram automorphism from Lemma \ref{lem:Pi}.
\end{lemma}
So, in particuar, if $S_\lambda=\Phi_+(s_\lambda)$ and $M_\lambda=\Phi_+(m_\lambda)$ are the images of the Schur function $s_\lambda$ and monomial symmetric function $m_\lambda$ under the homomorphism from Lemma \ref{lem:Lambda2U}, then %
\begin{equation}
S_\lambda^*=\overline{\Pi(S_\lambda)}=\Phi_-(s_\lambda)\quad\text{and}\quad M_\lambda^*=\overline{\Pi(M_\lambda)}=\Phi_-(m_\lambda)\;.
\end{equation}
\begin{proof}
As the power sums freely generate $\Lambda$, it suffices to check the relations $\Phi_+^*=\Phi_-$ and $\Phi_-=\overline{\Pi\circ \Phi}_+$ on the $p_r$. The first relation is a consequence of the definitions \eqref{P_r} and \eqref{star}, the second from Lemma \ref{lem:Pi}.
\end{proof}

\section{The affine symmetric group and wreath products}
Let $\rS_{k}$ be the symmetric group of the set $[k]=\{1,\ldots,k\}$ and denote by $\{\sigma _{1},\ldots ,\sigma _{k-1}\}$ the group's generators. %
Set $\cP_k = \bigoplus_{i=1}^k \mathbb{Z} \epsilon_i$, the $\mathfrak{gl}_k$ weight lattice with standard basis $\epsilon_1,\dots,\epsilon_k$ and inner product $(\epsilon_i,\epsilon_j)=\delta_{ij}$. Define the root lattice $\cQ_k\subset\cP_k$ as the sub-lattice generated by $\{\alpha_i=\epsilon_i-\epsilon_{i+1}\}_{i=1}^{k-1}$ and set $\alpha_0=\epsilon_n-\epsilon_1$. Denote by $\cP_k^+\subset\cP_k$ the {\em positive dominant weights} which we identify with the set of partitions, $\cP_k^+\cong\{\lambda~|~ \lambda_1\geq
\cdots\geq\lambda_k\geq 0\}$ which have at most $k$ parts. We use the notation $\cP_k^{++}=\{\lambda~|~ \lambda_1>
\cdots>\lambda_k\geq 0\}$ for the subset of {\em strict} dominant weights/partitions. 

\subsection{Action on the weight lattice}
Each $\lambda\in\cP_k$ defines a map $\lambda:[k]\to\mathbb{Z}$ in the obvious manner and we shall consider the right action $\cP_k\times \rS_k\to\cP_k$ given by $(\lambda,w)\mapsto \lambda\circ w=(\lambda_{w(1)},\ldots,\lambda_{w(k)})$. For a fixed weight $\lambda$ denote by $\rS_\lambda\subset \rS_k$ its stabiliser group. The latter is isomorphic to the Young subgroup 
\begin{equation}\label{stabS}
\rS_\lambda\cong\cdots\times \rS_{m_1(\lambda)}\times \rS_{m_0(\lambda)}\times \rS_{m_{-1}(\lambda)}\times\cdots 
\end{equation}
with $m_i(\lambda)$ being the multiplicity of the part $i$ in $\lambda$. Note that $|\rS_\lambda|=\prod_{i\in\mb{Z}} m_i(\lambda)!$ and we shall make repeatedly use of the multinomial coefficients
\begin{equation}\label{qdim}
d_\lambda=\frac{|\rS_k|}{|\rS_\lambda|}=\frac{k!}{\prod_{i\in\mb{Z}}m_i(\lambda)!}=\binom{k}{m(\lambda)}\,,
\end{equation}
which we call the {\em quantum dimensions} for reasons that will become clear in the following sections. Given any permutation $w\in \rS_k$ there exists a unique decomposition $w=w_\lambda w^\lambda$ with $w_\lambda\in \rS_\lambda$ and $w^\lambda$ a minimal length representative of the right coset $\rS_\lambda w$. Denote by $\rS^\lambda\subset \rS_k$ the set of all minimal length coset representatives in $\rS_\lambda\backslash \rS_k$.

\subsection{The extended affine symmetric group}
Let $\cP_k$ act on itself by translations. Then the {\em extended affine symmetric group} is defined as $\hat \rS_k=\rS_k\ltimes\cP_k$. In terms of generators and relations $\hat \rS_k$ is defined as the group generated by $\langle\tau,\sigma_0, \sigma_1, \ldots,\sigma_{k-1} \rangle$ subject to the identities (all indices are understood mod $k$)
\begin{eqnarray}\label{sigma}
  \sigma_i^2=1, \quad  \sigma_i\sigma_{i+1}\sigma_i = \sigma_{i+1}\sigma_i\sigma_{i+1}, \quad
  \sigma_i\sigma_j = \sigma_j\sigma_i,\;|i-j|>1\,,
 \end{eqnarray}
 and 
\begin{equation}\label{tau}
\tau\sigma_{i+1}=\sigma_i\tau\;.
\end{equation}
For our discussion it will be convenient to use instead of $\tau$ and $\sigma_0$ the generators
\begin{equation}\label{x}
x_k=\tau\sigma_1\sigma_2\cdots\sigma_{k-1}
\quad\text{ and }\quad x_i=\sigma_i x_{i+1}\sigma_i,\quad i=1,2,\ldots,k-1\,.
\end{equation}
Then any element $\hat w\in\hat \rS_k$ can be written as $\hat w=w x^\lambda=w x_1^{\lambda_1}\cdots x_k^{\lambda_k}$ for some $\lambda\in\cP_k$ and $w\in \rS_k$.

Fix as ground ring $\R=\mb{C}[z,z^{-1}]$ and consider the $\R$-module,
\begin{equation}\label{affHecke}
\cH_k=\R[x^{\pm 1}_1,\ldots,x^{\pm 1}_k]\otimes_{\R}\R[\rS_k ]\,.
\end{equation}
Define an $\R$-algebra by identifying (as subalgebras) $\R[x^{\pm 1}_1,\ldots,x^{\pm 1}_k]\otimes 1$ with the polynomial algebra $\R[x^{\pm 1}_1,\ldots,x^{\pm 1}_k]$ and $1\otimes\R[\rS_k]$ with the group algebra $\R[\rS_k]$, and in addition impose the relations 
\begin{equation}\label{affHecke_rel}
\sigma _{i}x_{i}\sigma _{i}=x_{i+1},\qquad
x_{i}\sigma _{j}=\sigma _{j}x_{i}\;\;\text{for\ }\;j\neq i,i-1\;.
\end{equation}
The latter algebra is a skew group ring or semi-direct product algebra which is some sort of ``classical limit'' of the affine Hecke algebra. The following facts about $\cH_k$ are known:

\begin{lemma}\label{lem:hecke} 
(i) The set $\{wx^{\lambda }~|~x^{\lambda }=x_1^{\lambda_1}\cdots x_k^{\lambda_k},\;\lambda\in\cP_k,\;w\in \rS_{k}\}$ is a basis of $\cH_k$.
(ii) The centre of $\cH_k$ is $\mathcal{Z}(\cH_k)=\R[x^{\pm 1}_1,\ldots,x^{\pm 1}_k]^{\rS_k}$. 
\end{lemma}
\begin{proof}
Claim (i) follows from the definition of $\cH_k$ and because the monomials $\{x^\lambda\}_{\lambda\in\cP_k}$ form a basis of $\R[x^{\pm 1}_1,\ldots,x^{\pm 1}_k]$. In particular, the skew group ring $\cH_k$ is a free module over $\R[x^{\pm 1}_1,\ldots,x^{\pm 1}_k]$. Assume that $f=\sum_{w}f_w w$ with $f_w\in\R[x^{\pm 1}_1,\ldots,x^{\pm 1}_k]$ is central, then it follows from  $x_i f=f x_i$ with $i=1,\ldots,k$ that $f_w=0$ for all $w\neq 1$. Thus, $f\in \R[x^{\pm 1}_1,\ldots,x^{\pm 1}_k]$, but since in addition $w f=f w$ for all $w\in \rS_k$ we must have $f\in \R[x^{\pm 1}_1,\ldots,x^{\pm 1}_k]^{\rS_k}$. Hence, $\mathcal{Z}(\cH_k)\subset\R[x^{\pm 1}_1,\ldots,x^{\pm 1}_k]^{\rS_k}$. The converse relation, $\R[x^{\pm 1}_1,\ldots,x^{\pm 1}_k]^{\rS_k}\subset\mathcal{Z}(\cH_k)$ is obvious. This proves (ii).
\end{proof}

\subsection{Representations in terms of the cyclic element} 
Let $V$ be the vector representation of $\g_n$ introduced earlier and recall from \eqref{P_r} the definition of the elements $P_{\pm 1}$ in the loop algebra $\s_n[z,z^{-1}]$ which we interpret as endomorphisms over $V[z,z^{-1}]=V\otimes\mb{C}[z,z^{-1}]$ corresponding to the matrices
\begin{equation}\label{X}
X=\left(
\begin{array}{cccc}
0 & 1 & & 0 \\
\vdots & \ddots& \ddots & \\
0 & & 0& 1 \\
z & 0 & \cdots & 0 
\end{array}
\right)
\qquad\text{and}\qquad
X^{-1}=\left(
\begin{array}{cccc}
0 & \cdots & 0 & z^{-1} \\
1 & 0 & & 0 \\
 & \ddots& \ddots&  \vdots \\
0 & & 1 & 0 
\end{array}
\right)\;.
\end{equation}
As the notation suggests, both matrices are the inverse of each other with matrix multiplication defined over the ring $\R=\mb{C}[z,z^{-1}]$. In fact, the following matrix identities in $\End_{\R} V[z,z^{-1}]$ hold true:
\begin{lemma}\label{lem:XP}
The Lie algebra elements $P_{r}\in\g_n[z,z^{-1}]$ with $r\in\mb{Z}$ defined in \eqref{P_r} and \eqref{allPs} are powers of the matrices \eqref{X}, 
\begin{equation}\label{XP}
P_r=X^r\;,\qquad\forall r\in\mb{Z}\;.
\end{equation}
In particular, $X^{mn}=z^{m}\bs{1}_n$, where $\bs{1}_n$ is the $n\times n$ identity matrix and $m\in\mb{Z}$. 
\end{lemma}
\begin{proof}
This is immediate from the definition \eqref{P_r} noting that the $e_{ij}$ are the unit matrices whose only nonzero entry is in the $i$th row and $j$th column.
\end{proof}

Consider the tensor product $(V[z,z^{-1}])^{\otimes k}\cong \bV_k=\R\otimes V^{\otimes k}$, where we identify both tensor products in the obvious manner via 
$$(f_1(z)v_{i_1})\otimes\cdots\otimes(f_k(z)v_{i_k})\mapsto (f_1(z)\cdots f_k(z))\otimes v_{i_1}\otimes\cdots\otimes v_{i_k}.$$ 
Let $\rS_k$ act on the right by permuting factors in $V^{\otimes k}$. Denote by $X_i^{\pm 1}$ the matrix which acts by multiplication with $X^{\pm 1}$ in the $i$th factor and trivially everywhere else. Define a $\hat \rS_k$-action and, hence, a $\cH_k$-action on $\bV_k$ by setting for $w\in \rS_k$, $\lambda\in\cP_k$,
\begin{equation}\label{Haction}
v_{i_{1}}\otimes \cdots \otimes v_{i_{k}} .w x^\lambda=X^{-\lambda_1}v_{i_{w(1)}}\otimes \cdots \otimes X^{-\lambda_k}v_{i_{w(k)}}\,.
\end{equation}
N.B. this $\cH_k$-action is not faithful, since we infer from \eqref{XP} that for any $\lambda\in\cP_k$ and $w\in \rS_k$ the element $w(x^\lambda-x^{\lambda+nm\alpha})$ with $m\in\mb{Z},\alpha\in\mc{Q}_k$ is sent to zero.

\begin{lemma} 
The action \eqref{Haction} factors through the quotient
\begin{equation}\label{Hquotient}
\cH_k(n):=\cH_k/\langle x_1^n-z\rangle\;.
\end{equation}
\end{lemma}
\begin{proof}
This is a direct consequence of the definition of the action and that $\{x^\lambda w~|~\lambda\in\cP_k,w\in \rS_k\}$ is a basis of $\cH_k$. 
\end{proof}
Besides the right $\cH_k$-action we also have a natural left action of the enveloping algebra $U=U(\mathfrak{gl}_n[z,z^{-1}])$ via the coproduct  $\Delta :U\rightarrow U\otimes U$. %
This left action of the loop algebra commutes with the right action of the symmetric group $\rS_k$ on $\bV_k$, which permutes the factors in the tensor product. Schur-Weyl duality states that the images of $\R[\rS_k]$ and $U$ in $\End_{\R} \bV_k$ are centralisers of each other. Therefore, given any element in the centre $\mc{Z}(\cH_k)\cong\R[x^{\pm}_1,\ldots,x_k^{\pm 1}]^{\rS_k}$ its image must coincide with the image of an element in $U$. We will now show that the ring of symmetric polynomials in the $X^{\pm 1}_i$ coincides with the image of $U'$ (defined after Lemma \ref{lem:Lambda2U}). 

\begin{lemma}\label{lem:P2X} 
We have the following identities in $\End_{\R} \bV_k$,
\begin{equation}
\Delta^{k-1}(P_r)=\sum_{i=1}^kX_i^r,\qquad\forall r\in\mb{Z},
\end{equation}
where the $P_r$ are defined in \eqref{P_r} and \eqref{allPs}. In particular, $\Delta^{k-1}(P_{mn})=z^m k\op{Id}_{\bV_k}$.
\end{lemma}
\begin{proof}
The assertion follows from Lemma \ref{lem:XP} and the definition of the coproduct $\Delta:U\to U\otimes U$.
\end{proof}
It follows that for $\lambda\in\cP_k^+$ the images of $S_\lambda$ and $M_\lambda$ in $\End_{\R} \bV_k$ can be written in terms of the variables $\{X_1,\ldots, X_k\}$ using the familiar definitions of Schur and monomial symmetric functions from $\Lambda_k=\mathbb{C}[x_1,\ldots,x_k]^{\rS_k}$,
\begin{equation}
\Delta^{k-1}(S_\lambda)=\sum_{|T|=\lambda}X^T\quad\text{ and }\quad \Delta^{k-1}(M_\lambda)=\sum_{\mu\sim\lambda}X^\mu,
\end{equation}
where the first sum runs over all semi-standard tableaux $T$ of shape $\lambda$ and the second sum runs over all distinct permutations $\mu$ of $\lambda$. The analogous expressions apply to $\Delta^{k-1}(S_\lambda^*)$ and $\Delta^{k-1}(M_\lambda^*)$ by replacing $X_i$ with $X^*_{i}=X_i^{-1}$. In what follows, we will drop the coproduct $\Delta^{k-1}$ from the notation to unburden formulae and it will always be understood that the elements $M_\lambda$, $S_\lambda$ and their adjoints act on tensor products via the natural action given by the coproduct.

Since the power sums $\{p_\lambda\}_{\lambda\in\cP^+_k}$ with $p_\lambda=p_{\lambda_1}\ldots p_{\lambda_k}$ form a basis of $\Lambda_k$, we have as an immediate consequence the following corollary.
\begin{coro}
Let $A$ be the image of $\mathcal{Z}(\cH_k)$ and $B$ the image of $U'$ %
in $\End_{\R} \bV_k$, then $A=B$. In particular, $B\subset \End_{\cH_k}\bV_k$.
\end{coro}

\subsection{The cyclotomic quotient and the generalised symmetric group}
In the context of affine Hecke algebras one is usually interested in representations where the action of the polynomial part 
$\R[x_1^{\pm 1},\ldots,x_k^{\pm 1}]\subset \cH_k$ is semi-simple, that is, there should exist a common eigenbasis of the $X^{\pm 1}_i$. To this end, we consider 
\begin{equation}
\cH_k'=\RR\otimes_\R\cH_k,\qquad\RR=\R[t]/(z-t^n)\,
\end{equation}
and the $\RR$-module
\begin{equation}\label{V'}
\bV'_k=\RR\otimes_{\mb{C}} V^{\otimes k}=\bigoplus_{\lambda\in\alc}\bV'_\lambda,\quad
\bV'_\lambda=\RR\otimes_{\mb{C}} V'_\lambda\,,
\end{equation}
where $V'_\lambda=\{v~|~e'_{ii}v=m_i(\lambda)v~\}$ are the $\g_n$-weight spaces  of the Cartan subalgebra $\h'$ in apposition defined in \eqref{Cartan'} and the direct sum runs over weights in the following set
\begin{equation}\label{alcove}
\alc=\{\lambda\in\cP_k^+~|~n\geq\lambda_1,\ldots,\lambda_k\geq 1\}\,.
\end{equation}
This decomposition into weight spaces of $\h'$ should be seen in connection with the algebra isomorphism $\phi:U\to U^{\Omega}$, where $U^{\Omega}=U(\mathfrak{gl}^\Omega[t,t^{-1}])$ is the enveloping algebra of the twisted loop algebra. Recall that the change in Cartan subalgebras and the isomorphism $\phi$ is implemented via the similarity transformation with $\mc{S}(t)$ defined in \eqref{modSt}. 

Define a $\cH'_k$-action on $\bV'_k$ analogous to \eqref{Haction} but setting $z^{\pm 1}=t^{\pm n}$ in \eqref{X}.
\begin{prop} 
(i)  The action of $\cH'_k$ on $\bV'_k$ factors through the cyclotomic quotient $\cH'_k(n)=\cH'_k/\mc{I}_n$, where $\mathcal{I}_{n}$ is the two-sided ideal generated by the degree $n$ polynomial
\begin{equation}\label{cyclotomic}
f(x_1)=(x_1-t)(x_1-t\zeta)\cdots (x_1-t\zeta^{n-1})=x_1^n-t^n\,. 
\end{equation}
(ii) The action of the abelian subalgebra $\RR[x^{\pm 1}_1,\ldots,x^{\pm 1}_k]\subset \cH'_k$ on $\bV'_\lambda$ is diagonal and, in particular, the centre $\mathcal{Z}(\cH'_k)$ acts by multiplication with symmetric functions in the variables $(t^{\mp 1}\zeta^{\pm\lambda_1},\dots,t^{\mp 1}\zeta^{\pm\lambda_k})$.  
\end{prop}

\begin{proof}
The assertion (i) is a direct consequence of the matrix identities \eqref{XP} when $r\in n\mb{Z}$: setting $z=t^n$ they imply that the ideal $\mc{I}_n$ defined via \eqref{cyclotomic} is mapped to zero. Statement (ii) follows from exploiting the $\mc{S}$-matrix defined in \eqref{modS1} and its $t$-deformed version \eqref{modSt} which diagonalises the principal Heisenberg algebra; compare with Lemma \ref{lem:twistedHeisenberg}. Thus, we deduce from \eqref{XP} that the similarity transformation with $\mc{S}(t)^{\otimes k}$ on $\bV'_k$ diagonalises the action of $\RR[x^{\pm 1}_1,\ldots,x^{\pm 1}_k]\subset \cH_{n,k}$ on each of the weight spaces $\bV'_{\lambda}$ in \eqref{V'}. 
\end{proof}
The quotient $\cH'_{k}(n)$ is closely related to the group ring of the {\em generalised symmetric group} $\GS$ \cite{specht1932}. %
Recall that the generalised symmetric group is a special case of the family $G(n,p,k)$ of complex reflection groups in the Shephard-Todd classification with $p=1$. It can be defined as the {\em wreath product} of the symmetric  group $\rS_k$ with the cyclic group $C_n$ of order $n$ (which we identify with the roots of unity of order $n$ in $\mb{C}$),
\begin{equation}
\GS=C_n^{\times k}\rtimes \rS_{k}\;.
\end{equation}
Here $\rS_k$ acts on the $k$-fold direct product $C_n^{\times k}$ by permutation of indices. Consider the following exact sequence of groups, %
$1\rightarrow C_{n}^{\times k}\hookrightarrow
\GS\twoheadrightarrow \rS_{k}\rightarrow 1$, %
then we denote by $\mathcal{N}$ the normal subgroup which is the image of $C_{n}^{\times k}$. Let $ y_i\in\mathcal{N}$ be the image of the generator of the $i$th copy of $C_n$ under the natural isomorphism $\mathcal{N}\cong C_{n}^{\times k}$. 

\begin{lemma}\label{lem:ringiso}
The map
\begin{equation}
wx^{\lambda }\mapsto t^{\sum_i\lambda_i} wy^{\lambda}\,,\qquad\lambda\in\cP_k\;
\end{equation} 
defines a ring isomorphism $\cH'_{k}(n)\cong\RR[\GS]$. 
\end{lemma}
\begin{proof}
Because of Lemma \ref{lem:hecke} (i) one deduces that a basis of $\cH'_{k}(n)$ is given by $\{x^\lambda w~|~n\geq \lambda_i\geq 1,\,w\in \rS_k\}$. Since $\{y^\lambda w~|~n\geq \lambda_i\geq 1,\,w\in \rS_k\}$ is a basis of $\RR[\GS]$ and $t^n y_i^n=t^n$ for all $i=1,\ldots,k$ as well as $\sigma_iy_i\sigma_i=y_{i+1}$ and $\sigma_i y_j=y_j\sigma_i$ for $j\neq i,i+1$ the assertion follows.
\end{proof}

Note that in general the centre $\mc{Z}(\cH'_k)$ does not map surjectively onto the centre of $\RR[\GS]$. For example, set $n=3$ and $k=2$ then the centre of the group algebra of $G(3,1,2)$ is 9-dimensional while the image of the centre of $\mc{Z}(\cH'_k)$ for $t=1$ is six-dimensional; see the example \cite[p.792]{ariki1996decomposition}. A basis of the centre for general wreath products with $\rS_k$ has been put forward in \cite{pushkarev1999representation}. However, in this article we are only interested in the image of the centre $\mc{Z}(\cH'_k)$ as we will be projecting onto the subspaces of symmetric and alternating tensors in \eqref{V'} below.

\begin{coro}
Let $M$ be a $\RR[\GS]$-module then $M$ is a $\cH'_{k}$-module.
\end{coro}
\begin{proof}
This is a direct consequence of Lemma \ref{lem:ringiso}.
\end{proof}
Since $\cH'_{k}(n)$ is essentially the group algebra of $\GS$ its irreducible modules are given in terms of the irreducible modules of $\GS$. The latter are known \cite{osima1954representations} to be the modules induced by the irreducible representations $\lambda:y^\mu\mapsto \zeta^{(\lambda,\mu)}$, $\lambda\in\alc$ of the normal subgroup $\mathcal{N}\cong\mathcal{C}_{n}^{\times k}$ and their associated stabiliser subgroup $G_\lambda\subset\GS$; see Appendix B for details. The following proposition identifies the weight spaces of the twisted loop algebra as irreducible $\RR[\GS]$-modules.
\begin{prop}\label{prop:V=L}
The weight spaces $\bV'_\lambda$ are irreducible $\cH'_{k}(n)$-modules that for $t=1$ are isomorphic to the irreducible $\GS$-modules $\mL_\lambda$ that are induced by the trivial representation of $G_\lambda\cong \rS_\lambda$.
\end{prop}
\begin{proof}
We state the isomorphism $V^{\otimes k}\rightarrow
\bigoplus_{\lambda \in\alc}\mL_{\lambda }$. 
Define a basis in $\bV'_k$ in terms of the discrete Fourier transform (\ref{modSt}) by setting 
\begin{eqnarray*}
v'_{p_{1}}\otimes \cdots \otimes v'_{p_{k}}
&=&t^{\sum_i p_i}\sum_{1\leq a_{1},\ldots ,a_{k}\leq n}\mathcal{%
S}^{-1}_{p_{1}a_{1}}v_{a_{1}}\otimes \cdots \otimes \mathcal{S}^{-1}%
_{p_{k}a_{k}}v_{a_{k}} \\
&=&t^{\sum_i p_i}\sum_{1\leq a_{1},\ldots ,a_{k}\leq n}\frac{\zeta ^{-(a, p)}}{n^{k/2}}~v_{a_{1}}\otimes \cdots \otimes v_{a_{k}}
\end{eqnarray*}%
where $p\in\alc \rS_k$ is any distinct permutation of an element in the alcove \eqref{alcove}. We map the vector $%
\otimes_i v'_{p_i}$ onto the unique $n$-tableau $T_p=T^{(1)}\otimes\cdots\otimes T^{(n)}$ of shape $(m_{1}(p),\ldots
,m_{n}(p))$ that fills the horizontal strip of length $m_{j}(p)$ with those $i\in \{1,\ldots ,k\}$ for which $p_{i}=j$; see Appendix B. Since all entries in each horizontal strip have to increase strictly from left to right this fixes the $n$-tableau $T_p%
=T^{(1)}\otimes \cdots \otimes T^{(n)}$ uniquely. Conversely, any $n$-tableau of the same shape defines uniquely a vector $\otimes_i v'_{q_i}$ with $q\in p\rS_k$. The latter span the irreducible representation $\mL_\lambda$ with $\lambda=(1^{m_1(p)}\ldots n^{m_n(p)})$. By construction the map  $\otimes_i v'_{p_i}\mapsto T_p$ preserves the (left) action of the normal subgroup $\mathcal{N}$, $(\otimes_i v'_{p_i}).x^\lambda=\zeta^{(\lambda,p)}(\otimes_i v'_{p_i})=y^{\lambda}(\otimes_i v'_{p_i})$ for $t=1$ see \eqref{actionzeta}, and of the stabiliser group $G_{p}\subset\GS$.
\end{proof}

\subsection{Characters and fusion product}
Recall from Appendix B that the irreducible representations $L(\bs{\lambda})$ of $\GS$ are labelled by $n$-multi\-par\-titions $\bs{\lambda}=(\lambda^{(1)},\dots,\lambda^{(n)})$ with $\sum_{i=1}^n|\lambda^{(i)}|=k$. Given such a $\boldsymbol{\lambda}$ we call the unique partition $\lambda=(1^{m_1}2^{m_2}\cdots n^{m_n})\in\alc$ with $m_i=|\lambda^{(i)}|$ its {\em type} and call two irreducible modules $\mL(\boldsymbol{\lambda})$ and $\mL(\boldsymbol{\mu})$ of the same type, if $\lambda=\mu$. The following result is taken from \cite{osima1954representations}.

\begin{lemma} (i) The restricted modules $\operatorname{Res}_{\mathcal{N}}^{\GS}\mL(\boldsymbol{\lambda})$ and $\operatorname{Res}_{\mathcal{N}}^{\GS}\mL(\boldsymbol{\mu})$ are isomorphic as $\mathcal{N}$-modules if and only if they are of the same type. (ii) The characters of $\mL(\boldsymbol{\lambda})$ restricted to the normal subgroup $\mathcal{N}$ are given by
\begin{equation}
\chi_{\boldsymbol{\lambda}} (y^{p}) =\Tr_{\mL(\boldsymbol{\lambda})}y^p=f_{\boldsymbol{\lambda}}\, m_{\lambda}(\zeta^{p}),
\qquad f_{\boldsymbol{\lambda}}=\prod_{i=1}^{n} f_{\lambda^{(i)}}\,, 
\end{equation}
where $p\in\cP_k$, $f_{\lambda^{(i)}}$ is the number of standard tableaux of shape $\lambda^{(i)}$ and $m_\lambda$ denotes the monomial symmetric function. That is,
\[
m_{\lambda}(\zeta^{p})=\sum_{\mu }\zeta ^{p _{1}\mu _{1}}\cdots \zeta
^{p_{k}\mu _{k}}~, 
\]
where the sum runs over all \emph{distinct} permutations $\mu$ of $\lambda$.
\end{lemma}
\begin{proof} 
From \eqref{actionzeta}, notice that the action of $y^{p}$ on $T$
does not depend on the shape of each $\lambda^{(i)}$, but only on $ \lvert \lambda^{(i)} \rvert$.
We thus have
\begin{equation}
\chi_{\boldsymbol{\lambda}} (y^{p}) = f_{\boldsymbol{\lambda}}\;\chi_{\lambda} (y^{p}),\qquad\forall p\in\cP_k\,,
\end{equation}
where $\chi_\lambda$ is the character of the irreducible modules appearing in Proposition \ref{prop:V=L}. This proves, (i) and (ii).
\end{proof}

Denote by $\operatorname{Rep}\GS$ the representation ring of the generalised symmetric group with structure constants
\[
\mL(\boldsymbol{\lambda})\otimes\mL(\boldsymbol{\mu})=\bigoplus_{\boldsymbol{\nu}}c_{\boldsymbol{\lambda}\boldsymbol{\mu}}^{\boldsymbol{\nu}}\mL(\boldsymbol{\nu})\,.
\]
Since the normal subgroup $\mathcal{N}$ is abelian and finite, we can identify its representation ring $\operatorname{Rep}\mathcal{N}$ with its character ring, $\operatorname{Char}(\mathcal{N})\cong\mathbb{Z}_n^k$. Consider the map $\operatorname{Rep}\GS\rightarrow \operatorname{Char}(\mathcal{N})$ given by $\mL(\boldsymbol{\lambda})\mapsto\chi_{\boldsymbol{\lambda}}|_{\mathcal{N}}$. Its inverse image leads to the definition of equivalence classes $[\mL(\boldsymbol{\lambda})]$ of irreducible $\GS$-modules of type $\lambda\in\alc$. That is, $\mL(\bs{\lambda})\sim\mL(\bs{\mu})$ if the multipartitions $\bs{\lambda}$ and $\bs{\mu}$ are of the same type. In particular, the irreducible modules $\{\mL_\lambda\}_{\lambda\in\alc}$ whose $n$-multipartitions are given by $n$ horizontal strips of boxes of length $m_{i}(\lambda )$, form a set of class representatives. These irreducible modules are the ones induced by the trivial representation of the Young subgroup $\rS_\lambda\cong G_\lambda$. Thus, we have arrived at the following:
\begin{prop}
The quotient  $\operatorname{Rep}\GS/\!\!\thicksim$ endowed with the {\em fusion product}
\begin{equation}\label{Lfusion}
[\mL_\lambda][\mL_\mu]
=\bigoplus_{\nu\in\alc}N_{\lambda\mu}^\nu[\mL_{\nu}],\qquad
N_{\lambda\mu}^{\nu}=\sum_{\operatorname{type}(\boldsymbol{\nu})=\nu}c_{\boldsymbol{\lambda}\,\boldsymbol{\mu}}^{\boldsymbol{\nu}}
\,\frac{f_{\boldsymbol{\nu}}}{f_{\boldsymbol{\lambda}}f_{\boldsymbol{\mu}}}\,,
\end{equation}
is isomorphic to  $\operatorname{Char}(\mathcal{N})$. Here $\boldsymbol{\lambda}$ and $\boldsymbol{\mu}$ in the definition of $N_{\lambda\mu}^\nu$ are any pair of multi-partitions of type $\lambda$ and $\mu$.
\end{prop}

A combinatorial version of the same result is the following product expansion of specialised monomial symmetric functions.

\begin{coro}
We have the following product expansion of monomial symmetric functions at roots of unity
\begin{equation}\label{mfusion}
m_\lambda(\zeta^p)m_\mu(\zeta^{p})=\sum_{\nu\in \alc}N_{\lambda\mu}^{\nu}m_\nu(\zeta^p) ,\quad\forall p\in\cP_k,
\end{equation}
where the expansion coefficients are the fusion coefficients in \eqref{Lfusion}. %
\end{coro}
Noting that $m_\lambda(1,\ldots,1)=d_\lambda$ for any $\lambda\in\alc$, we obtain as a special case of the last result the following identity for the quantum dimensions \eqref{qdim},
\begin{equation}
d_\lambda d_\mu =\sum_{\nu\in \alc}N_{\lambda\mu}^{\nu}d_\nu,\qquad d_\lambda=\dim L_\lambda\,.
\end{equation}
Of course the same identity follows directly from \eqref{Lfusion} and \eqref{dimL}.

\section{Frobenius structures on symmetric and alternating tensors}
Recall the definition of the symmetriser and antisymmetriser in $\mathbb{C}[\rS_k]$,
\begin{equation}\label{(anti)symmetriser}
e^{\pm}_k=\frac{1}{k!}\sum_{w\in \rS_k}(\pm 1)^{\ell(w)}w\,.
\end{equation}
In what follows we exploit Schur-Weyl duality and consider the $U^\Omega$-modules  obtained by projecting onto the symmetric, $\rS^kV=e_k^+V^{\otimes k}$, and antisymmetric, $\bigwedge^kV=e_k^-V^{\otimes k}$, subspaces for each fixed $k$. Applying the idempotents $e^{\pm}_k$ each weight space $\bV'_\lambda$ is mapped to a one-dimensional subspace simplifying the combinatorial description of the action of the subalgebra  $U'$ in the next section.

In this section we shall show that both subspaces, the symmetric and alternating tensors, carry the structure of symmetric Frobenius algebras with the algebra product defined in terms of the action of $U'$. As mentioned in the introduction Frobenius algebras are categorically equivalent to 2D TQFT. In the language of category theory a 2D TQFT is a monoidal functor $Z:\op{2Cob}\to\op{Vec}$ from the category of two-dimensional cobordisms $\op{2Cob}$ (generated via concatenation from the elementary 2-cobordisms shown below), 
\begin{equation}\label{2cob}%
\begin{tikzpicture}[baseline={([yshift=-.5ex]current bounding box.center)},tqft/boundary separation=30pt,tqft/cobordism height=30pt]
\pic[tqft,
incoming boundary components=1,
outgoing boundary components=1,
circle x radius=8pt,
circle y radius=4pt,
cobordism edge/.style={draw},
every boundary component/.style={draw,rotate=0},
rotate=0, name=a
];
\end{tikzpicture}
\quad
\begin{tikzpicture}[baseline={([yshift=-.5ex]current bounding box.center)},tqft/boundary separation=30pt,tqft/cobordism height=30pt]
\pic[tqft pair of pants,
circle x radius=8pt,
circle y radius=4pt,
cobordism edge/.style={draw},
every incoming boundary component/.style={draw, rotate=0},
every outgoing boundary component/.style={draw, rotate=0},
rotate=0, name=a
];
\end{tikzpicture}
\quad
\begin{tikzpicture}[baseline={([yshift=-.5ex]current bounding box.center)},tqft/boundary separation=30pt,tqft/cobordism height=30pt]
\pic[tqft,
incoming boundary components=0,
outgoing boundary components=2,
circle x radius=8pt,
circle y radius=4pt,
cobordism edge/.style={draw},
every boundary component/.style={draw,rotate=0},
rotate=0, name=a
];
\end{tikzpicture}
\quad
\begin{tikzpicture}[baseline={([yshift=-.5ex]current bounding box.center)},tqft/boundary separation=30pt,tqft/cobordism height=30pt]
\pic[tqft,
incoming boundary components=1,
outgoing boundary components=0,
circle x radius=8pt,
circle y radius=4pt,
cobordism edge/.style={draw},
every boundary component/.style={draw,rotate=0},
rotate=0, name=a
];
\end{tikzpicture}
\quad\text{ and }\quad
\begin{tikzpicture}[baseline={([yshift=-.5ex]current bounding box.center)},tqft/boundary separation=30pt,tqft/cobordism height=30pt]
\pic[tqft reverse pair of pants,
circle x radius=8pt,
circle y radius=4pt,
cobordism edge/.style={draw},
every boundary component/.style={draw,rotate=0},
rotate=0, name=a
];
\end{tikzpicture}
\quad
\begin{tikzpicture}[baseline={([yshift=-.5ex]current bounding box.center)},tqft/boundary separation=30pt,tqft/cobordism height=30pt]
\pic[tqft,
incoming boundary components=2,
outgoing boundary components=0,
circle x radius=8pt,
circle y radius=4pt,
cobordism edge/.style={draw},
every boundary component/.style={draw,rotate=0},
rotate=0, name=a
];
\end{tikzpicture}
\quad
\begin{tikzpicture}[baseline={([yshift=-.5ex]current bounding box.center)},tqft/boundary separation=30pt,tqft/cobordism height=30pt]
\pic[tqft,
incoming boundary components=0,
outgoing boundary components=1,
circle x radius=8pt,
circle y radius=4pt,
cobordism edge/.style={draw},
every boundary component/.style={draw,rotate=0},
rotate=0, name=a
];
\end{tikzpicture}
\;\;,
\end{equation}
to the category of finite-dimensional vector spaces $\op{Vec}$ (here over $\mb{C}$); see the textbook \cite{kock2004frobenius} for further details. Identifying 2-cobordisms which are homeomorphic implies that the image $\V$ of a circle under $Z$ must carry the structure of a symmetric Frobenius algebra with the 2-cobordisms in \eqref{2cob} (from left to right) corresponding to the identity map $\op{id}:\V\to\V$, multiplication $m:\V\otimes\V\to\V$, bilinear form $\eta:\V\otimes\V\to\mb{C}$, unit $e:\mb{C}\to\V$ and coproduct $\delta:\V\to\V\otimes\V$, co-form $\eta^*:\mb{C}\to\V\otimes\V$, co-unit or Frobenius trace $\veps:\V\to\mb{C}$.  We now define for each $k$ and $n$ such monoidal functors by explicitly constructing the latter maps for symmetric and alternating tensors in $\bV'_k$.

\subsection{Divided powers and a representation of the modular group}
For $\lambda\in\alc$ define the associated {\em divided powers} as the vector
\begin{equation}
v_\lambda=d_\lambda\, e^+_k (v_{\lambda_k}\otimes\cdots\otimes v_{\lambda_2}\otimes v_{\lambda_1}),
\end{equation}
which have the property that $e^+_{2k}(v_\lambda\otimes v_\mu)=d_\lambda d_\mu\,v_{\lambda\cup\mu}/d_{\lambda\cup\mu}$. Here $\lambda\cup\mu$ denotes the composition $(\lambda_1,\ldots,\lambda_k,\mu_1,\ldots,\mu_k)$. Let $V^*$ be the dual space of $V$ and denote by $\{v^1,\ldots,v^n\}$ the dual basis with $v^{i}(v_j)=\delta_{ij}$. Then we denote by $v^\lambda=\frac{1}{k!}\sum_{w\in \rS_k} v^{\lambda_{w(k)}}\otimes\cdots\otimes v^{\lambda_{w(1)}}$ the invariant tensor in $(V^*)^{\otimes k}$ that under the natural pairing satisfies $\langle v^{\lambda},v_\mu\rangle=\delta_{\lambda\mu}$. %

\begin{lemma}
The matrix elements of the $\cS$-matrix \eqref{modSt} in the basis of divided powers are given by
\begin{equation}\label{modSbos}
\cS_{\lambda\mu}(t)=\langle v^{\mu},\cS^{\otimes k}(t)v_\lambda\rangle=t^{-|\lambda|}\frac{m_{\lambda}(\zeta^\mu)}{\sqrt{n^{k}}}=t^{|\mu|-|\lambda|}
\frac{d_\lambda}{d_\mu}\cS_{\mu\lambda}(t)\;.
\end{equation}
In addition, the matrix elements of the inverse matrix are (setting $\bar t=t^{-1}$)
\begin{equation}\label{modSbosi}
\cS^{-1}_{\lambda\mu}(t)=t^{|\mu|}\frac{m_{\lambda}(\zeta^{-\mu})}{\sqrt{n^{k}}}=
t^{|\mu|-|\lambda|}\overline{\mc{S}_{\lambda\mu}(t)}=t^{|\mu|+|\lambda^*|}\mc{S}_{\lambda^*\mu}(t)
=t^{|\lambda|+|\mu|}\zeta^{-|\lambda|}\mc{S}_{\lambda\mu^\vee}(t),
\end{equation} 
where %
$
\lambda^*=(1^{m_{n-1}(\lambda)}2^{m_{n-2}(\lambda)}\ldots (n-1)^{m_1(\lambda)}n^{m_n(\lambda)})
$ and %
$\mu^\vee=(n+1-\mu_k,\ldots,n+1-\mu_1)$. 
The matrix elements of the $\cT$-matrix \eqref{modT1} in the basis of divided powers read,
\begin{equation}
\mc{T}_{\lambda\mu}=\langle v^{\mu},\mc{T}^{\otimes k}v_\lambda\rangle=\delta_{\lambda\mu}\,\zeta^{-\frac{kn(n-1)}{24}}
\prod_{i=1}^k\zeta^{\frac{\lambda_i(n-\lambda_i)}{2}}\,.
\end{equation}
Setting $t=1$ the matrices $\mc{S}=\mc{S}(1)$, $\mc{T}$ obey the relations \eqref{mod_relations} and \eqref{mod_relations2} with the matrix $\mc{C}$ given by $\mc{C}_{\lambda\mu}=\langle v^{\mu},\mc{C}^{\otimes k}v_\lambda\rangle=\delta_{\lambda^*\mu}$.
\end{lemma}

\begin{proof}
All of these identities are a direct consequence of the relations \eqref{mod_relations} and \eqref{mod_relations2} from Proposition \ref{STC} for $k=1$ and the definition of the matrix elements in the basis of divided powers. (N.B. here we have changed conventions and setting $k=1$ all of the resulting matrices are the transpose of the matrices considered previously for $k=1$ in Section 2.) In particular, note that the matrices $\mc{C}$ and $\mc{\hat C}$ from \eqref{C1} in the basis of divided powers read $\mc{C}_{\lambda\mu}=\langle v^{\mu},\mc{C}^{\otimes k}v_\lambda\rangle=\delta_{\lambda^*\mu}$ and $\mc{\hat C}_{\lambda\mu}=\langle v^{\mu},\mc{\hat C}^{\otimes k}v_\lambda\rangle=\delta_{\op{rot}(\lambda)\mu}$ with $\op{rot}(\lambda)=(1^{m_n(\lambda)}2^{m_1(\lambda)}\ldots n^{m_{n-1}(\lambda)})$. Their product then yields $(\mc{\hat C}\mc{C})_{\lambda\mu}=\delta_{\lambda^\vee\mu}$ and the last identity in \eqref{modSbosi} then follows. 
\end{proof}
As we will be making repeated use of them, it is worthwhile to express some of the modular $\cS$-matrix relations in terms of the matrix elements \eqref{modSbos} which amount to non-trivial summation formulae of monomial symmetric functions $m_\lambda$ and their augmented counterparts $m^\lambda=|\rS_\lambda|m_\lambda$ with $\lambda\in\alc$ when specialised at roots of unity.
\begin{coro}\label{cor:complete}
We have the identities 
\begin{equation}
\sum_{\sigma\in\alc}\frac{m_{\lambda}(\zeta^\sigma)m^{\mu}(\zeta^{-\sigma})}{n^k|\rS_\sigma|}=
\sum_{\sigma\in\alc}\frac{m_{\lambda}(\zeta^\sigma)m^{\mu^*}(\zeta^{\sigma})}{n^k|\rS_\sigma|}=
\delta_{\lambda\mu}
\end{equation}
and
\begin{equation}
\sum_{\sigma\in\alc}m_{\sigma}(\zeta^\lambda)m^{\sigma}(\zeta^{-\mu})=\delta_{\lambda\mu}n^k|\rS_\lambda|
=\delta_{\lambda\mu}\frac{|G(n,1,k)|}{\dim\mL_\lambda}\;.
\end{equation}
\end{coro}
\begin{proof}
All identities are a direct consequence of the relations \eqref{mod_relations} and \eqref{mod_relations2} for $k=1$.
\end{proof}

\begin{coro}\label{cor:Verlinde_bos}
Recall the definition \eqref{M} of the loop algebra element $M_\lambda\in U_+$ and denote by $M^*_\lambda\in U_-$ its adjoint. Then we have the following Verlinde-type formula for their matrix elements, 
\begin{equation}\label{Verlinde_bos}
\langle v^\nu,M^*_\lambda v_\mu\rangle
=\frac{|\rS_\nu|}{|\rS_\mu|}\,\overline{\langle v^\mu,M_\lambda v_\nu\rangle}
=\sum_{\sigma\in\alc}\frac{\cS_{\lambda\sigma}(t)\cS_{\mu\sigma}(t)\cS^{-1}_{\sigma\nu}(t)}{t^{kn}\cS_{n^k\sigma}(t)}
=t^{|\nu|-|\lambda|-|\mu|}N_{\lambda\mu}^{\nu},
\end{equation}
where $N_{\lambda\mu}^{\nu}$ are the fusion coefficients \eqref{Lfusion}. In particular, $N_{\lambda\mu}^{\nu}=0$ unless $|\lambda|+|\mu|-|\nu|=0\mod n$.
\end{coro}

\begin{proof}
Note that $v'_\mu=(\cS^{-1}(t))^{\otimes k}v_\mu$ is an eigenvector of $M^*_\lambda\in U_-$ with eigenvalue $t^{-|\lambda|}m_{\lambda}(\zeta^{\mu})$ and $\cS_{n^k\sigma}(t)=t^{-nk}m_{\sigma}(1,\ldots,1)/\sqrt{n^k}d_\sigma=t^{-nk}/\sqrt{n^k}$. Therefore,
\[
\langle v^\nu,M^*_\lambda v_\mu\rangle=\sum_{\sigma\in\alc}\langle v^\nu,v'_\sigma\rangle \frac{\cS_{\lambda\sigma}(t)\cS_{\mu\sigma}(t)}{t^{nk}\cS_{n^k\sigma}(t)}
\]
which gives the first equality in \eqref{Verlinde_bos}. To prove the second equality we set first $t=1$.  Then the identity \eqref{mfusion} can be rewritten as
\[
\frac{\cS_{\lambda\sigma}}{\cS_{n^k\sigma}}\frac{\cS_{\mu\sigma}}{\cS_{n^k\sigma}}=
\sum_{\upsilon\in\alc}N_{\lambda\mu}^\upsilon\frac{\cS_{\upsilon\sigma}}{\cS_{n^k\sigma}}\;.
\]
Multiplying on both sides with $\cS^{-1}_{\sigma\nu}\cS_{n^k\sigma}$ and summing over $\sigma\in\alc$ gives the desired result for $t=1$. The identity for general $t$ then follows by noting that the monomial symmetric function $m_\lambda$ is homogeneous of degree $|\lambda|$. The final statement is a direct consequence of observing that the matrix elements of $M_\lambda^*=\sum_{p}X^{-p}$, with the sum running over all distinct permutations $p$ of $\lambda$, only depend on powers of $z^{-1}=t^{-n}$ according to the definition \eqref{X}.
\end{proof}

Note the identity $d_\lambda=\dim\mL_\lambda=\cS_{\lambda n^k}/\cS_{n^kn^k}$ which explains our earlier convention to call the multinomial coefficients \eqref{qdim} {\em quantum dimensions}: they are the largest (integral) eigenvalue of the fusion matrix $N_\lambda$ with $\lambda\in\alc$.

As an easy consequence of these previous results one now derives several identities for the fusion coefficients noting that $d_\lambda=d_{\lambda^*}$ for $\lambda\in\alc$.  
\begin{coro}\label{cor:symmN}
One has the following equalities:
\begin{itemize}
\item[(i)] $N_{\lambda n^k}^{\mu}=\delta_{\lambda\mu}$ and $N_{\lambda\mu}^{\nu}=N_{\mu\lambda}^{\nu}$
\item[(ii)] $N_{\lambda\mu}^{\nu}=N_{\lambda^*\mu^*}^{\nu^*}=d_\lambda N_{\mu\nu^*}^{\lambda^*}/d_{\nu}$ %
and $N_{\lambda\mu}^{n^k}=\delta_{\lambda^*\mu}d_\lambda$
\item[(iii)] $N_{\rot^a\lambda\rot^b\mu}^{\rot^c\nu}=N_{\lambda\mu}^\nu$ for $a,b,c\in\mb{Z}$ with $a+b=c\mod n$, where 
\[
\rot^a \lambda=(1^{m_{1-a}(\lambda)}\ldots n^{m_{n-a}(\lambda)})
\]
and all indices are understood modulo $n$.
\end{itemize}
\end{coro}
\begin{proof}
The identities (i) and (ii) follow from the $\cS$-matrix identities \eqref{modSbos}, \eqref{modSbosi} for $t=1$. Similarly, (iii) can be deduced from the $\cS$-matrix identity (for $t=1$)
\[
\cS_{\rot(\lambda)\mu}=\zeta^{|\mu|}\cS_{\lambda\mu}=\overline{\cS}_{\mu,\rot(\lambda)}^{-1}
\] 
which follows from observing that $m_{\rot\lambda}(\zeta^\mu)=e_k(\zeta^\mu)m_\lambda(\zeta^\mu)=\zeta^{|\mu|}m_\lambda(\zeta^\mu)$. 
\end{proof}
We now endow $\V^+_k=\rS^kV\otimes\RR$ with the structure of a Frobenius algebra. First note that $\V^+_k$ is a free $\RR$-module of finite rank. Define a $\RR$-linear map $\varepsilon:\V^+_k\to\RR$ by setting 
\[
\veps(v_\lambda)=t^{-kn}n^k\delta_{\lambda n^k}\,.
\]
Then the induced map $\V^+_k\to\op{Hom}_{\RR}(\V^+_k,\RR)$ is a $\RR$-module isomorphism.
\begin{theorem}\label{thm:symmFrob}
Let $t_0\in\mb{C}^\times$ with $|t_0|=1$. Then $\V^+_k/(t-t_0)\V^+_k$ together with the fusion product $v_\lambda v_\mu:=M^*_\lambda v_\mu$ and the trace functional induced by $\veps$ is a commutative Frobenius algebra. Moreover, this algebra is semi-simple with idempotents $\e_\lambda=t_0^{nk}\cS_{n^k\lambda}(t_0)v'_\lambda$.
\end{theorem}
\begin{proof}
It follows from \eqref{Verlinde_bos} that the product $m:\V^+_k\otimes\V^+_k\to\V^+_k$ given by $m(v_\lambda,v_\mu)\equiv v_\lambda v_\mu=M^*_\lambda v_\mu$ is commutative. Associativity is then an easy consequence, %
$v_\lambda (v_\mu v_\nu)=M^*_\lambda (M^*_\nu v_\mu)=M^*_\nu (M^*_\lambda v_\mu)=(v_\lambda v_\mu) v_\nu$.
Define a bilinear form $\eta:\V^+_k\otimes\V^+_k\to\mb{C}$ and its dual $\eta^*:\mb{C}\to\V^+_k\otimes\V^+_k$ via
\begin{equation}\label{etas}
\eta(v_\lambda,v_\mu)=\veps(v_{\lambda}v_{\mu})=t_0^{-nk}n^kN_{\lambda\mu}^{n^k}
\quad\text{ and }\quad
\eta^*(1)=\sum_{\lambda\in\alc}m_\lambda\otimes m^{\lambda^*}\,,
\end{equation}
where $m^\lambda=|\rS_\lambda|m_\lambda$ with $\lambda\in\alc$ is the augmented monomial symmetric function. 
Employing Corollary \ref{cor:symmN} (ii) one deduces that $\eta$ is non-degenerate and, by definition, invariant. All the remaining maps corresponding to the other 2-cobordisms shown in \eqref{2cob} can now also be constructed. For example, using the 2-cobordisms in \eqref{2cob} one obtains for the coproduct $\delta:\V^+_k\to\V^+_k\otimes\V^+_k$ (the inverted pair-of-pants cobordism),
\begin{equation}\label{FrobDelta}
\delta=(\op{id}\otimes m)\circ(\eta^*\otimes\op{id})= (m\otimes\op{id})\circ(\op{id}\otimes\eta^*)\;.
\end{equation}

To show that the algebra is semi-simple, recall from the proof of Corollary \ref{cor:Verlinde_bos} that the modular $\mc{S}$-matrix diagonalises the fusion matrices $M^*_\lambda$ in the definition of the algebra product. This fixes the idempotents in terms of the eigenvectors $v'_\mu$. That the set of idempotents is complete follows from observing that the $\mc{S}$-matrix is invertible; compare with the identities from Corollary \ref{cor:complete}.
\end{proof}
Consider the ring $\RR[x_1,\ldots,x_k]^{\rS_k}\subset \mc{Z}(\cH'_k)$ of symmetric functions in $k$ variables over $\RR$ 
and denote by $\mc{J}_n$ the ideal generated by
\begin{equation}\label{pideal}
\mc{J}_n=\langle p_n-z\,k,p_{n+1}-z\,p_1,\ldots,p_{n+k-1}-z\,p_{k-1}\rangle\;,
 \end{equation}
 where $p_r=\sum_{i=1}^k x_i^{r}$ are the power sums in the variables $x_i$ and $z=t^n$.
 \begin{theorem}\label{thm:bosquotient}
 The map $\RR[x_1,\ldots,x_k]^{\rS_k}/\mc{J}_n\to\V^+_k$ which sends $m_\lambda(x^{-1}_1,\ldots,x^{-1}_k)\mapsto v_\lambda$ is a ring isomorphism. Here $x_i^{-1}=z x^{n-1}_i$ for $i=1,\ldots,k$.
\end{theorem}
  \begin{proof}
  We first show that $\mc{J}_n$ is equal to the ideal $\mc{J}'_n$ generated by the relations
  \begin{equation}\label{BAE}
  \langle x_1^{n}-z,\ldots,x_k^{n}-z\rangle\,.
  \end{equation}
One direction is trivial, the relations \eqref{pideal} are obviously satisfied if \eqref{BAE} hold. To show the converse recall the generating function for power sums, $P(u)=\sum_{i \ge 1} p_i u^{i-1}= \sum_{i=1}^k \frac{x_i}{1-ux_i}$, where the last expression is understood in terms of a geometric series expansion. From Newton's formulae it follows via a proof by induction that \eqref{pideal} implies $p_{n+r}-z\,p_r=0$ for all $r\geq 0$ and, hence, one shows that 
\begin{eqnarray*}
 P(u) = 
 \sum_{i=1}^n p_i u^{i-1} + z u^n \sum_{i \ge 1} p_i u^{i-1} = 
 \sum_{i=1}^n p_i u^{i-1} + z u^n P(u)
\end{eqnarray*}
which can be rearranged as (replacing $u$ with $u^{-1}$)
\begin{equation*}
 \sum_{i=1}^n p_i u^{n+1-i} = (u^n-z) P(u^{-1}) = (u^n-z ) \sum_{i=1}^k \frac{u x_i}{u-x_i}\,.
\end{equation*}
This implies that the formal series expansion of $(u^n-z)  P(u^{-1})$ in $u$ terminates after finitely 
many terms. Therefore, the residues of $(u^n-z)  P(u^{-1})$ at $u=x_i$, for $i=1,\dots, k$ must 
vanish which is equivalent to \eqref{BAE}. This proves the claim.

The relations \eqref{BAE} imply that the $x_i$ are invertible, $x_i^{-1}=zx_i^{n-1}$. Recall that $\mc{Z}(\cH'_k)$ acts on the one-dimensional subspace $e^+_k\bV'_\lambda\otimes\RR'$ by multiplication with symmetric polynomials in the variables $(t^{\pm 1}\zeta^{\mp\lambda_1},\ldots,t^{\pm 1}\zeta^{\mp\lambda_k})$ with $\lambda\in\alc$. Hence, employing the Nullstellensatz 
we see that $\RR[x_1,\ldots,x_k]^{\rS_k}/\mc{J}_n$ is isomorphic to the image of $\mc{Z}(\cH'_k)$ in $\End_{\RR}\bV^+_k$. But the action of $\mc{Z}(\cH'_k)$ on $\V_k^+$  fixes the ring structure upon noting that $v_\lambda=z^k M_\lambda^*v_{n^k}$ and $M_\lambda^*=\sum_{\mu\sim\lambda}X^{-\mu}$ in $\End_{\RR}\bV^+_k$. 
  \end{proof}

\subsection{Alternating tensors and Gromov-Witten invariants}
We now turn to the subspaces $\bigwedge^kV\subset V^{\otimes k}$ of alternating tensors. Consider the alcove of strict partitions,
\begin{equation}\label{strictA}
\salc=\alc\cap\cP^{++}_k=\{\lambda~|~n\geq \lambda_1>\cdots>\lambda_k> 0\}\;.
\end{equation}
Obviously the latter have $m_i(\lambda)=0$ or $1$ and, therefore, the stabiliser groups $\rS_\lambda$ are trivial for $\lambda\in\salc$. Note also that $\salc=\emptyset$ unless $k\leq n$. For $\lambda\in\salc$ fix the basis
\begin{equation}\label{alt_tensors}
v_{\bar \lambda}=v_{\lambda_k} \wedge \cdots\wedge v_{\lambda_2}\wedge v_{\lambda_1}=e^-_k v_{\lambda_k}\otimes
\cdots \otimes v_{\lambda_2}\otimes v_{\lambda_1}\,,
\end{equation}
where $\bar\lambda=\lambda-\rho$ with $\rho=(k,\ldots,2,1)$ is the reduced partition whose Young diagram lies inside a bounding box of width $n-k$ and height $k$. Denote by $v^{\bar\lambda}=\sum_{w\in \rS_k}(-1)^{\ell(w)}v^{\lambda_{w(k)}}\otimes\cdots\otimes v^{\lambda_{w(1)}}$ its dual basis in $(V^*)^{\otimes k}$ with respect to the natural pairing. 

Along a similar vein as in the symmetric case one finds the following presentation of $PSL(2,\mb{Z})$ on alternating tensors if $k$ is odd. Let $\lambda,\mu\in\salc$ and denote by $a_{\lambda}(x_1,\ldots,x_k)=\det(x_j^{\lambda_i})
_{1\leq i,j\leq k}$ the alternating monomial.
\begin{lemma}  The matrix elements of the $t$-deformed modular $\cS$-matrix \eqref{modSt} in the basis of alternating tensors read
\begin{equation}\label{modSfermi}
\cS_{\lambda\mu}(t)=(-1)^{\frac{k(k-1)}{4}}\langle v^{\bar\mu},\mc{S}^{\otimes k}(t)v_{\bar\lambda}\rangle=
(-1)^{\frac{k(k-1)}{4}}t^{-|\lambda|}\frac{a_{\lambda}(\zeta^\mu)}{\sqrt{n^k}}
\end{equation}
while the inverse matrix elements obey the identities
\begin{equation}\label{modSfermii}
\cS^{-1}_{\lambda\mu}(t)=(-1)^{(k-1)\delta_{\lambda_1n}}\cS_{\mu\lambda^*}(t^{-1})
=t^{(n+1)k}\zeta^{-|\lambda|}\cS_{\mu^\vee\lambda}(t)
=\overline{\mc{S}_{\mu\lambda}(t)}\,.
\end{equation}
For the $\mc{T}$-matrix we obtain similar to the symmetric case
\begin{equation}\label{modTfermi}
\mc{T}_{\lambda\mu}=(-1)^{-\frac{k(k-1)}{12}}\langle v^{\bar\mu},\mc{T}^{\otimes k}v_{\bar\lambda}\rangle=\delta_{\lambda\mu}\,\zeta^{-\frac{kn(n-1)+k(k-1)}{24}}\prod_{i=1}^k\zeta^{\frac{\lambda_i(n-\lambda_i)}{2}}
\end{equation}
But only for $k$ odd (and $t=1$) do both matrices, $\cS=\cS(1)$ and $\mc{T}$, obey the relations \eqref{mod_relations}, \eqref{mod_relations2} with $\mc{C}_{\lambda\mu}=\delta_{\lambda\mu^*}$. 
\end{lemma}
\begin{proof}
The relations are a straightforward consequence of the definitions and the properties of $\cS,\mc{T}$ for $k=1$. Note that the additional factor $(-1)^{\frac{k(k-1)}{4}}$ in the $\cS$-matrix is inserted because $\overline{a_\lambda(\zeta^\mu)}=\overline{a_\mu(\zeta^\lambda)}=(-1)^{\frac{k(k-1)}{2}}(-1)^{(k-1)\delta_{\lambda_1n}}a_{\lambda^*}(\zeta^{\mu})$. Similarly, we entered a factor $(-1)^{-\frac{k(k-1)}{12}}$ in the $\cT$-matrix so that $\mc{S}^2=(\mc{S}\mc{T})^3$ continues to hold. However, note that 
\[
\mc{C}^{\otimes k}v_{\bar \lambda}=v_{n-\lambda_k}\wedge\cdots\wedge v_{n-\lambda_1}=(-1)^{\frac{k(k-1)}{2}}(-1)^{(k-1)\delta_{\lambda_1n}}v_{\bar\lambda^*}\,.
\]
Hence the equality $\mc{S}^2=\mc{C}$ only continues to hold for $k$ odd.
\end{proof}
\begin{rema}
The presentation of $PSL(2,\mb{Z})$ for $t=1$ and $k$ odd has been previously considered by Naculich and Schnitzer \cite{naculich2007level} in connection with the $U(n)$ Wess-Zumino-Witten model. 
\end{rema}
In complete analogy with the previous case of divided powers we now consider the Verlinde formula for the case of alternating tensors. 
\begin{prop}\label{prop:BVI}
The matrix elements of the loop algebra element $S_{\bar\lambda}\in U_+$ defined in \eqref{defS} with $\lambda\in\salc$ and its adjoint $S^*_{\bar\lambda}\in U_-$  read
\begin{equation}\label{BVI_formula}
\langle v^{\bar\nu},S^*_{\bar\lambda} v_{\bar\mu}\rangle
=\overline{\langle v^{\bar\mu},S_{\bar\lambda} v_{\bar\nu}\rangle}
=\sum_{\sigma\in\salc}\frac{\cS_{\lambda\sigma}(t)\cS_{\mu\sigma}(t)\cS^{-1}_{\sigma\nu}(t)}{\cS_{\rho\sigma}(t)}
=(-1)^{d(k-1)}t^{-nd}C_{\bar\lambda\bar\mu}^{\bar\nu,d},
\end{equation}
where $nd=|\bar\lambda|+|\bar\mu|-|\bar\nu|$ and the structure constants $C_{\bar\lambda\bar\mu}^{\bar\nu,d}$ are determined by the following  product expansion of Schur functions at roots of unity,
\begin{equation}
s_{\bar\lambda}(\zeta^\sigma)s_{\bar\mu}(\zeta^\sigma)=\sum_{\nu\in\salc} (-1)^{d(k-1)}C_{\bar\lambda\bar\mu}^{\bar\nu,d}s_{\bar\nu}(\zeta^\sigma)
\end{equation}
with $\lambda=\bar\lambda+\rho$. Setting $t=\zeta^{\frac{k+1}{2}}q^{-\frac{1}{n}}$ the matrix elements \eqref{BVI_formula} are equal to the 3-point genus zero Gromov-Witten invariants of $qH^*(\operatorname{Gr}(k,n))$. In particular, $C_{\bar\lambda\bar\mu}^{\bar\nu,d}=0$ unless $|\bar\lambda|+|\bar\mu|-|\bar\nu|=0\mod n$.
\end{prop}
\begin{proof}
The proof follows along the same lines as in the previous case of symmetric tensors and employing that
\[
t^{-|\bar\lambda|}s_{\bar\lambda}(\zeta^{\mu})=t^{-|\bar\lambda|}\frac{a_{\lambda}(\zeta^\mu)}{a_\rho(\zeta^\mu)}
=\frac{\cS_{\lambda\mu}(t)}{\cS_{\rho\mu}(t)}\;.
\]
In particular, the matrix elements can be rewritten as
\begin{equation}\label{BVI2}
C_{\bar\lambda\bar\mu}^{\bar\nu}(t)=(-1)^{d(k-1)}t^{-nd}C_{\bar\lambda\bar\mu}^{\bar\nu,d}=
t^{|\bar\nu|-|\bar\lambda|-|\bar\mu|}
\sum_{\sigma\in\salc}\frac{a_{\lambda}(\zeta^\sigma)a_{\mu}(\zeta^\sigma)a_{\nu}(\zeta^{-\sigma})}{n^ka_{\rho}(\zeta^\sigma)}
\end{equation}
and one then recognises for the stated value of $t$ the Bertram-Vafa-Intriligator formula for Gromov-Witten invariants; see e.g. \cite{bertram1997quantum} and \cite{rietsch2001quantum} and references therein.
\end{proof}
Note that the last proposition implies that $C_{\bar\lambda\bar\mu}^{\bar\nu}(\zeta^{\frac{k+1}{2}})=C_{\bar\lambda\bar\mu}^{\bar\nu,d}$ is a non-negative integer. 
\begin{coro}\label{cor:GWsymm}
We have the following identities for the matrix elements \eqref{BVI_formula}:
\begin{equation}\label{GWsymm}
C_{\bar\lambda\bar\mu}^{\bar\nu}(t)=C_{\bar\mu\bar\lambda}^{\bar\nu}(t)=C_{\bar\nu^\vee\bar\mu}^{\bar\lambda^\vee}(t)
\qquad\text{ and }\qquad
C_{\emptyset\bar\mu}^{\bar\nu}(t)=\delta_{\mu\nu}\;.
\end{equation}
\end{coro}
\begin{proof}
All of the asserted equalities follow from \eqref{BVI_formula} and the identities \eqref{modSfermii} for the inverse of the modular $\mc{S}$-matrix. In particular, $C_{\bar\lambda\bar\mu\bar\nu}(t):=C_{\bar\lambda\bar\mu}^{\bar\nu^\vee}(t)$ is invariant under permutations of $\bar\lambda,\bar\mu$ and $\bar\nu$. 
\end{proof}

Having identified the matrix elements \eqref{BVI_formula} for $t=\zeta^{\frac{k+1}{2}}$ with the Gromov-Witten invariants $C_{\bar\mu\bar\nu}^{\bar\lambda,d}$, we obtain as a corollary the simplest case of the Satake correspondence for quantum cohomology (c.f. \cite{golyshev2011quantum}): 

\begin{coro}\label{cor:Vminus}
Let $\V^-_k=\bigwedge^kV\otimes\RR$ be the subspace of alternating tensors. Then $\V^-_k/(t-\zeta^{\frac{k+1}{2}})\V^-_k$ together with the product $v_{\bar\lambda}v_{\bar\mu}:=S^*_{\bar\lambda}v_{\bar\mu}$ and trace functional $\veps(v_{\bar\lambda})=\delta_{\bar\lambda,\bar\rho^\vee}$ is a commutative semi-simple Frobenius algebra, which is isomorphic to the specialisation of $qH^*(\op{Gr}(k,n))$ at $q=1$.
\end{coro}

\begin{proof}
That the product $m:\V^-_k\otimes\V^-_k\to\V^-_k$ with $m(v_{\bar\lambda},v_{\bar\mu}):=S^*_{\bar\lambda}v_{\bar\mu}$ is commutative is a direct consequence of \eqref{BVI_formula}. Associativity then follows by the analogous argument used in the case of symmetric tensors; see the proof of Theorem \ref{thm:symmFrob}. The invariance of the bilinear form $\eta(v_{\bar\lambda},v_{\bar\mu}):=\veps(v_{\bar\lambda}v_{\bar\mu})=\delta_{\bar\lambda^\vee\bar\mu}$ is proved using the identities in \eqref{GWsymm}. Since the map $\bar\lambda\mapsto\bar\lambda^\vee$ is an involution one easily verifies that $\eta$ is non-degenerate. The remaining maps corresponding to the 2-cobordisms in \eqref{2cob} are then constructed from $\eta$ and the multiplication map $m$ as in the case of symmetric tensors. 

For semi-simplicity we again employ the modular $\cS$-matrix which encodes the idempotents and that the latter span $\V^-_k$ since $\mc{S}$ is invertible.  Proposition \ref{prop:BVI} then entails that the map which sends $v_{\bar\lambda}$ onto a Schubert class gives an algebra isomorphism $\V^-_k/(t-\zeta^{\frac{k+1}{2}})\V^-_k\cong qH^*(\op{Gr}(k,n))/\langle q-1\rangle$.
\end{proof}

For completeness we mention here some differences with the case of symmetric tensors. Firstly recall that $qH^*(\op{Gr}(k,n))$ has the following known presentation (see e.g. \cite{abrams2000quantum}) in terms the Landau-Ginzburg or super-potential
\[
W_q(x_1,\ldots,x_k)=\frac{p_{n+1}(x_1,\ldots,x_k)}{n+1}+(-1)^k q p_1(x_1,\ldots,x_k)\;,
\]
where $p_r(x_1,\ldots,x_k)=\sum_{i=1}^kx_i^r$ are the power sums,
\[
qH^*(\op{Gr}(k,n))\cong \mathbb{C}[q][e_1,\ldots,e_k]/\langle\frac{\partial W_q}{\partial e_1},\ldots,\frac{\partial W_q}{\partial e_k}\rangle\;.
\]
We were unable to find an analogous super-potential for the ideal \eqref{pideal} used in the case of symmetric tensors. However, one readily sees the similarity with the equations \eqref{BAE} which are a direct consequence of \eqref{cyclotomic}.

The other difference lies in the Frobenius structure fixed via the Frobenius trace $\veps$, which is different from the trace functional the algebra would inherit viewed as a Verlinde algebra: there one fixes the value of $\veps$ on the idempotents in terms of the quantum dimensions $\cS_{\lambda\rho}/\cS_{\rho\rho}$. However, in contrast to the case of symmetric tensors, the latter are (in general) neither positive nor real for alternating tensors. Instead the ring $qH^*(\op{Gr}(k,n))$ inherits the intersection form from $H^*(\op{Gr}(k,n))$. Thus, even for $k=1$ both Frobenius algebras are different due to the different bilinear forms, $\V^+_1\cong\V_1(\sh_n)$, the Verlinde algebra of the $\sh_n$-WZW model, and $\V^-_1\cong qH^*(\mb{P}^{n-1})/(q-1)$, the specialised quantum cohomology of projective space. 

\section{Cylindric symmetric functions}
In the previous section we have computed the matrix elements of $M_\mu$ and $S_\mu$ (as well as those of their adjoints) and identified them with the fusion coefficients of 2D TQFTs (structure constants of symmetric Frobenius algebras), which for particular values of the loop parameter, turn out to be non-negative integers. It is therefore natural to ask for a combinatorial description: we show that the fusion coefficients give the expansion coefficients of so-called cylindric symmetric functions by taking matrix elements of the following Cauchy identities in $\End_{\R}\bV_k$.
\begin{lemma} 
Recall the definitions \eqref{E} and \eqref{H}. The following equalities hold in $\End_{\R}\bV_k$,
\begin{equation}\label{Cauchy1}
\prod_{j\geq 1}H(u_j)=
\prod_{i=1}^k\prod_{j\geq 1}(1-X_i u_j)^{-1}=\sum_{\lambda\in\cP_k^+}M_\lambda h_{\lambda}(u)=\sum_{\lambda\in\cP_k^+}S_\lambda s_{\lambda}(u),
\end{equation}
and
\begin{equation}\label{Cauchy2}
\prod_{j\geq 1}E(u_j)=
\prod_{i=1}^k\prod_{j\geq1}(1+X_i u_j)=\sum_{\lambda\in\cP_k^+}M_\lambda e_{\lambda}(u)=\sum_{\lambda\in\cP_k^+}S_\lambda s_{\lambda'}(u),
\end{equation}
where the $X_i$ are the matrices  \eqref{X} defined in the presentation of $\cH_k$ discussed in the previous section and the $u_j$ are some commuting indeterminates. The analogous identities hold for $H^*(u_j)$, $E^*(u_j)$ and $M_\lambda^*$, $S_{\lambda}^*$ with the $X_i$ replaced by $X_i^{-1}$.
\end{lemma}
\begin{proof}
Employing the Hopf algebra homomorphisms $\Phi_{\pm}$ from Lemma \ref{lem:Lambda2U} the Cauchy identities are a direct consequence of the corresponding identities in $\Lambda$. The restriction of the sums to partitions $\lambda\in\cP^+_k$ follows from observing that $M_\lambda$ and $S_\lambda$ must identically vanish on $\bV_k$ for $\ell(\lambda)>k$. Again this follows from the familiar relations in the ring of symmetric functions: recall the projection $\Lambda\twoheadrightarrow\Lambda_k=\mb{C}[x_1,\ldots,x_k]^{\rS_k}$ by setting $x_i=0$ for $i>k$. This introduces linear dependencies among the power sums allowing one to express $p_{\lambda}$ with $\ell(\lambda)>k$ in terms of $\{p_\mu\}_{\mu\in\cP_k^+}$ as the latter form a basis of $\Lambda_k$; see \cite{macdonald1998symmetric}. As a result the projections of monomial symmetric and Schur functions are identically zero for $\ell(\lambda)>k$.  Employing Lemma \ref{lem:P2X} the same linear dependencies among the $\{P_\lambda\}_{\lambda\in\cP^+}$ then ensure that $M_\lambda$ and $S_\lambda$ are both the zero map for $\ell(\lambda)>k$. The analogous identities for the adjoint operators are obtained by the same arguments using $\Phi_-$ and Lemma \ref{lem:adjoint}.
\end{proof}

Note that the $X_i$ in \eqref{Cauchy1}, \eqref{Cauchy2} all (trivially) commute. Hence, each matrix element must yield a symmetric function in the $u_j$ and the latter are the mentioned cylindric symmetric functions. Their name originates from their definition as sums over so-called cylindric tableaux; see e.g. \cite{gessel1997cylindric,postnikov2005affine,mcnamara2006cylindric}. The latter are fillings of skew shapes (lattice paths) on the cylinder $\mf{C}_{k,n}=\mb{Z}^2/(-k,n)\mb{Z}$. The aspect in our work we like to stress here is the definition of new families of cylindric functions as well as the linking all of them to the same combinatorial action of the extended affine symmetric group in terms of ``infinite permutations''.

\subsection{Infinite permutations and the extended affine symmetric group}
Recall the realisation of the affine symmetric group $\tilde{\rS}_k=\cQ_k\rtimes \rS_k$ in terms of bijections $\tilde w:\mathbb{Z}\to\mathbb{Z}$; see \cite{lusztig1983some,bjorner1996affine,eriksson1998affine,shi2006kazhdan}.  Here we state a generalisation of this presentation for the extended affine symmetric group $\hat \rS_k=\cP_k\rtimes \rS_k$.
\begin{prop}
The extended affine symmetric group $\hat \rS_k$ can be realised as the set of bijections $\hat w:\mathbb{Z}\to\mathbb{Z}$ subject to the two conditions
\begin{equation}\label{extS}
\hat w(m+k)=\hat w(m)+k,\;\forall m\in\mathbb{Z}\quad\text{and}\quad\sum_{m=1}^k {\hat w(m)} =\binom{k}{2}\mod k\;.
\end{equation}
The group multiplication is given via composition. The subset of bijections $\tilde w$ for which $\sum_{m=1}^k {\tilde w(m)} =\binom{k}{2}$ gives the affine symmetric group $\tilde \rS_k\subset\hat \rS_k$. 
\end{prop}
\begin{proof}
For $\tilde \rS_k$ this is the known presentation for the affine symmetric group from \cite{lusztig1983some}. In particular, the simple Weyl reflections $\{ \sigma_0, \sigma_1, \ldots,\sigma_{k-1} \}$ are the maps $\mathbb{Z}\to\mathbb{Z}$ defined via
\begin{equation}
\sigma_i(m) = 
\begin{cases}
m+1, & m=i \text{ mod } k \\
m-1, & m=i+1 \text{ mod } k \\
m, & \text{otherwise .}
\end{cases}
\end{equation}
Introduce the shift operator $\tau:\mathbb{Z}\to\mathbb{Z}$ by $m\mapsto\tau(m)=m-1$. Then one has the identities %
$
\tau\circ\sigma_{i+1}=\sigma_i\circ\tau
$, where indices are understood modulo $k$. %
One easily verifies that any $\hat w= \tilde w\circ\tau^d$ with $\tilde w\in\tilde \rS_k$ and $d\in\mathbb{Z}$ obeys the stated conditions. Likewise any such map can be written in the form $\hat w= \tilde w\circ\tau^d$. Thus, the group generated by $\langle\tau,\sigma_0, \sigma_1, \ldots,\sigma_{k-1} \rangle$ is the extended affine symmetric group $\hat \rS_k$.
\end{proof}
Our main interest in this realisation of $\hat \rS_k$ is that it naturally leads to the consideration of cylindric loops.

\subsection{Cylindric skew shapes and reverse plane partitions}
Fix $n\in\mathbb{N}$. We are now generalising the notion of the weight lattice in order to define a level-$n$ action of the extended affine symmetric group. Let $\cP_{k,n}$ denote the set of functions $\lambda:\mathbb{Z}\to\mathbb{Z}$ subject to the constraint $\lambda_{i+k}=\lambda_i-n$ for all $i\in\mathbb{Z}$. 
\begin{lemma}\label{lem:Saction}
The map $\cP_{k,n}\times\hat \rS_k\to \cP_{k,n}$ with $(\lambda,\hat w)\mapsto \lambda\circ\hat w$, where the `infinite permutation'  $\hat w:\mb{Z}\to\mb{Z}$ is a bijection satisfying \eqref{extS}, defines a right action.  
\end{lemma}
\begin{proof}
A straightforward computation.
\end{proof}
One can convince oneself that the above is the familiar level-$n$ action of $\hat \rS_k$ on the weight lattice $\cP_k$ by observing that each $\lambda\in\cP_{k,n}$ is completely fixed by its values $(\lambda_1,\ldots,\lambda_k)$ on the set $[k]$. Employing this identification between weights in $\lambda\in\cP_k$ and their associated maps in $\lambda\in\cP_{k,n}$ (which we shall denote by the same symbol), the  set of partitions \eqref{alcove} defined earlier, constitutes an `alcove': a fundamental domain with respect to the above level-$n$ action of $\hat \rS_k$ on $\cP_{k,n}$. That is, for any $\lambda\in\cP_{k,n}$ the orbit $\lambda\hat \rS_k$ intersects $\alc$ in a unique point. 

Note that when employing this identification of weights and maps one needs to be careful not to identify the sum $\lambda+\mu$ of two weights $\lambda,\mu\in\cP_k$ with the usual addition of maps, where $\lambda+\mu:\mb{Z}\to\mb{Z}$ is defined as $(\lambda+\mu)_i=\lambda_i+\mu_i$.
  
Given $\lambda\in\alc$ and $d\in\mathbb{Z}$ denote by $\lambda[d]$ the (doubly) infinite sequence
 \[
 \lambda[d]=( \dots,\lambda[d]_{-1},\lambda[d]_0,\lambda[d]_1,\dots)=
 ( \dots,\lambda_{-d-1},\lambda_{-d},\lambda_{1-d},\dots)\,,
 \]
that is, the image $\lambda\circ\tau^d(\mathbb{Z})$ of the map $\lambda\circ\tau^d:\mathbb{Z}\to\mathbb{Z}$. This sequence defines a lattice path $\{(i,\lambda[d]_i)\}_{i\in\mb{Z}}\subset\mathbb{Z}\times\mathbb{Z}$ which projects onto the cylinder $\mf{C}_{k,n}=\mb{Z}^2/(-k,n)\mb{Z}$ and is therefore called a {\em cylindric loop}. 

While we have adopted here the notation from \cite{postnikov2005affine}, see also \cite{mcnamara2006cylindric}, our definition of cylindric loops is {\em different} from the one used in these latter works. In {\em loc. cit.} $\lambda[d]$ is obtained by shifting $\lambda[0]$ in the direction of the lattice vector $(1,1)$ in $\mb{Z}^2$, while we shift here by the lattice vector $(1,0)$ instead. We will connect with the cylindric loops from \cite{postnikov2005affine,mcnamara2006cylindric} below when discussing the shifted level-$n$ action \eqref{shifted_action} of $\hat \rS_k$.

A {\em cylindric skew diagram} or {\em cylindric shape} is defined as the number of lattice points between two cylindric loops: let  $\lambda,\mu \in \alc$ be such that $\mu_i \le \lambda_{i-d}=(\lambda\circ\tau^d)_i$ for all $i \in \mathbb{Z}$, then we write
$\mu[0] \le \lambda[d]$ and say that the set 
 \begin{equation}\label{cylshape}
 \lambda / d / \mu = \{ (i,j) \in \mathbb{Z}^2 ~|~ \mu[0]_i< j \le \lambda[d]_i  \}
 \end{equation}
 is a cylindric skew diagram (or shape) of degree $d$.
\begin{defi}
A {\em cylindric reverse plane partition} (CRPP) of shape $\Theta=\lambda/d/\mu$ is a map $\hat\pi:\Theta \to \mathbb{N}$ such that for any $(i,j) \in \Theta$ one has  $\hat\pi(i,j) = \hat\pi(i+k,j-n)$ together with
 \[
  \hat\pi(i,j) \le \hat\pi(i+1,j)\quad\text{ and }\quad\hat\pi(i,j) \le \hat\pi(i,j+1),
\]
provided $(i+1,j),(i,j+1) \in \Theta$. In other words, the entries in the squares between the cylindric loops $\mu[0]$ and $\lambda[d]$ are non-decreasing from left to right in rows and down columns.
\end{defi}
Alternatively, $\hat\pi$ can be defined as a sequence of cylindric loops 
\begin{equation} \label{SequenceCylLoops}
( \lambda^{(0)}[0]=\mu[0], \lambda^{(1)}[d_1], \dots, \lambda^{(l)}[d_l]=\lambda[d] )
\end{equation}
with $\lambda^{(i)} \in \alc$ and $ d_i - d_{i-1} \ge 0$ such that 
$\hat\pi^{-1}(i)=\lambda^{(i)} / (d_i-d_{i-1}) / \lambda^{(i-1)}$ is a cylindric skew diagram; see Figure \ref{fig:CRPPexamples} for examples when $n=4$ and $k=3$. %
The weight of $\hat\pi$ is the vector $\op{wt}(\hat\pi)=(\op{wt}_1(\hat\pi),\ldots,\op{wt}_l(\hat\pi))$ where $\op{wt}_i(\hat\pi)$ is the number of lattice points $(a,b)\in \lambda^{(i)} / (d_i-d_{i-1}) / \lambda^{(i-1)}$ with $1\leq a\leq k$. 

\subsection{Cylindric complete symmetric functions} \label{CylCom}
Using CRPPs we now introduce a new family of symmetric functions which can be viewed as generalisation of functions which arise from the coproduct of complete symmetric functions; see Appendix A for details. In the last part of this section we then show that they form a positive subcoalgebra of $\Lambda$ whose structure constants are the fusion coefficients from \eqref{Lfusion} and \eqref{Verlinde_bos}.

Given $\mu\in\alc$ note that $\mu_1-\mu_k<n$ and, hence, its stabiliser group $\rS_\mu\subset \rS_k\subset\hat \rS_k$. Define $\tilde \rS^{\mu}$ 
as the minimal length representatives of the cosets $\rS_\mu\backslash\tilde \rS_k$. The following is a generalisation of the set \eqref{setTheta} in Appendix A to the cylindric case: for $\lambda,\mu \in \alc$ and $d \in \mathbb{Z}_{\ge 0}$ define the set
\begin{equation} \label{SetAffinePermTheta}
  \{ \tilde{w}\in \tilde{\rS}^\mu~|~\mu\circ\tilde{w} \le\lambda\circ\tau^d \}
\end{equation}
and denote by $\theta_{\lambda/d/\mu}$ its cardinality.
\begin{lemma} \label{lem:ExpressSetTheta}
The set \eqref{SetAffinePermTheta} is non-empty if and only if $\lambda/d/\mu$ is a valid cylindric skew shape, i.e. if $\mu\leq\lambda\circ\tau^d$. In the latter case we have the following expression for its cardinality,
\begin{equation} \label{CylindricTheta}
 \theta_{\lambda/d/\mu}=\prod_{i=1}^n { {(\lambda\circ\tau^d)'_i-\mu'_{i+1}} \choose {\mu'_i-\mu'_{i+1}} } - 
  \prod_{i=1}^n { {(\lambda\circ\tau^{d-1})'_i-\mu'_{i+1}} \choose {\mu'_i-\mu'_{i+1}} }\;,
\end{equation}
where the binomial coefficients are understood to vanish whenever one of their arguments is negative.
\end{lemma}

\begin{proof}
First note that if $d=0$ then we must have $\tilde w\in \rS^\mu$, because $\lambda,\mu\in\alc$. In this case we recover \eqref{SetTheta} and Lemma \ref{lem:BinomialTheta} proved in Appendix A.   

Assume now that $d>0$. Restricting the maps $\mu$ and $\lambda\circ\tau^d$ to the set $[k]$ we recover the corresponding weights in $\cP^+_k$. By a similar argument as in the case $d=0$, however with more involved steps due to the level-$n$ action of the affine symmetric group, one then arrives at the first statement, i.e. that \eqref{SetAffinePermTheta} is non-empty if and only if $\mu_i\leq (\lambda\circ\tau^d)_i$ for $i\in[k]$. This inequality is then extended to all $i\in\mb{Z}$ using that $\mu_{i+k}=\mu_i-n$ and $(\lambda\circ\tau^d)_{i+k}=(\lambda\circ\tau^d)_i-n$ which implies that $\lambda/d/\mu$ is a valid cylindric skew shape. Since $\mu\leq (\lambda\circ\tau^d)$ if and only if $\mu'\leq (\lambda\circ\tau^d)'$ the right hand side in \eqref{CylindricTheta} is zero if $\lambda/d/\mu$ is not a cylindric skew shape. 

Thus, we now assume $\mu\leq\lambda\circ\tau^d$ for the remainder of the proof. To compute the cardinality we will rewrite the set \eqref{SetAffinePermTheta} such that it can be expressed in terms of non-cylindric weights for $\rS_{k+d}$ and then apply again the result from Appendix A with $k$ replaced by $k+d$. 

Each element in $\tilde{\rS}^\mu$ can be expressed as $w \circ x^{-\alpha}$ with $w \in \rS^\mu$ and $\alpha \in \cQ_k$ (that is, $|\alpha|=\sum_i\alpha_i=0$). Thus, the set \eqref{SetAffinePermTheta} can be rewritten as
\[
  \big\{(w,\alpha) \in \rS^\mu \times \cQ_k ~|~ 
  \mu \circ w \circ x^{-\alpha} \circ \tau^{-d} \le \lambda  \big\} \;.
 \]
Since $\tau=x_k \circ \sigma_{k-1} \circ \cdots \circ  \sigma_1$ it follows that $\tau^{-d}=(\sigma_{k-1} \circ \cdots \circ \sigma_1)^{-d} \circ x^{-\beta}$ for some $\beta \in \cP_k$ with $|\beta|=d$. Noting further that 
for any $w' \in \rS_k$  we have that $ x^{-\alpha} \circ w' = w' \circ x^{-\alpha \circ w'}$, we can conclude that $\beta'=-(\alpha \circ w'+\beta)$
ranges over all the elements in $\cP_k$ with $|\beta'|=d$ if $\alpha$ ranges over all the elements in $\cQ_k$. Thus, we arrive at the alternative expression 
 \[
  \big\{(w,\beta') \in \rS^\mu \times \cP_k ~|~
  \mu \circ w \circ (\sigma_{k-1} \circ \cdots \circ \sigma_1)^{-d} \circ x^{-\beta'} \le \lambda,\,|\beta'|=d \big\} \;.
 \] 
Each element $w \circ (\sigma_{k-1} \circ \cdots \circ \sigma_1)^{-d}$
has a unique decomposition $w_\mu \circ w^\mu$, with $w_\mu\in \rS_\mu$ and $w^\mu \in \rS^\mu$, such that
 distinct $w$ correspond to distinct $w^\mu$. Therefore, the set \eqref{SetAffinePermTheta} is in bijection with the set 
 \begin{equation} \label{SubSetAffineSmu}
\mb{A}=\big\{(w,\beta) \in \rS^\mu \times \cP_k~|~ 
  \mu \circ w \circ x^{-\beta} \le \lambda, \, |\beta|=d\big\} \;,
 \end{equation}
where $\beta\in\cP_k$ can only have non-negative parts $\beta_i\ge 0$ as $\lambda,\mu \in \alc$.

We now express the cardinality of the set \eqref{SubSetAffineSmu} for $\hat \rS_k$ in terms of the cardinality of the (non-cylindric) set \eqref{setTheta} for $\rS_{k+d}$. Define the two weights in $\cP^+_{k+d}$ setting $\Lambda^{(d)}=(n,n,\dots,n,\lambda_1,\dots,\lambda_k)$ and 
 $\mu^{(d)}=(\mu_1,\dots,\mu_k,0,\dots,0)$.
 We now construct a bijection between \eqref{SubSetAffineSmu} and the set 
 \begin{equation} \label{SetBarMuLambdaD}
 \mb{B}= \{ \bar{w} \in \rS^{\mu^{(d)}} ~|~ \mu^{(d)} \circ \bar{w} \le \Lambda^{(d)},\;
  (\mu^{(d)} \circ \bar{w})_1 > 0 \} \;.
 \end{equation}
Fix $(w,\beta)\in\mb{A}$ and let $J(\beta)=\{j_1,\dots,j_l\} \subset [k]$ be the set of indices for which $\beta_{j_i}> 0$, $i=1,\dots,l$. 
Denote by $\bar J(\beta)=[k]\backslash J(\beta)$ its complement. Define a weight $\gamma=\gamma(w,\beta,d,\mu)\in\cP_{k+d}$ whose parts $\gamma_j$ for $1\le j\le d$ are fixed by the vector
\[
 \big((\mu \circ w)_{j_1},\underbrace{0,\dots,0}_{\beta_{j_1}-1},(\mu \circ w)_{j_2},\underbrace{0,\dots,0}_{\beta_{j_2}-1},
  \dots,(\mu \circ w)_{j_l},\underbrace{0,\dots,0}_{\beta_{j_l}-1} \big) \;.
 \]
 and for $1\le j\le k$ we set 
\[
\gamma_{j+d}=\left\{
\begin{array}{ll}
(\mu\circ w)_{j}, & j\in\bar J(\beta)\\
0,& \text{ else}
\end{array}\right.
\,.
\]
See Figure \ref{fig:proof} for an illustration. Define $\bar w(w,\beta)\in \rS^{\mu^{(d)}}\subset \rS_{k+d}$ to be the unique permutation such that $\gamma=\mu^{(d)}\circ\bar w$. By construction, it follows that $\gamma\le\Lambda^{(d)}$ and $\gamma_1>0$. Hence, $\bar w(w,\beta)\in\mb{B}$. 

Conversely, given $\bar w\in\mb{B}$, define $\gamma=\gamma(\bar w)=\mu^{(d)}\circ\bar w$. Then the parts of $\gamma$ fix the weight $\beta(\bar w)\in\cP_k$ in $\mb{A}$ by reversing the above construction. In particular, the positions of nonzero parts $\gamma_j$ with $d+1\le j\le d+k$ fix the set $\bar J(\beta)$. From $\bar J(\beta)$ and its complement $J(\beta)$ in $[k]$ one constructs a vector $\bar\gamma\in\cP_k$ by setting $\bar\gamma_j=\gamma_{j+d}$ if $j\in\bar J(\beta)$ and if $j=j_i\in J(\beta)$ then let $\bar \gamma_j$ be the $i$th nonzero part among the first $d$ parts of $\gamma$. Define $w=w(\bar w)\in \rS^{\mu}\subset \rS_k$ via $\bar\gamma=\mu\circ w$. It then follows again by construction that $(w(\bar w),\beta(\bar w))\in\mb{A}$.

By distinguishing the cases $(\mu^{(d)} \circ w')_1>0$ and $(\mu^{(d)} \circ w')_1=0$  with $w'\in \rS^{\mu^{(d)}}\subset \rS_{k+d}$, the cardinality of the set \eqref{SetBarMuLambdaD} can be written as the difference of the cardinalities of the sets  
$\{ w' \in \rS^{\mu^{(d)}}~|~ \mu^{(d)} \circ w' \leq \Lambda^{(d)} \}\subset \rS_{k+d} $
and $\{ w'' \in \rS^{\mu^{(d-1)}}~|~ \mu^{(d-1)} \circ w'' \leq \Lambda^{(d-1)} \} \subset \rS_{k+d-1}$. Namely, suppose $(\mu^{(d)}\circ w')_1=0$, then we may assume $w'(1)=k+d$, because otherwise we simply apply a permutation $w'''\in \rS_{\mu^{(d)}}$ such that the assumption holds (recall that the last $d$ parts of $\mu^{(d)}$ are all zero by definition). Define $w''\in \rS_{k+d-1}$ by setting $w''(i)=w'(i+1)$ for $i=1,\ldots,k+d-1$. Thus, using Lemma \ref{lem:BinomialTheta} and \eqref{setTheta} from the appendix we arrive at
\[
\theta_{\lambda/d/\mu}=\theta_{\Lambda^{(d)}/\mu^{(d)}}-\theta_{\Lambda^{(d-1)}/\mu^{(d-1)}} \;,
\]
and since $(\Lambda^{(d)})'_i=\lambda'_i+d=(\lambda \circ \tau^d)'_i$ equation \eqref{CylindricTheta} follows.
\end{proof}
Similar to the non-cylindric case treated in Lemma \ref{lem:SkewComplete} in the appendix, we employ \eqref{CylindricTheta} 
to define weighted sums over cylindric reverse plane partitions:
given a CRPP $\hat\pi$ set 
\begin{equation}\label{thetapihat}
\theta_{\hat \pi}=\prod_{i\ge 1} \theta_{\lambda^{(i)}/(d_i-d_{i-1})/\lambda^{(i-1)}},
\end{equation}
where the cylindric skew diagram $\lambda^{(i)}/(d_i-d_{i-1})/\lambda^{(i-1)}$ 
is the pre-image $\hat\pi^{-1}(i)$, and we denote by $u^{\hat\pi}$ the monomial 
$u_1^{\op{wt}_1(\hat\pi)}u_2^{\op{wt}_2(\hat\pi)}\cdots$ in some commuting indeterminates $u_i$. If $d=0$ we recover the definition \eqref{theta} from Appendix A, i.e. $\theta_{\lambda/0/\mu}=\theta_{\lambda/\mu}$.

\begin{figure}
\centering
\includegraphics[width=.9\textwidth]{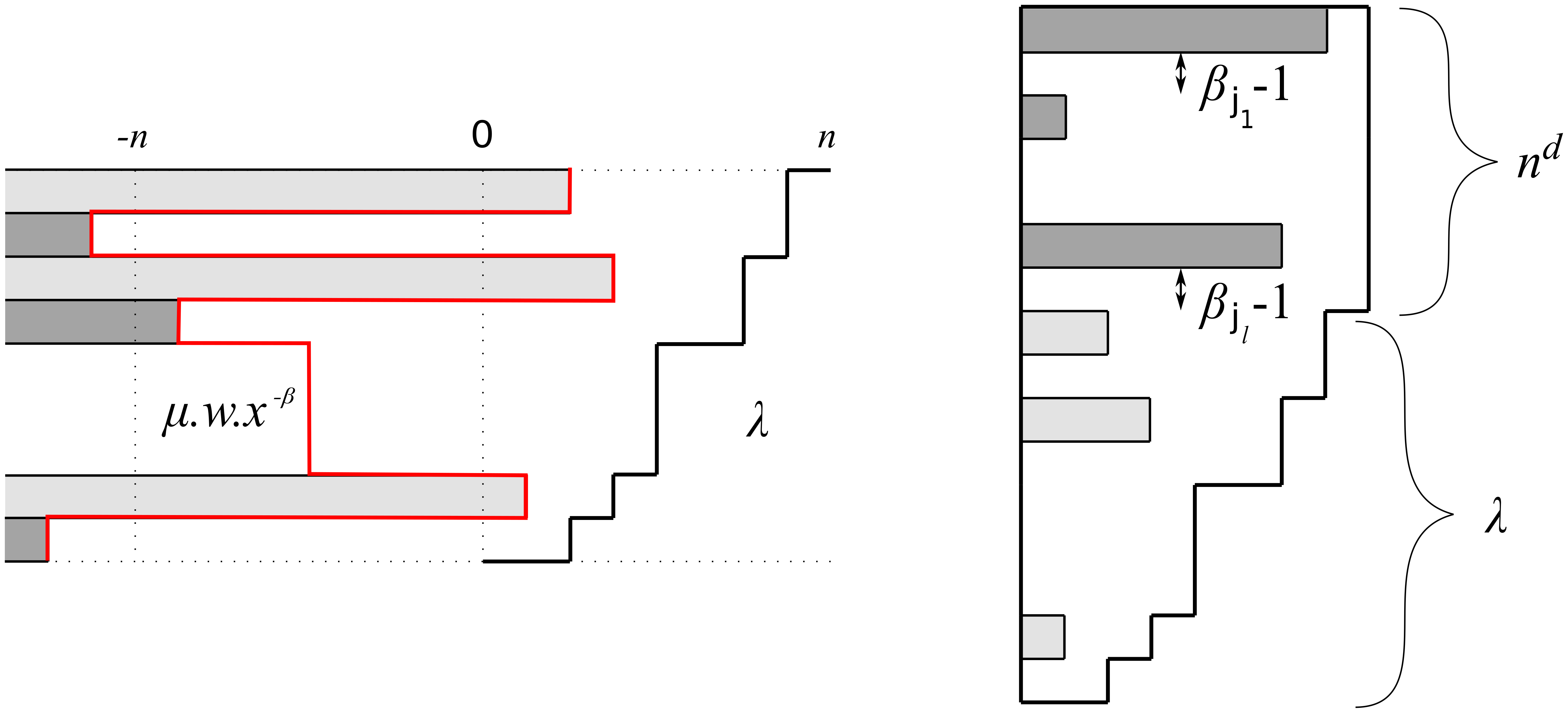}  
\caption{A graphical depiction of the construction of the weight vector $\gamma\in\cP_{k+d}$ (shown on the right) in the proof of Lemma \ref{lem:ExpressSetTheta}. }
\label{fig:proof}
\end{figure}

\begin{defi}
For $\lambda,\mu \in \alc$ and $d \in \mathbb{Z}_{\ge 0}$, define the {\em cylindric complete symmetric function} $h_{\lambda/d/\mu}$ as the weighted sum
\begin{equation}\label{cylh}
h_{\lambda/d/\mu}(u)=\sum_{\hat\pi} \theta_{\hat\pi} u^{\hat\pi}
\end{equation}
over all cylindric reverse plane partitions $\hat\pi$ of shape $\lambda/d/\mu$.
\end{defi}
Note that when setting $d=0$ we recover the (non-cylindric) skew complete cylindric function discussed in Appendix A, that is $h_{\lambda/0/\mu}=h_{\lambda/\mu}$. We now prove for $d>0$ that $h_{\lambda/d/\mu}$ is a symmetric function by expanding it into the bases of monomial and complete symmetric functions. Proceeding in close analogy to the non-cylindric case $d=0$ discussed in Appendix A, we first link the expansion coefficients to product identities in the quotient ring $\Lambda_k\otimes\RR/\mc{J}_n$ from Theorem \ref{thm:bosquotient}. As the latter is isomorphic to $\V^+_k$ we shall use the same notation for both of them in what follows.

Let $\lambda,\mu \in \alc$, $\nu\in \cP^+$, $d\in\mb{Z}_{\ge 0}$ and define
\begin{equation} \label{thetaLambdaMuNu}
 \theta_{\lambda / d / \mu}(\nu)= \sum_{\hat \pi} \theta_{\hat \pi}\;,
\end{equation}
where the sum is restricted to CRPP $\hat \pi$ of shape $\lambda/d/\mu$ and weight $\nu$. 
\begin{lemma}\label{lem:mMuhNuQuotient}
The following product rule holds in $\V^+_k$
 \begin{equation} \label{hm2m}
  m_\mu(x_1^{-1},\ldots,x^{-1}_k) h_\nu (x_1^{-1},\ldots,x^{-1}_k)   =  \sum_{\lambda \in \alc} z^{-d} \theta_{\lambda / d / \mu}(\nu) m_\lambda(x_1^{-1},\ldots,x^{-1}_k) \;.
 \end{equation}
where $d = \frac{|\mu|+|\nu|-|\lambda|}{n}$ in the sum on the right hand side. In particular, $\theta_{\lambda / d / \mu}(\nu)$ is nonzero only if $d n+|\lambda|=|\mu|+|\nu|$.
\end{lemma}
\begin{proof}
It suffices to show that in $\V^+_k$
  \begin{equation}
   m_\mu(x_1^{-1},\ldots,x^{-1}_k)  h_r(x_1^{-1},\ldots,x^{-1}_k) = \sum_{\lambda \in \alc} z^{-d} \theta_{\lambda / d / \mu} m_\lambda(x_1^{-1},\ldots,x^{-1}_k),
 \end{equation}
 where the sum runs over all $\lambda \in \alc$ such that $\mu \le \lambda \circ \tau^d$ with $d=\frac{r+|\mu|-|\lambda|}{n}\in \mathbb{Z}_{\ge 0}$.  The general case \eqref{hm2m} then follows by repeatedly applying the latter expansion.
 
First note that the coefficient of $m_\lambda$ in $m_\mu h_r$ must equal the coefficient of  the monomial term $x^{-\lambda}$ in the same product. Since $m_\mu$ and $h_r$ are polynomials of degree $\mu$ and $r$, respectively, and $x_i^{-n}=z^{-1}$, it follows that $r+|\mu|-|\lambda|=0\mod n$. Hence, the term $x^{-\lambda}$ with $\lambda\in\alc$ occurs in the product expansion if and only if $\mu \circ w \leq \lambda \circ x^{-\beta}$ with $\beta \in \cP_k$, $\beta_i\geq 0$, and $|\lambda|=r+|\mu|-n|\beta|$. (N.B. the symbol $x$ appears here twice, once in the role as variable and another time as translation acting on a weight.) Thus, for any such $\lambda\in\alc$ the coefficient of $x^{-\lambda}$ in $ m_\mu(x_1^{-1},\ldots,x^{-1}_k) h_r(x_1^{-1},\ldots,x^{-1}_k)$ equals $z^{-|\beta|}$ times the  cardinality of the set 
 \begin{equation} \label{wAlpha}
 \Big\{ (w,\beta) \in \rS^\mu \times \cP_k~|~  \beta_i\geq 0,\,|\beta|=d,\, \mu \circ w \le \lambda \circ x^\beta \Big\} \;,
 \end{equation}
 where $nd=r+|\mu|-|\lambda|$. Comparing with \eqref{SubSetAffineSmu} we see that the coefficient is equal to $z^{-d} \theta_{\lambda/d/\mu}$ with $d=\frac{r+|\mu|-|\lambda|}{n}$.
\end{proof}

 Note that the last lemma implies that $\theta_{\lambda/d/\mu}(\nu)=\theta_{\lambda/d/\mu}(\beta)$ for $\beta \sim \nu$,
  where $\theta_{\lambda/d/\mu}(\beta)$ with $\beta$ a composition is defined analogous to \eqref{thetaLambdaMuNu}. Thus, we have as immediate corollary:
\begin{coro} 
The function \eqref{cylh} has the expansion
 \begin{equation}
  h_{\lambda/d/\mu} = \sum_{\nu\in\cP^+} \theta_{\lambda/d/\mu}(\nu) m_\nu \label{CylhM}
 \end{equation}
 in the ring of symmetric functions $\Lambda$. 
\end{coro}
Note that if we set $\nu=(r,0,0,\ldots)$ with $r=dn+|\lambda|-|\mu|$ then $\theta_{\lambda/d/\mu}(\nu)=1$, as long as $\lambda/d/\mu$ is a valid cylindric shape, since then there exists precisely one CRPP of that weight, namely the cylindric shape $\lambda/d/\mu$ itself. Hence, $h_{\lambda/d/\mu}$ is not identically zero provided $\lambda/d/\mu$ is a valid cylindric skew diagram. 

Similar to the product expansion \eqref{hm2m} we also wish to express the fusion coefficients from \eqref{Lfusion}, \eqref{mfusion} which appear in the expansion of the product $m_\mu m_\nu$ in $\V^+_k$ in terms of the cardinalities of sets involving affine permutations. To this end, we now extend the definition of the fusion coefficients to weights outside the alcove \eqref{alcove}.

For $\lambda,\mu,\nu \in \cP^+_k$ define $\bar N_{\mu \nu}^\lambda$ as the cardinality of the set 
\begin{equation} \label{Nlambdamunu}
\bigg\{(w,w')\in \rS^\mu\times \rS^\nu~|~\mu\circ w+\nu\circ w'=\lambda\circ x^\alpha
\text{ for some }\alpha\in\cP_k\bigg\} \;.
\end{equation}
Note that any such weight $\alpha\in\cP_k$ appearing in the above definition does have to satisfy $n |\alpha|=|\mu|+|\nu|-|\lambda|$.
 \begin{lemma}
  For $\mu,\nu \in \cP^+_k$ we have the following product expansion in $\V^+_k$ 
  \begin{equation} \label{Nmunulambda}
   m_\mu(x_1^{-1},\ldots,x^{-1}_k) m_\nu(x_1^{-1},\ldots,x^{-1}_k) = \sum_{\lambda \in \alc} z^{\frac{|\lambda|-|\mu|-|\nu|}{n}} \bar N_{\mu \nu}^\lambda m_\lambda(x_1^{-1},\ldots,x^{-1}_k)\,.
  \end{equation}
So, in particular, $\bar N^{\lambda}_{\mu\nu}=N^{\lambda}_{\mu\nu}$ for $\lambda,\mu,\nu\in\alc$. Moreover, we have the following `reduction formula' for monomial symmetric functions in $\V^+_k$,
\begin{equation}\label{mreduce}
 m_{\check \lambda}(x_1^{-1},\ldots,x^{-1}_k)= \frac{|\rS_{\check \lambda}|}{|\rS_\lambda|} m_\lambda(x_1^{-1},\ldots,x^{-1}_k) z^{\frac{-|\check \lambda|+|\lambda|}{n}},
\end{equation}
where $\check\lambda$ is the unique intersection point of the orbit $\lambda\hat \rS_k$ with $\alc$.
 \end{lemma}
 Because we have equality between the fusion coefficients in \eqref{Lfusion} and the coefficients $\bar N_{\mu \nu}^\lambda$ if $\lambda,\mu,\nu\in\alc$, we shall henceforth use the same notation for both and it will be understood that $N_{\mu \nu}^\lambda$ is defined via the cardinality of the set \eqref{Nlambdamunu} whenever one of the weights lies outside the alcove \eqref{alcove}.
 
  \begin{proof}
Since the monomial symmetric function $m_\lambda$ is homogeneous and of degree $|\lambda|$ it follows from $x_i^{-n}-z^{-1}=0$, that we must have $|\mu|+|\nu|-|\lambda|=0\mod n$. Therefore, the monomial $x_1^{-\lambda_1} \cdots x_k^{-\lambda_k}$ can only occur in the product $m_\mu(x_1^{-1},\ldots,x^{-1}_k) m_\nu(x_1^{-1},\ldots,x^{-1}_k)$ provided there exist $ w \in \rS^\mu$ and $w' \in \rS^\nu$ such that  $w(\mu)+w'(\nu)= \lambda \circ x^\alpha$ for some $\alpha \in \cP_k$ satisfying $|\mu|+|\nu|-|\lambda|=n|\alpha|$, which proves the asserted product expansion.

To prove the reduction formula note that in $\Lambda_k$ we have that $m_\lambda(x_1^{-1},\ldots,x^{-1}_k)=\sum_{w \in \rS^\lambda} x^{\lambda \circ w}= \frac{1}{|\rS_\lambda|} \sum_{w \in \rS_k} x^{-\lambda \circ w}$. Since $x_i^{-n}=z^{-1}$ in the quotient we arrive at the stated formula.
 \end{proof}
 Since the set $\{m_\lambda(x_1^{-1},\ldots,x^{-1}_k)\}_{\lambda\in\alc}$ forms a basis of the quotient ring $\V^+_k$ the coefficients $N^{\lambda}_{\mu\nu}$ with $\lambda,\mu,\nu \in \cP^+_k$ must be expressible in terms of the coefficients where $\lambda,\mu,\nu\in\alc$. The following lemma together with the identities from Corollary \ref{cor:symmN} gives an explicit reduction formula.
 
\begin{lemma}\label{lem:Nreduction}
Let $\lambda,\mu,\nu\in\cP^+_k$. Denote by $\check\nu\in\alc$ the unique intersection point of the orbit $\nu \hat \rS_k$ with the alcove \eqref{alcove}. 
Then 
\begin{equation}\label{reducedN}
N_{\mu\nu}^{\lambda}=N_{\mu\check\nu}^{\lambda}
\binom{m_n(\check\nu)}{m_0(\nu),m_{n}(\nu),m_{2n}(\nu),\ldots}
\prod_{i=1}^{n-1}\binom{m_i(\check\nu)}{m_i(\nu),m_{i+n}(\nu),m_{i+2n}(\nu),\ldots}\;,
\end{equation}
where $m_j(\nu)$ and $m_j(\check\nu)$ are the multiplicities of the part $j$ in $\nu$ and $\check\nu$, respectively.
\end{lemma}
\begin{proof}
Recall that for $\nu \in \cP^+_k$ the cardinality of $\rS_\nu$ is given by $|\rS_\nu|= \prod_{i \ge 0} m_i(\nu)!$.  Noting the equalities $m_n(\check \nu)=m_0(\nu)+m_n(\nu)+m_{2n}(\nu)+\dots$ and $m_i(\check \nu)=m_i(\nu)+m_{i+n}(\nu)+\dots$, one applies the definition of multinomial coefficients to arrive at the relation
\[
 |\rS_{\check \nu}|= |\rS_{\nu}| \binom{m_n(\check\nu)}{m_0(\nu),m_{n}(\nu),m_{2n}(\nu),\ldots}
\prod_{i=1}^{n-1}\binom{m_i(\check\nu)}{m_i(\nu),m_{i+n}(\nu),m_{i+2n}(\nu),\ldots} \;.
\]
Applying  equation \eqref{mreduce} in \eqref{Nmunulambda} completes the proof.
\end{proof}

Let $L_{\alpha\beta}$ be the number of $\mb{N}_0$-matrices whose row sums are fixed by the components of the vector $\alpha$ and whose column sums are fixed by the components of the vector $\beta$; see Appendix A. The next lemma is  the generalisation of the first identity in \eqref{skewThetaPsi} to the cylindric case.

\begin{lemma} \label{lem:theta2N}
Let $\lambda,\mu \in \alc$ and $\nu \in \cP^+$. Set $d = \frac{|\mu|+|\nu|-|\lambda|}{n}$, then the following equality holds
 \begin{equation} 
  \theta_{\lambda/d/\mu}(\nu) = \sum_{\sigma \in \cP^+_k} L_{\nu \sigma} N_{\sigma \mu}^\lambda\;. \label{ThetaN}
     \end{equation}
\end{lemma}
\begin{proof}
 Insert the known expansion $h_\nu(x_1^{-1},\ldots,x^{-1}_k)=\sum_{\sigma \in \cP^+_k} L_{\nu \sigma} m_\sigma(x_1^{-1},\ldots,x^{-1}_k)$ (see Appendix A) into 
 the product $m_\mu(x_1^{-1},\ldots,x^{-1}_k) h_\nu(x_1^{-1},\ldots,x^{-1}_k)$ in $\V^+_k$ and compare with \eqref{hm2m}, using the 
 fact that $\{ m_\lambda(x_1^{-1},\ldots,x^{-1}_k) \}_{\lambda \in \alc}$ is a basis of $\V^+_k$. 
 Note that $L_{\nu \sigma}$ is nonzero only if $|\nu|=|\sigma|$,
which implies that on the right hand side of the asserted equation only the coefficients $N_{\sigma \mu}^\lambda$ appear for which $\frac{|\sigma|+|\mu|-|\lambda|}{n}=d$.
\end{proof}

We have now all the results in place to state the main result of this section which connects the cylindric complete symmetric functions with our discussion in the previous sections.
\begin{theorem} \label{thm:ExpansionCylindricH}
Let $\lambda,\mu\in\alc$ and $d\in\mathbb{Z}_{\ge 0}$. Then (i) the symmetric function $h_{\lambda/d/\mu}$ has the expansion
\begin{equation}
h_{\lambda/d/\mu}=\sum_{\nu\in\cP^+_{k}}N_{\mu\nu}^{\lambda} \;h_\nu \label{cylh2h} 
\end{equation}
into the basis $\{h_\nu\}_{\nu\in\cP^+}\subset \Lambda$, where the sum is restricted to those $\nu\in\cP^+_k$ for which $|\nu|=dn+|\lambda|-|\mu|$. (ii) We have the following formal power series expansions in $z$,
\begin{eqnarray}
\langle v^{\lambda^{\vee}},\prod_{j\geq 1}H(u_j)v_{\mu^\vee}\rangle =
\overline{\langle v^\lambda,\prod_{j\geq 1}H^*(u_j)v_\mu\rangle}=
&=&\sum_{d\geq 0}z^{d} h_{\lambda/d/\mu}(u),
\end{eqnarray}
where $\lambda^{\vee}=(1^{m_n(\lambda)}2^{m_{n-1}(\lambda)}\ldots n^{m_1(\lambda)})$, $\bar u_j=u_j$ and $\bar z=z^{-1}$. 
\end{theorem}

\begin{proof} 
We only need to prove (i) as (ii) is then a direct consequence of \eqref{Cauchy1} and Corollary \ref{cor:symmN}. Using \eqref{ThetaN} one sees that $\sum_{\sigma} \theta_{\lambda/d/\mu}(\sigma) L^{-1}_{\nu \sigma}$ equals
$N^\lambda_{\mu \nu}$ if $d = \frac{|\mu|+|\nu|-|\lambda|}{n}$ and $\ell(\nu) \le k$, and $0$ otherwise.
From \eqref{CylhM} we then have
\[
 h_{\lambda / d / \mu} = \sum_{\nu, \sigma} \theta_{\lambda/d/\mu}(\sigma) \langle m_\sigma, m_\nu \rangle h_\nu = 
 \sum_{\nu} \bigg( \sum_{\sigma} \theta_{\lambda/d/\mu}(\sigma) L^{-1}_{\nu \sigma} \bigg) h_\nu
\]
which proves \eqref{cylh2h}. 
\end{proof}

There are several corollaries of the last theorem which are worth exploring. First note that the expansion coefficients in \eqref{cylh2h} from (i) in Theorem \ref{thm:ExpansionCylindricH} do involve $N^{\lambda}_{\mu\nu}$ where $\nu\in\cP^+_k$ might be outside the alcove \eqref{alcove}. While according to the reduction formula \eqref{reducedN} we can express these coefficients in terms of the fusion coefficients $N^{\lambda}_{\mu\check\nu}$ where $\check\nu\in\alc$ there is an alternative expansion of $h_{\lambda/d/\mu}$ into the special set $\{h_{\sigma/d/n^k}\}_{\sigma\in\alc}$ of cylindric complete symmetric functions that only features the original fusion coefficients $N^{\lambda}_{\mu\sigma}$ with $\lambda$, $\mu$ and $\sigma\in\alc$.
\begin{coro}\label{cor:cylh2cylh}
Let $\lambda,\mu\in\alc$ and $d\in\mb{Z}_{\geq 0}$. Then we have the expansion
\begin{equation}\label{cylh2cylh}
h_{\lambda/d/\mu}=\sum_{d'=0}^{d+k}\sum_{\sigma}N_{\mu\sigma}^{\lambda}h_{\sigma/k+d-d'/n^k}\,,
\end{equation}
where the sum runs over all $\sigma\in\alc$ such that $|\sigma|=d'n+|\lambda|-|\mu|$. 
\end{coro}
\begin{proof}
Starting from the second identity in \eqref{Cauchy1} we take matrix elements in the subspace of symmetric tensors to find,
\begin{eqnarray*}
\langle v^{\lambda},\prod_{i\geq 0}H^*(u_i)v_{\mu}\rangle &=& \sum_{d\geq 0}z^{-d} h_{\lambda/d/\mu}(u)\\
&=& \sum_{\nu\in\cP^+_k}\langle v^{\lambda},M^*_{\nu}v_{\mu}\rangle h_{\nu}(u)
= z^k\sum_{\nu\in\cP^+_k}\langle v^{\lambda},M^*_{\nu}M^*_{\mu}v_{n^k}\rangle h_{\nu}(u)\\
&=&\sum_{\sigma\in\alc}\langle v^{\lambda},M^*_{\mu}v_{\sigma}\rangle \,
z^k\sum_{\nu\in\cP^+_k}\langle v^{\sigma},M^*_{\nu}v_{n^k}\rangle h_{\nu}(u)\\
&=&\sum_{\sigma\in\alc}\langle v^{\lambda},M^*_{\mu}v_{\sigma}\rangle \,\sum_{d''\geq 0} z^{k-d''} h_{\sigma/d''/n^k}(u)\,. 
\end{eqnarray*}
In the second line of this computation we applied the product identity $v_{\mu}=z^{k}v_{\mu}v_{n^k}=z^kM^*_{\mu}v_{n^k}$ in $\V^+_k$. Equating the coefficients of the same powers in $z^{-1}$, we obtain the asserted expansion.
\end{proof}

The cylindric functions used in the expansion \eqref{cylh2cylh} are particularly simple. To see this we note that we can re-parametrise the cylindric complete symmetric functions $h_{\lambda/d/\mu}$ in terms of skew shapes $\tilde\lambda/\tilde d/\tilde\mu$ where $\tilde\lambda$, $\tilde\mu$ are the partitions obtained from $\lambda,\mu\in\alc$ by deleting all parts of size $n$ and setting $\tilde d=d+m_n(\lambda)-m_n(\mu)$. The resulting set $\{\tilde\lambda\in\cP^+_k~|~n>\tilde\lambda_1\geq\cdots\geq\tilde\lambda_k\geq 0\}$ of these `reduced' partitions forms an alternative alcove for the $\hat \rS_k$-action on cylindric loops. It is not difficult to verify that the skew shapes $\lambda/d/\mu$ and $\tilde\lambda/\tilde d/\tilde\mu$ are the same up to a simple overall translation in the $\mb{Z}^2$-plane and, hence, that $h_{\lambda/d/\mu}=h_{\tilde\lambda/\tilde d/\tilde\mu}$. Setting $\mu=n^k$ and shifting $d$ by $k$, this becomes 
\begin{equation}
h_{\lambda/d+k/n^k}=h_{\tilde\lambda/d+k-\ell(\lambda)/\emptyset}=h_{\lambda/d/\emptyset}
\end{equation}
from which it is now evident that the latter functions are cylindric analogues of the (non-skew) complete symmetric functions $h_\lambda$. 

\begin{lemma}\label{lem:nscylh}
 Let $ \lambda \in \alc$ and $d \ge -m_n(\lambda)$, then we have the expansion
 \begin{equation}
  h_{\lambda/d/\emptyset}= \sum_{\nu} \frac{|\rS_\lambda|}{|\rS_\nu|}\, h_\nu\;,
 \end{equation}
where the sum runs over all $\nu\in \lambda \hat \rS_k\subset\cP_k$ with $|\nu|-|\lambda|=dn$. For all other values of $d\in\mb{Z}$ the function $h_{\lambda/d/\emptyset}$ is identically zero. Moreover, the (non-skew) cylindric complete symmetric functions $\{h_{\lambda/d/\emptyset}~|~\lambda\in\alc,\;d\geq -m_n(\lambda)\}$ are linearly independent.
\end{lemma}

\begin{proof}
 The constraint on $d$ follows trivially from observing that $\lambda/d/\emptyset$ for $d<\ell(\tilde\lambda)-k=-m_n(\lambda)$ is not a valid cylindric skew shape. From Theorem \ref{thm:ExpansionCylindricH}
 we have the expansion $h_{\lambda/d/\emptyset}=h_{\lambda/d+k/n^k}=\sum_{\nu \in \cP^+_k} N_{n^k \nu}^\lambda h_\nu$,
 where the sum runs over all $\nu \in \cP^+_k$ such that $dn=|\nu|-|\lambda|$.  Employing Lemma \ref{lem:Nreduction} this can be rewritten as $h_{\lambda/d+k/n^k}=\sum_{\nu \in \cP^+_k} N_{n^k \check\nu}^\lambda \frac{|\rS_{\check\nu}|}{|\rS_\nu|} h_\nu$, where $\check\nu$ is the unique intersection point of the orbit $\nu \hat \rS_k$
 with the alcove \eqref{alcove}. Using the equality $N_{n^k \check\nu}^\lambda =\delta_{\lambda \check\nu}$ proved in Corollary \ref{cor:symmN}, the claim follows since the only weights $\nu \in \cP^+_k$ for which $\check\nu=\lambda$ are the ones 
 satisfying the constraint $\nu \in \lambda \hat \rS_k$.
 
To show linear independence, note that each $\nu\in\cP^+_k$ has a unique intersection point with the alcove \eqref{alcove} under the level-$n$ action of $\hat \rS_k$. Hence, in an arbitrary linear combination of non-skew cylindric complete symmetric functions $h_{\lambda/d/\emptyset}$ we cannot get any cancellation since the $h_\nu$ themselves are linearly independent.
\end{proof}

As another immediate consequence of Theorem \ref{thm:ExpansionCylindricH}, namely of (ii), one has the following  equalities between matrix elements and coefficient functions,
\begin{equation}
\langle v^\lambda,H_\nu v_\mu\rangle=z^d\theta_{\lambda^{\vee}/d/\mu^{\vee}}(\nu)
\quad\text{and}\quad
\langle v^\lambda,H^*_\nu v_\mu\rangle=z^{-d}\theta_{\lambda/d/\mu}(\nu)\,.
\end{equation}
In particular, the identity $|\rS_\mu|\,\theta_{\lambda/d/\mu}(\nu)=|\rS_\lambda|\,\theta_{\mu^{\vee}/d/\lambda^{\vee}}(\nu)$ holds.
\begin{coro} Let $\lambda,\mu\in\alc$ and $d\in\mb{Z}_{\geq 0}$. Then
\begin{equation}
h_{\lambda/d/\mu}=\frac{|\rS_\lambda|}{|\rS_\mu|}\,h_{\mu^{\vee}/d/\lambda^{\vee}}\,.
\end{equation}
\end{coro}

  \begin{figure}
\centering
\includegraphics[width=1\textwidth]{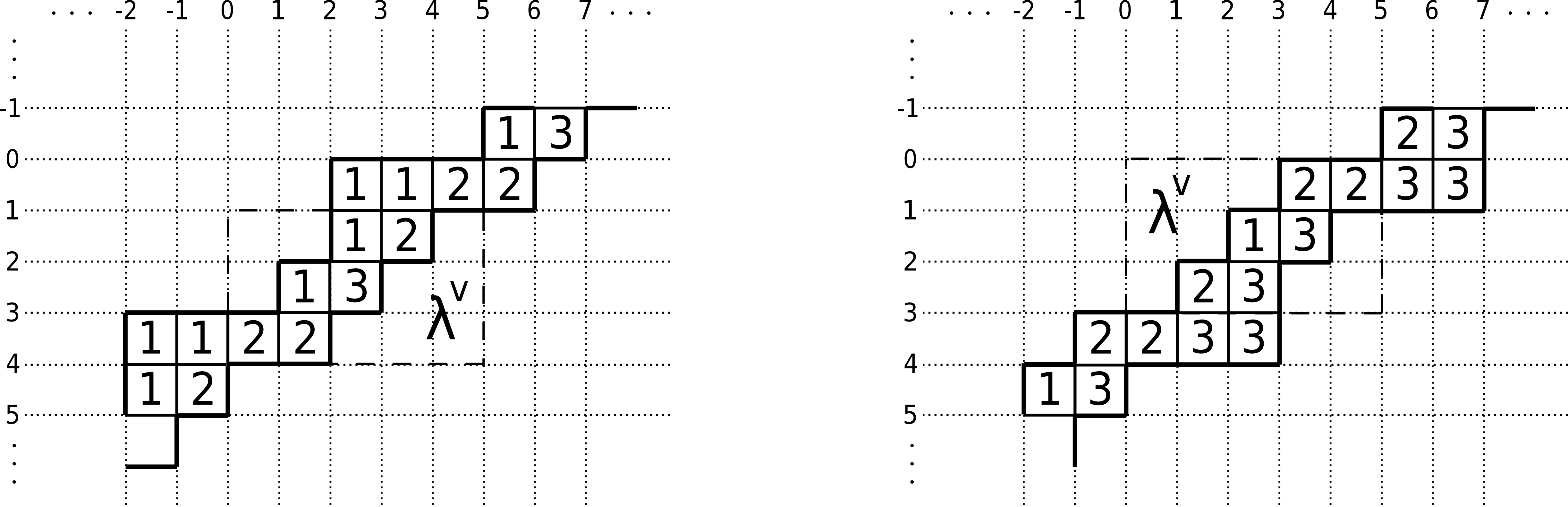}        
\caption{Set $n=4,k=3,d=1$ and choose $\lambda=(4,3,2)$, $\mu=(2,2,1)$ from $\mathcal{A}^+_{3}(4)$. Shown on the left is a CRPP $ \hat \pi $ of shape $\lambda/d/\mu$ and weight $(4,3,1)$. Note the bounding box with top left corner at $(d,0)$, height $k$ and width $n+1$. The image $\hat \pi^\vee$ of $\hat\pi$ under the map $\vee: \Pi_{k,n} \to \Pi_{k,n}$ is displayed on the right.
This is a CRPP of shape $\mu^\vee/d/\lambda^\vee$ and weight $(1,3,4)$, where $\mu^\vee=(4,3,3)$ and $\lambda^\vee=(3,2,1)$. One sees how $\hat \pi^\vee$ is obtained from $\hat \pi$ by
a rotation of $180^\circ$ and the swapping the numbers $1 \leftrightarrow 3$.
}
\label{ExamplePiVee}
\end{figure}

One might ask whether there is a bijection between CRPP of shape $\lambda/d/\mu$ and those of shape $\mu^{\vee}/d/\lambda^{\vee}$ which explains the above relation combinatorially.
\begin{prop}\label{prop:CRPPVee}
 Let $\Pi_{k,n}$ be the set of all CRPP. The map $\vee: \Pi_{k,n} \to \Pi_{k,n}$
which sends the CRPP $\hat \pi$ of shape $\lambda/d/\mu$ given by
\[
( \lambda^{(0)}[0]=\mu[0], \lambda^{(1)}[d_1], \dots, \lambda^{(l)}[d_l]=\lambda[d] )
\]
to the CRPP $\hat \pi^\vee$ of shape $\mu^\vee/d/\lambda^\vee$ given by
\[
 (\lambda^\vee[0], \lambda^{(l-1)^\vee}[d-d_{l-1}], \dots, \lambda^{(1)^\vee}[d-d_1],
 \lambda^{(0)^\vee}[d-d_0] = \mu^\vee [d])
\]
is an involution. Moreover, we have the equality,
 \begin{equation} 
  \theta_{\hat \pi} = \frac{|\rS_\lambda|}{|\rS_\mu|} \,\theta_{\hat \pi^\vee}\,.
  \end{equation} 
\end{prop} 
\begin{proof}
First we show that the set of points $\lambda/d/\mu$ is a cylindric skew diagram if and only if $\mu^\vee/d/\lambda^\vee$ is also one. Recall 
that $\lambda/d/\mu$ is a cylindric skew diagram if and only if $\mu[0] \le \lambda[d]$. From the equality $\lambda[d]_{i}-\mu[0]_i=\mu^\vee[d]_{k+1-i+d}-\lambda^\vee[0]_{k+1-i+d}$ for $i \in \mathbb{Z}$, which is a straightforward computation, noting that $\lambda[d]_i=\lambda_{i-d}$, one sees that $\lambda^\vee[0] \le \mu^\vee[d]$ must hold. This proves the claim.

After a manipulation of \eqref{CylindricTheta} using the identity 
$(\lambda \circ \tau^d)'_i=\lambda_i'+d$ and $\lambda'_i=k-(\lambda^{\vee})'_{n+2-i}$, we arrive at 
 \begin{equation} \label{ThetaVeeTheta}
 \theta_{\lambda/d/\mu} = \frac{|\rS_\lambda|}{|\rS_\mu|}\, \theta_{\mu^\vee/d/\lambda^\vee}\,.
\end{equation}
This proves our assertion for CRPPs $\hat\pi$ with $l=1$. The case $l>1$ now follows by observing that $\hat\pi$ is a sequence of cylindric skew shapes and that 
\begin{eqnarray*}
 \theta_{\hat \pi} = \prod_{i = 1}^l \theta_{\lambda^{(i)} /d_i-d_{i-1}/\lambda^{(i-1)}} =\prod_{i=1}^l \tfrac{|\rS_{\lambda^{(i)}}|}{|\rS_{\lambda^{(i-1)}}|} 
 \theta_{\lambda^{(i-1)^\vee}/(d-d_{i-1})-(d-d_i)/\lambda^{(i)^\vee}}
 = \frac{|\rS_\lambda|}{|\rS_\mu|} \theta_{\hat \pi^\vee}\,.
\end{eqnarray*}
\end{proof}

\subsection{Cylindric elementary symmetric functions} 
We now present the result analogous to Theorem \ref{thm:ExpansionCylindricH} for the second identity \eqref{Cauchy2}. As the line of argument apart from minor differences parallels closely the one in the previous section, we will mostly omit proofs unless there are important differences.
 
As a special case of cylindric reverse plane partitions one can define cylindric tableaux $\hat T$, where the entries are either strictly increasing down columns or along rows from left to right. Here we are interested in row strict CRPP defined as follows.
\begin{defi}
A {\em row strict} CRPP of shape $\Theta=\lambda/d/\mu$ is a map $\hat T:\Theta \to \mathbb{N}$ such that for any $(i,j) \in \Theta$ one has 
 \begin{eqnarray*}
 \hat T(i,j) &=& \hat T(i+k,j-n) \;, \\
  \hat T(i,j) &\le& \hat T(i+1,j), \text{ if } (i+1,j) \in \Theta \;, \\
    \hat T(i,j) &<& \hat T(i,j+1),\quad \text{ if }(i,j+1) \in \Theta \;.
\end{eqnarray*}
\end{defi}
Alternatively, a row strict CRPP $\hat T$ can be defined as a sequence of cylindric loops 
\begin{equation} \label{SequenceCylLoops}
( \lambda^{(0)}[0]=\mu[0], \lambda^{(1)}[d_1], \dots, \lambda^{(l)}[d_l]=\lambda[d] )
\end{equation}
with $\lambda^{(i)} \in \alc$ and $ d_i - d_{i-1} \ge 0$ such that 
$\hat T^{-1}(i)=\lambda^{(i)} / (d_i-d_{i-1}) / \lambda^{(i-1)}$ is a cylindric vertical strip. That is, in each row we have at most one box. 
We denote the weight of $\hat T$ by $\op{wt}(\hat T)=(\op{wt}_1(\hat T),\ldots,\op{wt}_l(\hat T))$. An example of a row strict CRPP for $n=4$ and $k=3$ is displayed in Figure \ref{fig:CRPPexamples}.

We will now generalise the set \eqref{setPsi} from Appendix A to the cylindric case and proceed 
in a similar fashion to Section \ref{CylCom}.

For $\lambda, \mu \in \alc$ and $d \in \mathbb{Z}_{\ge 0}$  define the set
\begin{equation} \label{NumberPsiLambdaDMu}
 \{ \tilde{w} \in \tilde{\rS}^{\mu}:  (\lambda \circ \tau^d)_i-(\mu \circ \tilde{w})_i=0,1,\,\forall i\in\mb{Z} \} \;,
\end{equation}
and denote by $\psi_{\lambda/d/\mu}$ its cardinality.

\begin{lemma} \label{lem:ExpressSetPsi}
The set \eqref{NumberPsiLambdaDMu} is non-empty if and only if $\lambda/d/\mu$ is a cylindric vertical strip. In the latter case we have that
\begin{equation} \label{CylindricPsiLambdaMu}
\psi_{\lambda / d / \mu} =
 \prod_{i=1}^n \binom{(\lambda \circ \tau^d)'_i-(\lambda \circ \tau^d)'_{i+1}}{(\lambda \circ \tau^d)'_i-\mu'_i}  \;.
 \end{equation}
\end{lemma}
Note once more that we define the binomial coefficients in \eqref{CylindricPsiLambdaMu} to be zero whenever one of their arguments is negative.
\begin{proof}
The first part of the statement follows from an analogous line of argument as in the proof of Lemma \ref{lem:ExpressSetTheta}.  Thus, we assume that $\lambda/d/\mu$ is a cylindric vertical strip and proceed by a similar strategy as in the proof of Lemma \ref{lem:ExpressSetTheta} rewriting the set \eqref{NumberPsiLambdaDMu} in the following alternative form,
  \begin{equation} \label{SubSetAffineSmuE}
 \mb{A}=\Big\{ (w,\alpha) \in \rS^\mu \times \cP_k~|~
 \lambda_i - (\mu \circ w \circ x^{-\alpha})_i=0,1 , \,  |\alpha|=d ,\,\alpha_i\geq 0\Big\} \;.
 \end{equation}
Next we construct a bijection between $\mb{A}$ and the set
\begin{equation} \label{OtherSetE}
\mb{B}= \{ \bar{w} \in \rS^{\mu \circ \tau^{-d} } ~|~ \lambda_i - (\mu \circ \tau^{-d} \circ \bar{w})_i =0,1 \} \;,
\end{equation} 
where in the latter the weights $\lambda$ and $\mu \circ \tau^{-d}$ belong to $\cP^+_k$. %
Fix an element $(w,\alpha)\in\mb{A}$. Because $\lambda,\mu \in \mathcal{A}^+_k(n)$ 
it follows that $\alpha_i$ is nonzero only if $(\mu \circ w)_i = n$,
in which case $\alpha_i=1$. This implies that $m_i(\mu \circ w \circ x^{-\alpha})=m_i(\mu \circ \tau^{-d})$ for $i=0,\dots,n$,
and thus there exists a unique permutation
 $\bar{w}(w,\alpha) \in \rS^{\mu \circ \tau^{-d} }$ such that 
$ \mu \circ \tau^{-d} \circ \bar{w} = \mu \circ w \circ x^{-\alpha}$ as weights in $\cP_k$. By construction $\bar w(w,\alpha)\in\mb{B}$.

Conversely, 
given $\bar{w} \in \mb{B} $ define a weight $\alpha(\bar w)\in \cP_k$ such that $\alpha_i=1$ if $(\mu \circ \tau^{-d} \circ \bar{w})_i=0$ and $\alpha_i=0$ otherwise.  Then there exists a unique permutation $w(\bar w) \in \rS^{\mu} $ such that 
$\mu \circ w \circ x^{-\alpha} = \mu \circ \tau^{-d} \circ \bar{w}$ as elements in $\cP_{k,n}$. By construction $(w(\bar w),\alpha(\bar w))\in\mb{A}$.

Noting that $(\mu \circ \tau^{-d})'_i=\mu'_i-d$ the asserted equality then follows, since the cardinality of 
\eqref{OtherSetE} is equal to $\psi_{\lambda/\mu \circ \tau^{-d}}$ by Lemma \ref{lem:CardinalityPsi} and \eqref{setPsi}. Hence, 
\begin{eqnarray*}
 \psi_{\lambda/d/\mu} = \psi_{\lambda/\mu \circ \tau^{-d}} = \prod_{i=1}^n \binom{\lambda'_i-\lambda'_{i+1}}{(\mu \circ \tau^{-d})'_{i}-\lambda'_{i+1}}
 = \prod_{i=1}^n \binom{(\lambda \circ \tau^d)'_i-(\lambda \circ \tau^d)'_{i+1}}{\mu'_{i}-(\lambda \circ \tau^d)'_{i+1}}  \;.
\end{eqnarray*}
\end{proof}

Given a row strict CRPP $\hat T$ set $\psi_{\hat T}=\prod_{i\ge 1} \psi_{\lambda^{(i)}/(d_i-d_{i-1})/\lambda^{(i-1)}}$, 
where the cylindric skew diagram $\lambda^{(i)}/(d_i-d_{i-1})/\lambda^{(i-1)}$ is 
the pre-image $\hat T^{-1}(i)$, and denote by $u^{\hat T}$
the monomial $u_1^{\op{wt}_1(\hat T)} u_2^{\op{wt}_2(\hat T)} \cdots$ in the indeterminates $u_i$.

\begin{defi}
For $\lambda,\mu \in \alc$ and $d \in \mathbb{Z}_{\ge 0}$, introduce the {\em cylindric elementary symmetric function} $e_{\lambda/d/\mu}$ as the weighted sum
\begin{equation}\label{cyle}
e_{\lambda/d/\mu}(u)=\sum_{\hat T}  \psi_{\hat T} u^{\hat T}
\end{equation}
over all row strict CRPP $\hat T$ of shape $\lambda/d/\mu$.
\end{defi}
Since $\psi_{\lambda/0/\mu}=\psi_{\lambda/\mu}$, according to Lemma \ref{lem:ExpressSetPsi}, it follows that for $d=0$ we recover the (non-cylindric) skew elementary cylindric function $e_{\lambda/0/\mu}=e_{\lambda/\mu}$ from Appendix A; see \eqref{e2T}.

In a similar vein as in the case of cylindric complete symmetric functions one proves that $e_{\lambda/d/\mu}$ is also a symmetric function by deriving first the following product expansion in the quotient ring $\V^+_k$. 

Let $\lambda,\mu \in \mathcal{A}_{k,n}$, $\nu \in \cP^+$, $d\in\mb{Z}_{\ge 0}$ and define
\begin{equation} \label{psiLambdaMuNu}
 \psi_{\lambda / d / \mu}(\nu)= \sum_{\hat T} \psi_{\hat T},
\end{equation}
where the sum is restricted
to row strict CRPP $\hat T$ of shape $\lambda/d/\mu$ and weight $\op{wt}(\hat T)=\nu$. 
\begin{lemma}
The following product rule holds in $\V^+_k$,
 \begin{equation} \label{mMueNuCyl}
   m_\mu(x_1^{-1},\ldots,x^{-1}_k) e_\nu(x_1^{-1},\ldots,x^{-1}_k)  =  \sum_{\lambda \in \alc} z^{-d} \psi_{\lambda / d / \mu}(\nu) 
   m_\lambda(x_1^{-1},\ldots,x^{-1}_k) \;,
 \end{equation}
where $d = \frac{|\mu|+|\nu|-|\lambda|}{n}$. In particular, $\psi_{\lambda / d / \mu}(\nu)$ is nonzero only if $dn = |\mu|+|\nu|-|\lambda|$.
\end{lemma}
Analogous to the line of argument followed in Section \ref{CylCom}, we deduce from \eqref{mMueNuCyl} that $\psi_{\lambda / d / \mu}(\nu)$ is invariant under permutations of $\nu$. Moreover, setting $\nu=(1^r)$ with $r=dn+|\lambda|-|\mu|$ there exists at least one $\hat T$ of that weight and, hence, $\psi_{\lambda/d/\mu}(1^r)$ is nonzero as long as $\lambda/d/\mu$ is a valid cylindric shape.

Because the proof of the following two statements parallels closely our previous discussion we omit it.
\begin{coro}
(i) The function $e_{\lambda/d/\mu}$ has the expansion
\begin{equation}\label{CyleM}
e_{\lambda/d/\mu}=\sum_{\nu\in\cP^+}\psi_{\lambda/d/\mu}(\nu)m_\nu
\end{equation}
into monomial symmetric functions and, hence, is symmetric. 

(ii) The expansion coefficients \eqref{psiLambdaMuNu} have the following alternative expression, 
\begin{equation}\label{PsiM}
 \psi_{\lambda/d/\mu}(\nu) = \sum_{\sigma \in \cP^+_k} M_{\nu \sigma} N_{\sigma \mu}^\lambda\,,
 \end{equation}
where $M_{\alpha\beta}$ is the number of all $(0,1)$-matrices with row sums equal to $\alpha_i$ and column sums equal to $\beta_i$.
\end{coro}

Taking matrix elements in the identity \eqref{Cauchy2} we obtain the following: 
\begin{theorem} \label{thm:ExpansionCylindricE}
Let $\lambda,\mu\in\alc$ and $d\in\mathbb{Z}_{\ge 0}$. Then (i) the symmetric function $e_{\lambda/d/\mu}$ has the expansion
\begin{equation}
e_{\lambda/d/\mu}=\sum_{\nu\in\cP^+_k}N_{\mu\nu}^{\lambda} \;e_\nu \label{cyle2e} 
\end{equation}
into the basis $\{e_\nu\}_{\nu\in\cP^+}\subset \Lambda$, where the sum is restricted to those $\nu\in\cP^+_k$ for which $|\nu|=dn+|\lambda|-|\mu|$, and (ii) we have the formal power series expansions 
\begin{eqnarray}
\langle v^{\lambda^{\vee}},\prod_{j\geq 1}E(u_j)v_{\mu^{\vee}}\rangle =
\overline{\langle v^\lambda,\prod_{j\geq 1}E^*(u_j)v_\mu\rangle}
&=& \sum_{d\geq 0}z^{d} e_{\lambda/d/\mu}(u)\;.
\end{eqnarray}
\end{theorem}
Note that (ii) implies the following equalities between matrix elements and coefficient functions, 
\begin{equation}
\langle v^\lambda,E_\nu v_\mu\rangle=z^{d}\psi_{\lambda^{\vee}/d/\mu^{\vee}}(\nu)
\qquad\text{and}\qquad\langle v^\lambda,E^*_\nu v_\mu\rangle=z^{-d}\psi_{\lambda/d/\mu}(\nu)\,.
\end{equation}
In particular, the identity $|\rS_\mu|\,\psi_{\lambda/d/\mu}(\nu)=|\rS_\lambda|\,\psi_{\mu^{\vee}/d/\lambda^{\vee}}(\nu)$ holds. By a similar line of argument as in the case of cylindric complete symmetric functions one arrives at the following `duality relations' for cylindric elementary symmetric functions and row strict CRPP under the involution $\vee:\Pi_{k,n}\to\Pi_{k,n}$:

\begin{prop}
The involution $\vee: \Pi_{k,n} \to \Pi_{k,n}$ from Proposition \ref{prop:CRPPVee} preserves the subset of row strict CRPP $\hat T$ and one has the equalities
 \begin{equation} 
  \psi_{\hat T} = \frac{|\rS_\lambda|}{|\rS_\mu|} \,\psi_{\hat T^\vee}\qquad\text{ and }\qquad
 e_{\lambda/d/\mu}=\frac{|\rS_\lambda|}{|\rS_\mu|}\,e_{\mu^{\vee}/d/\lambda^{\vee}} \;.
  \end{equation} 
\end{prop}
We omit the proof as the steps are analogous to the ones when proving Proposition \ref{prop:CRPPVee}.

\subsection{Cylindric symmetric functions as positive coalgebras}
As an easy consequence of Theorems \ref{thm:ExpansionCylindricH} and \ref{thm:ExpansionCylindricE} we now compute the coproduct of the cylindric symmetric functions in $\Lambda$ viewed as a Hopf algebra (see Appendix A). This will allow us to identify certain subspaces of $\Lambda$ whose non-negative structure constants are the fusion coefficients $N_{\mu\nu}^\lambda\in\mb{Z}_{\geq 0}$ with $\lambda,\mu,\nu\in\alc$ and for which we have derived three equivalent expressions in \eqref{Lfusion}, \eqref{Verlinde_bos} and \eqref{Nmunulambda}. 
  \begin{coro}
 The image of the cylindric complete symmetric functions under the coproduct in the Hopf algebra $\Lambda$  
 is given by
 \begin{eqnarray}  \label{DeltaCylh}
  \Delta(h_{\lambda/d/\mu}) &=& \sum_{d_1+d_2=d} \sum_{\nu \in \alc}
  h_{\lambda/d_1/\nu} \otimes h_{\nu / d_2 / \mu}  \;.
 \end{eqnarray}
 The analogous formula holds for $e_{\lambda/d/\mu}$ and both families of functions are related via 
 \begin{equation}\label{cylhanticyle}
\gamma(h_{\lambda/d/\mu})= (-1)^{|\lambda|-|\mu|+dn} e_{\lambda/d/\mu}\,,
\end{equation}
where $\gamma:\Lambda\to\Lambda$ is the antipode. 
 \end{coro}
 Note that \eqref{DeltaCylh} implies the recurrence formula
\[
h_{\lambda/d/\mu}(u_1,u_2,\ldots)= \sum_{d_1+d_2=d} \sum_{\nu \in \alc}
  u_1^{|\nu|-|\mu|+d_2 n}\theta_{\nu/d_2/\mu} h_{\lambda/d_1/\nu}(u_2,u_3,\ldots)
\]
with the coefficient $\theta_{\nu/d_2/\mu}$ given by \eqref{CylindricTheta}. The analogous identity holds for $e_{\lambda/d/\mu}$ with $\psi_{\lambda/d_2/\mu}$ from \eqref{CylindricPsiLambdaMu} instead.
 
 \begin{proof}
The coproduct expressions follow from inserting the identity map into (ii) of Theorem \ref{thm:ExpansionCylindricH},
\begin{multline*}
\langle v^{\lambda},H^*(u_1)H^*(u_2)\cdots H^*(v_1)H^*(v_2)\cdots v_\mu\rangle=\\
\sum_{\nu\in\alc} \langle v^{\lambda},H^*(u_1)H^*(u_2)\cdots v_{\nu}\rangle\langle v^{\nu},H^*(v_1)H^*(v_2)\cdots v_\mu\rangle\,.
\end{multline*}
Using the power series expansions from (ii) of Theorem \ref{thm:ExpansionCylindricH} and comparing coefficients of each power $z^{d}$ the asserted equality follows. The same trick applies to the coproduct formula for $e_{\lambda/d/\mu}$.

The relation involving the antipode $\gamma$ (see Appendix A for its definition) follows from the expansions (i) in Theorems \ref{thm:ExpansionCylindricH} and \ref{thm:ExpansionCylindricE} as well as the identity $\gamma(h_\lambda)=(-1)^{|\lambda|}e_{\lambda}$.
\end{proof}

Setting $\mu=\emptyset$ in \eqref{DeltaCylh} and using the expansion \eqref{cylh2cylh} as well as \eqref{cylhanticyle} we arrive at the desired result:
\begin{coro}\label{cor:coalgebra}
The respective subspaces spanned by %
\begin{equation}\label{cylsets}
\{h_{\lambda/d/\emptyset}~|~\lambda\in\alc,\;d\ge -m_n(\lambda)\}\qquad\text{and}\qquad 
 \{e_{\lambda/d/\emptyset}~|~\lambda\in\alc,\;d\ge -m_n(\lambda)\}
\end{equation}
each form a positive subcoalgebra of $\Lambda$ with structure constants $N_{\mu\nu}^\lambda$, $\lambda,\mu,\nu\in\alc$, 
 \begin{equation}\label{coalgebra}
 \Delta(h_{\lambda/d/\emptyset})=\sum_{d_1+d_2=d}\sum_{\mu\in\alc}h_{\lambda/d_1/\mu}\otimes h_{\mu/d_2/\emptyset},\quad
 h_{\lambda/d_1/\mu}=\sum_{d'_1\in\mb{Z}}\sum_{\nu}N^{\lambda}_{\mu\nu}h_{\nu/d_1-d'_1/\emptyset}\;,
 \end{equation}
 where the second sum runs over all $\nu\in\alc$ such that $|\nu|=|\lambda|+d'_1n-|\mu|$. The analogous coproduct expansion holds for the functions $e_{\lambda/d/\emptyset}$.
\end{coro}

\subsection{Expansions into powers sums}
In light of Theorems \ref{thm:ExpansionCylindricH} and \ref{thm:ExpansionCylindricE} and the definitions \eqref{defE&H}, \eqref{E} and \eqref{H}, we discuss the expansion of the cylindric functions from the previous sections into power sums; compare with the formulae \eqref{skewhe2p} in Appendix A. The resulting expansions coefficients describe the inverse image of the cylindric functions under the characteristic map \eqref{ch}. Note that according to Theorems \ref{thm:ExpansionCylindricH} and \ref{thm:ExpansionCylindricE} the cylindric functions $h_{\lambda/d/\mu}$ and $e_{\lambda/d/\mu}$ have both degree $m=|\lambda|-|\mu|+dn$.

We wish to obtain the analogue of Lemma \ref{lem:varphi2ACT} for the cylindric case and start by introducing the generalisation of an adjacent column tableau; see Appendix A. 
\begin{defi}\label{def:CACPP}
For $\lambda,\mu\in\alc$ and $d\geq 0$ call the cylindric skew shape $\lambda/d/\mu$ a {\em cylindric adjacent column strip} (CACS) if it is either a cylindric horizontal strip whose boxes lie in adjacent columns or a translation thereof, i.e. there exists $0<d'\leq d$ such that $\lambda/d-d'/\mu$ obeys the former conditions. We call a CRPP $\hat \pi$ a {\em cylindric adjacent column plane partition} (CACPP) if each cylindric skew shape $\hat \pi^{-1}(i)=\lambda^{(i)}/d_i-d_{i-1}/\lambda^{(i-1)}$ is a cylindric adjacent column strip. 
\end{defi}
Implicit in the above definition is that $d'$ can only take the values $d'=d$ or $d'=d-1$, since the cylindric skew shape $\lambda/d-d'/\mu$ can only be a horizontal strip if $d-d'=0,1$. See Figure \ref{fig:CRPPexamples} for an example of a CACPP when $n=4$ and $k=3$.

By similar arguments as in the non-cylindric case discussed in Appendix A one shows the following:
\begin{lemma}\label{lem:CACS}
Let $\lambda/d/\mu$ be a CACS with $r=(|\lambda|-|\mu|+nd)\not\in n\mb{N}$. Then there exists a unique $1\leq a\leq n$ such that $\lambda\circ\tau^d=\mu(a,r)$, where $\mu(a,r)$ is the unique element in $\cP_{k,n}$ whose cylindric loop is obtained as follows: starting in columns $a\mod n$ of the cylindric loop of $\mu$ consecutively add one box in each of the $(r-1)$ adjacent columns on the right. 
\end{lemma}

This lemma explains our previous definition of a CACS: similar as in the non-cylindric case discussed in Appendix A we consider cylindric skew shapes $\lambda/d/\mu$ that are obtained by consecutively adding one box in adjacent columns;  see Figure \ref{fig:CACTexample} for an example. 

However, in contrast to the non-cylindric case it can now happen that if $r=|\lambda|-|\mu|+nd\geq n$ the strip `winds around the cylinder' and due to the periodicity condition the final skew shape $\lambda/d/\mu$ with $d\geq 1$ will contain more than one box in the same column; see once more Figure \ref{fig:CACTexample} for an example. If this is the case then the skew shape can be reduced by applying the translation operator $\tau^{-1}$, say $d'$ times, until $\lambda/d-d'/\mu$ contains less than $n$ boxes. This `reduced skew shape' must again be a horizontal strip. In particular, if $r=|\lambda|-|\mu|+nd=mn$ is a multiple of $n$ then we have the trivial case where $\lambda=\mu$ and the skew shape is obtained by acting $m$ times with $\tau$.

\begin{figure}
\centering
\includegraphics[width=1\textwidth]{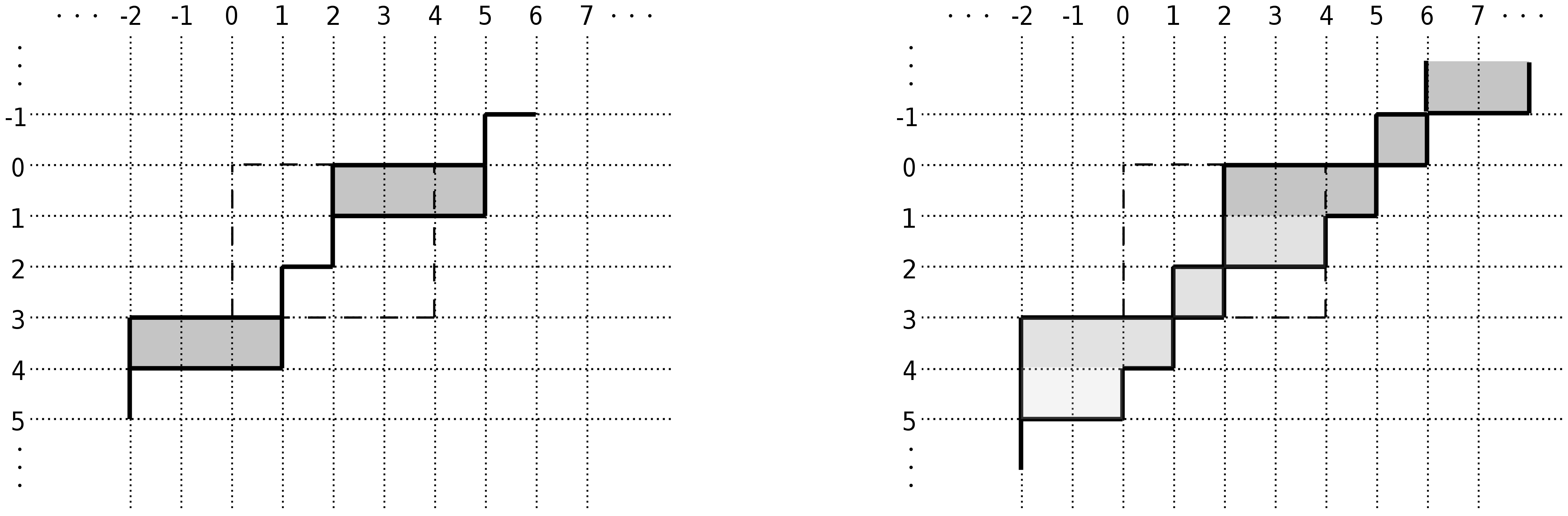}   
\caption{Choose $\lambda=(2,1,1)$, $\mu=(2,2,1) $ and $\nu=(4,2,1)$ in $\mathcal{A}^+_{3}(4)$. On the left we have the CACS $\lambda/1/\mu$ with $\lambda \circ \tau = \mu(3,3)$ obtained by adding one box in columns $3,4,5$ modulo $n$ to the cylindric loop of $\mu$. Since $m_5(\lambda\circ\tau)=m_1(\lambda)=2$ we have $\varphi_{\lambda/1/\mu}=2$. Depicted on the right the CACS $\nu/1/\mu$ where $\nu \circ \tau = \mu(3,6)$ is obtained by adding a box in columns $3$ to $8$ modulo $n$. As indicated by the shading of boxes, here the CACS winds around the cylinder. Note that this CACS is obtained by translation of the horizontal strip $\nu/0/\mu=\nu/\mu$. Since $m_8(\nu\circ\tau)=m_4(\nu)=1$ we have $\varphi_{\nu/1/\mu}=1$. }\label{fig:CACTexample}
\end{figure}

Let us now extend the definition of the function $\varphi_{\lambda/0/\mu}=\varphi_{\lambda/\mu}$ from Lemma \ref{lem:varphi2ACT} to $d>0$, where $\lambda/d/\mu$ with $\lambda,\mu\in\alc$ is a CACS. Employing Lemma \ref{lem:CACS} we introduce for each such adjacent column strip the coefficient
\begin{equation} \label{varphivalues}
 \varphi_{\lambda/d/\mu} = \begin{dcases}
                            m_{a-1+r}(\lambda \circ \tau^d), & r \text{ mod } n \neq 0 \\
                            k, & \text{ else }
                           \end{dcases}
\end{equation}
and for a given CACPP $\hat \pi$  define $\varphi_{\hat \pi}= \prod_{i=1}^{l} \varphi_{\lambda^{(i)} / (d_i-d_{i-1})/\lambda^{(i-1)}}$. We adopt the convention of setting $\varphi_{\lambda^{(i)} / (d_i-d_{i-1})/\lambda^{(i-1)}}=0$ if $\lambda^{(i)} / (d_i-d_{i-1})/\lambda^{(i-1)}$ is not an adjacent column strip and define the weight of $\hat\pi$ as in the case of a more general CRPP. Denote by 
 \begin{equation} \label{CylvarphiLambdaMuNu}
 \varphi_{\lambda/d/\mu}(\nu) = \sum_{\hat T} \varphi_{\hat T} \;, 
\end{equation}
the weighted sum over all CACPP of shape $\lambda / d/\mu$ and weight $\nu$. Note that it follows from these definitions that for $\nu\in\cP^+$ with some $\nu_i>n$ we have the equality $\varphi_{\lambda/d/\mu}(\nu) =\varphi_{\lambda/(d-1)/\mu}(\ldots,\nu_i-n,\ldots)$, since both sets of CACPP are in bijection by translating the $i$th CACS with $\tau$ and because $m_{a-1+r}(\lambda \circ \tau^{d})=m_{a-1+r-n}(\lambda \circ \tau^{d-1})$ for $r\geq n$.

\begin{theorem}
Let $\lambda,\mu \in \alc$ and $d \in \mathbb{Z}_{\ge 0}$. The cylindric functions $h_{\lambda/d/\mu}$ and $e_{\lambda/d/\mu}$ have the expansions
    \begin{equation}\label{cylhe2p}
   h_{\lambda/d/\mu} = \sum_{\nu\in\cP^+} \frac{\varphi_{\lambda/d/\mu}(\nu)}{\z_\nu} p_\nu 
      \quad\text{and}\quad
      e_{\lambda/d/\mu} = \sum_{\nu\in\cP^+} \frac{\varphi_{\lambda/d/\mu}(\nu)}{\z_\nu} \epsilon_\nu p_\nu 
  \end{equation}
into the basis of power sums. In particular, $\varphi_{\lambda / d / \mu}(\nu)$ vanishes unless $dn = |\mu|+|\nu|-|\lambda|$.
\end{theorem}
The proof of the theorem follows the same strategy as the proof of the previous expansion formulae. Namely, first one shows the following product expansion in the quotient ring $\V_k^+$:
 \begin{lemma}
 For $\mu\in\alc$ and $\nu\in\cP^+$ we have the following product identity in $\V^+_k$,
 \begin{equation} \label{mMupNuCyl}
   m_\mu(x_1^{-1},\ldots,x^{-1}_k) p_\nu(x_1^{-1},\ldots,x^{-1}_k)  =  \sum_{\lambda \in \alc} z^{-d} \varphi_{\lambda / d / \mu}(\nu) 
   m_\lambda(x_1^{-1},\ldots,x^{-1}_k) \;,
 \end{equation}
where only those $\lambda$ appear in the sum for which $|\lambda| = |\mu|+|\nu|-dn$ with $d\in\mb{Z}_{\geq 0}$. 
\end{lemma}  

 The second lemma needed to prove the expansion into power sums is the following generalisation of equation \eqref{varphiR} from Appendix A. 
 \begin{lemma}
 Let $R_{\lambda\mu}$ be the transition matrix from the basis of monomial symmetric functions to power sums, then we have the following equality
  \begin{equation}
   \varphi_{\lambda/d/\mu}(\nu) = \sum_{\sigma \in \cP^+_k} R_{\nu \sigma} N_{\sigma \mu}^\lambda,
  \end{equation}
  where, once more, $d=\frac{|\mu|+|\nu|-|\lambda|}{n}\in\mb{Z}_{\geq 0}$.
\end{lemma}
We omit the proof as it is similar to the non-cylindric case.

\begin{coro}
We have the following formal power series expansions
\begin{eqnarray}
\langle v^{\lambda^{\vee}},\prod_{j\geq 1}H(u_j)v_{\mu^{\vee}}\rangle= 
\overline{\langle v^\lambda,\prod_{j\geq 1}H^*(u_j)v_\mu\rangle}
&=& \sum_{d\geq 0}z^{d}\sum_{\nu}\frac{\varphi_{\lambda/d/\mu}(\nu)}{\z_{\nu}}p_\nu(u)\;,
\end{eqnarray}
where the sum runs over all $\nu\in\cP^+$ such that $dn+|\lambda|=|\mu|+|\nu|$. The analogous expansions hold for the cylindric elementary symmetric functions when replacing $H\to E$ and $p_\nu\to\epsilon_\nu p_\nu$; compare with \eqref{defE&H}. 
\end{coro}
\begin{proof}
The power series expansions follow from the identity
\begin{equation}\label{Cauchy3}
\prod_{j\geq 1}H(u_j)=
\prod_{i=1}^k\prod_{j\geq 1}(1-X_i u_j)^{-1}=\sum_{\lambda\in\cP^+}\z_\lambda^{-1}P_\lambda p_{\lambda}(u),
\end{equation}
and its analogue for $P^*_\lambda$, in $\End_{\R} \bV_k$. These latter identities are a direct consequence of Lemma \ref{lem:Lambda2U}. Taking matrix elements in the subspace of symmetric tensors we arrive at the desired equalities.
\end{proof}
Note that we allow for weights $\nu\in\cP^+$ in the matrix elements $\langle v^\lambda,P^*_\nu v_\mu\rangle$ and that this is consistent with the definition of $\varphi_{\lambda/d/\mu}(\nu)$ in terms of cylindric adjacent column strips. Namely, suppose that $\nu_1\geq n$ then
\begin{multline*}
z^{-d}\varphi_{\lambda/d/\mu}(\nu)=\langle v^\lambda,P^*_\nu v_\mu\rangle=\\
z^{-1}\langle v^\lambda,P^*_{\nu_1-n}P^*_{\nu_2}P^*_{\nu_3}\ldots v_\mu\rangle=
z^{-d}\varphi_{\lambda/(d-1)/\mu}(\nu_1-n,\nu_2,\nu_3,\ldots)\;,
\end{multline*}
where the last equality is the previously explained identity between weighted sums of CACPP which differ by shifting one of their CACS with $\tau$.

Furthermore, the above expansion into power sums implies the following equalities between matrix elements and coefficient functions, 
\begin{equation}
\langle v^\lambda,P_\nu v_\mu\rangle=z^{d}\varphi_{\lambda^{\vee}/d/\mu^{\vee}}(\nu)\quad\text{and}\quad
\langle v^\lambda,P^*_\nu v_\mu\rangle=z^{-d}\varphi_{\lambda/d/\mu}(\nu)\,.
\end{equation}
These formulae describe the image of $U'$ in $\End_{\R}\bV^+_k$ in purely combinatorial terms.

\subsection{Shifted action and cylindric Schur functions}
We now describe the connection between our construction and the cylindric Schur functions discussed by several authors in the literature; see e.g. \cite{postnikov2005affine, mcnamara2006cylindric, lam2006affine}. The new aspect in this article is their link to the {\em shifted action} of $\hat \rS_k$, the principal subalgebra and the proof that they form a positive subcoalgebra of $\Lambda$ together with the expansion \eqref{mcnamara_conj}.

Let $\rho=(k,\ldots,2,1)\in\cP_k^+$ be the staircase partition of length $k$, then for any $\lambda\in\salc$ the difference $\bar\lambda=\lambda-\rho$ is a partition whose Young diagram fits inside a bounding box of height $k$ and width $n-k$. Denote the set of these `boxed partitions' by 
\begin{equation}
\balc=\{\bar\lambda~|~n-k\geq\bar\lambda_1\geq\ldots,\geq \bar\lambda_k\geq 0\}\,.
\end{equation} 
To obtain the corresponding cylindric loops, observe that because of the shift by $\rho$ weights $\bar\lambda\in\cP_k$ are now identified with maps $\bar\lambda:\mb{Z}\to\mb{Z}$ in the set $\cP_{k,n-k}$, i.e. they satisfy the condition $\bar\lambda_{i+k}=\bar\lambda_{i}-n+k$. As before these maps describe lattice paths $\{(i,\bar\lambda_i)\}_{i\in\mb{Z}}$ on a cylinder $\mf{C}_{k,n-k}=\mb{Z}^2/(-k,n-k)\mb{Z}$ albeit with circumference $n-k$ instead of $n$. 

Denote by $\bar\rho:\mb{Z}\to\mb{Z}$ the map whose values on $[k]\subset\mb{Z}$ are fixed by $\rho\in\cP_k$ and which satisfies $\bar\rho_{i+k}=\bar\rho_i-k$.
\begin{lemma}
The map $\cP_{k,n-k}\times\hat \rS_k\to \cP_{k,n-k}$ given by
\begin{equation}\label{shifted_action}
(\lambda,\hat w)\mapsto\bar\lambda\odot \hat w=(\bar\lambda+\bar\rho)\circ\hat w-\bar\rho\;,
\end{equation}
where the plus and minus sign refer to the usual addition and subtraction of maps, defines a right action. The latter coincides with the known shifted level-$n$ action of $\hat \rS_k$ on $\cP_k$.
\end{lemma}
\begin{proof}
First we note that $(\bar\lambda+\bar\rho)_{i+k}=\bar\lambda_{i+k}+\bar\rho_{i+k}=(\bar\lambda+\bar\rho)_i-n$ and, hence, $(\bar\lambda+\bar\rho)$ is in the set $\cP_{k,n}$. So, we can exploit the previous right level $n$ action on $\cP_{k,n}$ from Lemma \ref{lem:Saction}. By construction the resulting map $\bar\lambda\odot \hat w$ is then in $\cP_{k,n-k}$. Moreover, restricting each map to $[k]\subset\mb{Z}$, one proves that this gives the familiar shifted Weyl group action on $\cP_k$.
\end{proof}
Note that $\balc$ is a fundamental region with respect to the action \eqref{shifted_action}. The action of the shift operator $\tau$ on a boxed partition $\bar\lambda\in\balc$ consists of adding a {\em rim hook} or {\em ribbon} of length $n$ starting in row $k$ and ending in row one. Here a ribbon is a skew shape which does not contain a $2\times 2$ block of boxes and which is edgewise connected, i.e. all the boxes share a common edge with another box in the skew shape; see Figure \ref{fig:CylindricRibbonexample} for an example. The particular $n$-ribbon generated by $\tau$ is obtained by adding a horizontal strip of width $n-k$ and then a vertical strip of height $k$ (or vice versa) and has been called a {\em circular ribbon} in \cite{postnikov2005affine} as it `winds once around the cylinder' $\mf{C}_{k,n-k}$. Thus, repeated action with $\tau$ now corresponds to a shift of the cylindric loop $\bar\lambda[0]$ in the direction of the lattice vector $(1,1)$ in $\mb{Z}^2$. This coincides with the conventions used in \cite{postnikov2005affine, mcnamara2006cylindric}.

\begin{figure}
\centering
\includegraphics[width=1\textwidth]{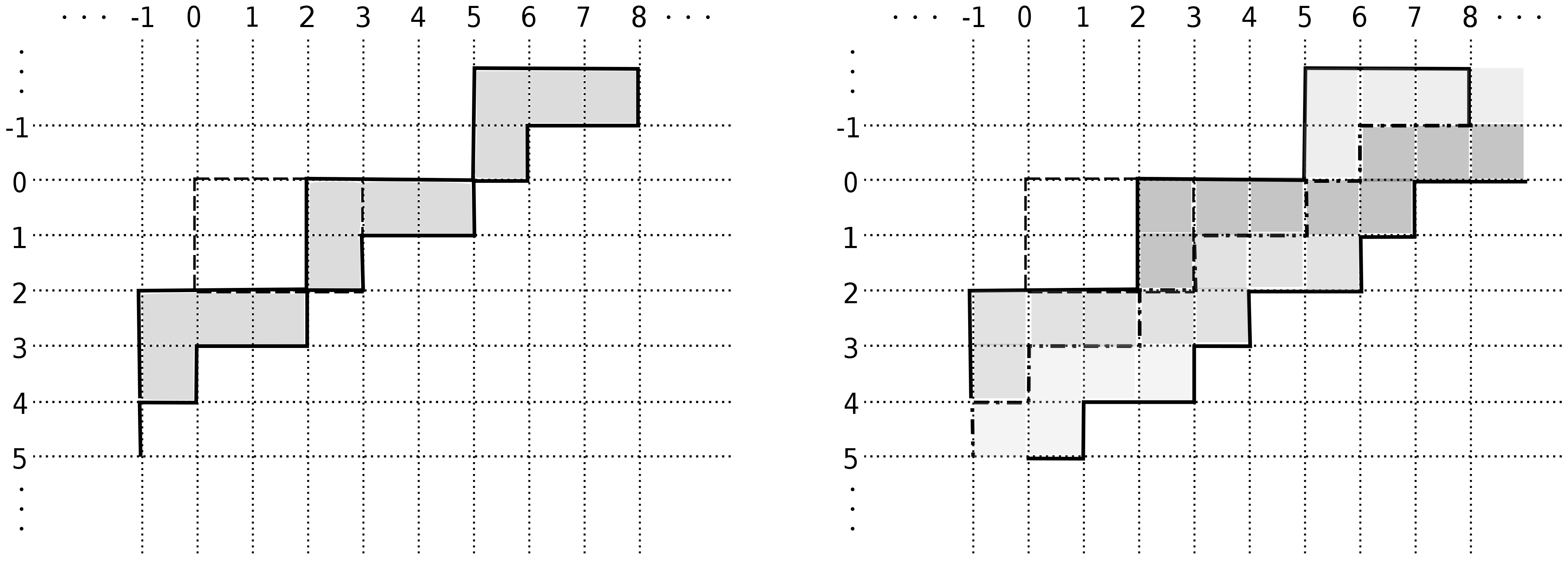}   
\caption{Set $n=5$ and $k=2$. Shown above are examples of cylindric ribbons for $\bar\lambda=(2,1)$, $\bar\mu=(2,2)$. On the left we have the cylindric skew shape $\bar\lambda/1/\bar\mu$ which is a ribbon of length 4 and height 2. On the right the cylindric ribbon $\bar\lambda/2/\bar\mu$ which has length 5+4=9 and height 2+2-1=3 because a circular $n$-ribbon has been added. Hence, it winds once around the cylinder. The resulting cylindric loop belongs to the partition $\bar\mu(3,9)$: starting in the shifted diagonals $3\mod\!\!n$ one consecutively adds 9 boxes, one in each subsequent diagonal as indicated by the shading. }\label{fig:CylindricRibbonexample}
\end{figure}

\begin{defi}
Define a {\em cylindric ribbon} as a cylindric skew shape $\bar\lambda/d/\bar\mu$ which is either a ribbon or a translation thereof, i.e. is obtained by adding additional circular $n$-ribbons by acting with $\tau$ on $\bar\lambda$. A {\em cylindric ribbon plane partition} is a cylindric reverse plane partition $\bar\pi$ such that each cylindric skew shape $\bar \pi^{-1}(i)$ is a cylindric ribbon.
\end{defi}
The shift by $\rho$ provides an obvious bijection between cylindric adjacent column strips $\lambda/d/\mu$ with $\lambda,\mu\in\salc$ and cylindric ribbons $\bar\lambda/d/\bar\mu$ and one obtains therefore the following analogue of Lemma \ref{lem:CACS}:
\begin{lemma}\label{lem:CR}
Let $\bar\lambda/d/\bar\mu$ be a cylindric ribbon with $r=(|\bar\lambda|-|\bar\mu|+nd)\not\in n\mb{N}$. Then there exists a unique $1\leq a\leq n$ such that $\bar\lambda\circ\tau^d=\bar\mu(a,r)$, where $\bar\mu(a,r)$ is the unique element in $\cP_{k,n-k}$ whose cylindric loop is obtained by consecutively adding one box $(i,j)$ per (shifted) diagonal $k+j-i$ in the cylindric loop of $\bar\mu$ starting in the diagonals $a\mod n$ and ending in the diagonals $(a+r-1)\mod n$.
\end{lemma}
The case when $r=mn$ corresponds to $\bar\lambda=\bar \mu$ and adding $m$ circular $n$-ribbons; compare with our previous discussion of cylindric adjacent column strips. Define the {\em height} $\op{ht}(\bar\lambda/d/\bar\mu)$ of a cylindric ribbon $\bar\lambda/d/\bar\mu$ to be the number of rows $\bar\mu(a,r)$ occupies. Each circular ribbon has by definition height $k$. For $\bar\lambda,\bar\mu\in\balc$ set
\begin{equation}
\chi_{\bar\lambda/d/\bar\mu}(\nu)=\sum_{\bar\pi}\prod_{i\geq 0}(-1)^{\op{ht}(\bar\pi^{-1}(i))-1},
\end{equation}
where the sum runs over all cylindric ribbon plane partitions $\bar\pi$ of shape $\bar\lambda/d/\bar\mu$ and weight $\nu$.

\begin{lemma}
Let $\lambda,\mu\in\salc$. Then we have the following equality between matrix elements in the subspace of alternating tensors \eqref{alt_tensors} and weighted sums,
 \begin{equation}
 \langle v^{\bar\lambda}, P_\nu^* v_{\bar\mu}\rangle=z^{-d}(-1)^{d(k-1)}\chi_{\bar\lambda/d/\bar\mu}(\nu)\,,
 \end{equation}
 where both sides are identically zero unless $d=(|\mu|+|\nu|-|\lambda|)/n\in\mb{Z}_{\geq 0}$.
\end{lemma}

\begin{rema}
For $z=(-1)^{k-1}q^{-1}$ the above matrix elements encode for the special case $\nu=(r)$ with $r\leq n$ the Murnaghan-Nakayama rule for the quantum cohomology ring $qH^*(\op{Gr}(k,n))$ described in \cite[Theorem 2]{morrison2018two}.
\end{rema}

\begin{proof}
It suffices to prove the assertion for $\nu=(r)$ as the general case then follows from repeated application of the formula where $\nu$ has only one nonzero part.

Let us first assume $r<n$. Parametrise the partition $\mu\in\salc$ in terms of a $01$-word of length $n$ where the 1-letters are at the positions $\mu_i$ from left to right. The action of the Lie algebra element $e_{ij}$ with $i>j$ on $v_{\bar\mu}$ then yields the vector $(-1)^{\sum_{l=j}^im_l(\mu)-1}v_{\bar\mu(j,i-j)}$ provided $m_j(\mu)=1$. The minus sign results from the fact that we take matrix elements in the subspace of alternating tensors. Thus, $P^*_r$ adds all possible ribbons of length $r$ to $\bar\mu$ with sign factor $(-1)^{\op{ht}(\bar\lambda/d/\bar\mu)-1+d(k-1)}$ with $\bar\lambda\odot\tau^d=\bar\mu(j,r)$ and $d=0,1$. This proves the assertion for $r<n$.

For $r\geq n$ note that the cylindric ribbon winds around the cylinder at least one time. Each such winding is represented by the action of $\tau$, which adds a circular ribbon of length $n$ and height $k$, before the remaining ribbon of length $r-n$ is added. Thus, we have the recurrence relation
\[
\chi_{\bar\lambda/d/\bar\mu}(r)=(-1)^{k-1}\chi_{\bar\lambda/(d-1)/\bar\mu}(r-n)\,.
\]
Repeatedly employing that $P^*_r=z^{-1}P^*_{r-n}$ one then proves the assertion by reducing it to the previous case of ribbons shorter than $n$. 
\end{proof}
\begin{defi}
Let $\bar\lambda,\bar\mu\in\balc$ and $d\geq 0$. Denote by $\bar\lambda'$ and $\bar\mu'$ the conjugate partitions and %
define the following element in the ring of symmetric functions $\Lambda$,
\begin{equation}\label{cylSchur}
s_{\bar\lambda'/d/\bar\mu'}=\sum_{\nu\in\cP^+}\frac{\chi_{\bar\lambda/d/\bar\mu}(\nu)}{\z_\nu}\,\epsilon_\nu p_\nu\,.
\end{equation}
\end{defi}

The latter definition can be interpreted as a cylindric or affine version of the Murnaghan-Nakayama rule: if $d=0$ we recover the skew Schur function $s_{\bar\lambda'/0/\bar\mu'}=s_{\bar\lambda'/\bar\mu'}$ in which case it is known that its inverse image under the characteristic map \eqref{ch} is the character of the $S_m$-module (with $m=|\bar\lambda|-|\bar\mu|$) obtained by generalising Young's representation to the skew tableau $\bar\lambda'/\bar\mu'$. The character values are then known to be equal to $\chi_{\bar\lambda'/0/\bar\mu'}(\nu)$.

We have chosen to parametrise \eqref{cylSchur} in terms of the conjugate partitions, because we will identify the function \eqref{cylSchur} with the (conjugate version of the) known cylindric (skew) Schur function in the literature using level-rank duality. The following lemma will allow us to do this.
\begin{lemma}\label{lem:quantumKostka}
Let $\bar\lambda,\bar\mu\in\balc$ and $\alpha,\beta\in\cP^+$ with $\alpha_1\leq n-k$ and $\beta_1\leq k$. Setting $z=(-1)^{k-1}q^{-1}$ we have the following equalities,
\begin{eqnarray}
\langle v^{\bar\lambda},H^*_\alpha v_{\bar\mu}\rangle=
\overline{\langle v^{\bar\lambda^{\vee}},H_\alpha v_{\bar\mu^{\vee}}\rangle}&=&q^{d}K_{\bar\lambda/d/\bar\mu,\alpha},\\
\langle v^{\bar\lambda},E^*_\beta v_{\bar\mu}\rangle=
\overline{\langle v^{\bar\lambda^{\vee}},E_\beta v_{\bar\mu^{\vee}}\rangle}&=&q^{d}K'_{\bar\lambda/d/\bar\mu,\beta},
\end{eqnarray}
where $K_{\bar\lambda/d/\bar\mu,\alpha}$ and $K'_{\bar\lambda/d/\bar\mu,\beta}$ are respectively the number of column and row strict cylindric reverse plane partitions of shape $\bar\lambda/d/\bar\mu$ with weights $\alpha$ and $\beta$. 
\end{lemma}
The numbers $K_{\bar\lambda/d/\bar\mu,\alpha}$ and $K'_{\bar\lambda/d/\bar\mu,\beta}$ are known to be identical with the {\em quantum Kostka numbers} which describe the (repeated) multiplication of a Schubert class with Chern classes of respectively the quotient and canonical bundle; see \cite[Props 3.1 and 4.3]{bertram1999quantum} and \cite[Cor 4.2]{postnikov2005affine} for details. Moreover, due to level-rank duality $qH^*(\op{Gr}(k,n))\cong qH^*(\op{Gr}(n-k,n))$  \cite[Prop 4.1]{bertram1999quantum}, one has $K_{\bar\lambda/d/\bar\mu,\alpha}=K'_{\bar\lambda'/d/\bar\mu',\alpha}$, where $\bar\lambda',\bar\mu'$ are the conjugate partitions of $\bar\lambda,\bar\mu$.
\begin{proof}
We employ the mentioned results \cite[Props 3.1 and 4.3]{bertram1999quantum} to prove the assertion. Let $\{\sigma_{\bar\lambda}\}_{\bar\lambda\in\balc}$ denote the basis of Schubert classes in $qH^*(\op{Gr}(k,n))$. Then we have $\sigma_r \sigma_{\bar\mu}=\sum_{\bar\lambda\in\balc}q^{d}\sigma_{\bar\lambda}$, where the sum runs over all $\bar\lambda\in\balc$ such that $\bar\lambda/d/\bar\mu$ is a cylindric horizontal strip and $d=(r+|\bar\mu|-|\bar\lambda|)/n$. Using the ring isomorphism $qH^*(\op{Gr}(k,n))\cong\V_k^-$ with $z=(-1)^{k-1}q^{-1}$ from Corollary \ref{cor:Vminus}, it then follows that $H^*_rv_{\bar\mu}= \sum_{\bar\lambda}q^{d}v_{\bar\lambda}$, where the sum runs over the same partitions $\bar\lambda$ as before. The general case now follows by repeated action with $H^*_{\nu_1}$, $H^*_{\nu_2}$, etc. Similarly, one shows the identities for $K'_{\bar\lambda/d/\bar\mu,\beta}$.
\end{proof}

\begin{theorem}\label{thm:CylSchurExpansion}
Set $q=(-1)^{k-1}z^{-1}$. Then one has the formal power series expansions
\begin{eqnarray}
\langle v^{\bar\lambda},\prod_{j\geq 1}E^*(u_j)v_{\bar\mu}\rangle =
\overline{\langle v^{\bar\lambda^{\vee}},\prod_{j\geq 1}E(u_j)v_{\bar\mu^{\vee}}\rangle} 
 &=& \sum_{d\geq 0}q^{d} s_{\bar\lambda'/d/\bar\mu'}(u)\;,
\end{eqnarray}
where $\bar q=q^{-1}$ and $\bar u_j=u_j$. In particular, the function \eqref{cylSchur} equals the (conjugate) cylindric Schur function. That is, we have the identities
\begin{equation}\label{cylSchurmatch} 
s_{\bar\lambda/d/\bar\mu}=\sum_{\nu_1\leq n-k}K'_{\bar\lambda'/d/\bar\mu',\nu}m_\nu=\sum_{\nu_1\leq n-k}K_{\bar\lambda/d/\bar\mu,\nu}m_\nu\;,
\end{equation}
where $\nu\in\cP^+$ and we have used level-rank duality in the second equality. 
\end{theorem}
Recall that under the specialisation $u_i=0$ for $i>k$ the cylindric Schur functions becomes the toric Schur polynomial from \cite{postnikov2005affine} which has the expansion \cite[Thm 5.3]{postnikov2005affine},
\begin{equation}
s_{\bar\lambda/d/\bar\mu}(u_1,\ldots,u_k)=\sum_{\bar\nu\in\balc}C_{\bar\mu\bar\nu}^{\bar\lambda,d}s_{\bar\nu}(u_1,\ldots,u_k),
\end{equation}
where $\langle v^{\bar\lambda},S_{\bar\nu}^*v_{\bar\mu}\rangle=q^d C_{\bar\mu\bar\nu}^{\bar\lambda,d}$ are the Gromov-Witten invariants of $qH^*(\op{Gr}(k,n))$ and the sum runs over all $\bar\nu$ such that $nd=|\bar\mu|+|\bar\nu|-|\bar\lambda|$. Here we have used the equality $C_{\bar\mu\bar\nu}^{\bar\lambda,d}=C_{\bar\mu'\bar\nu'}^{\bar\lambda',d}=\langle v^{\bar\lambda'},S_{\nu'}^*v_{\bar\mu'}\rangle$ and, thus, the isomorphism $qH^*(\op{Gr}(k,n))\cong qH^*(\op{Gr}(n-k,n))$, which can for example be derived from \eqref{BVI_formula}, \eqref{BVI2}.
\begin{proof}
Consider the subspace $\bV^-_k$ and take matrix elements in the Cauchy identities
\begin{equation}\label{Cauchy4}
\prod_{j\geq 1}E^*(u_j)=
\prod_{i=1}^k\prod_{j\geq 1}(1+X^*_i u_j)=\sum_{\lambda\in\cP^+}z^{-1}_\nu\epsilon_\nu P_\nu^*p_\nu(u)=\sum_{\nu\in\cP^+}E^*_\nu m_{\nu}(u)\;.
\end{equation}
Then we obtain from the second equality the power series expansions and from the third one -- with help of Lemma \ref{lem:quantumKostka} -- the asserted equality between \eqref{cylSchur} and the cylindric Schur function. Namely, observe that in $\End_{\RR}\bV'_k$ we have $E^*_r=0$ for $r>k$. Hence, we arrive at
\[
\langle v^{\bar\lambda},\prod_{j\geq 1}E^*(u_j)v_{\bar\mu}\rangle =
\sum_{\nu_1\leq k}\langle v^{\bar\lambda},E^*_\nu v_{\bar\mu}\rangle m_{\nu}(u)=
\sum_{d\geq 0}q^d s_{\bar\lambda'/d/\bar\mu'}(u)\;.
\]
Swapping the partitions with their conjugates, $\bar\lambda\to\bar\lambda'$ and $\bar\mu\to\bar\mu'$, and exchanging $k$ with $n-k$ in the above equalities, we arrive at the asserted identity with the cylindric Schur function upon employing level-rank duality, $K_{\bar\lambda/d/\bar\mu,\alpha}=K'_{\bar\lambda'/d/\bar\mu',\alpha}$. 
\end{proof}

We recall the following fact proved in \cite[Main Lemma]{bertram1999quantum}: let $\nu\in\cP^+_k$, then the image of the Schur polynomial $s_\nu\in\Lambda_k$ under the projection $\Lambda_k\twoheadrightarrow qH^*(\op{Gr}(k,n))$ is given by
\begin{equation}\label{ribbon_red}
s_\nu\mapsto\left\{
\begin{array}{cl}
(-1)^{kr-\sum_{i=1}^r\op{ht}(\rho_i)}q^r s_{\dot\nu}, &\dot \nu_1\leq n-k\\
0, &\text{else}
\end{array}\right.\,,
\end{equation}
where $\dot\nu$ is the $n$-core of the partition $\nu$ and the $\rho_i$ are a sequence of $n$-ribbons removed from $\nu$ to obtain $\dot\nu$. The number $r$ of $n$-ribbons removed defines the so-called $n$-weight of $\nu$. Note that while the sequence of $n$-ribbons is not unique, the final result will not depend on the particular choice of the $\rho_i$, see e.g. \cite{james1981representation}. As in our previous discussion of the cylindric Schur function, our version of the quotient map \eqref{ribbon_red} is related to the one in \cite{bertram1999quantum} by level-rank duality, that is, we take the conjugate partitions instead in order to emphasise the similarity to the Murnaghan-Nakayama rule.
\begin{coro}\label{cor:CylSchur2Schur}
Let $\bar\lambda,\bar\mu\in\balc$ and $\nu\in\cP^+_k$. We have the following identity for matrix elements,
\begin{equation}
\langle v^{\bar\lambda},S^*_{\nu}v_{\bar\mu}\rangle=\left\{
\begin{array}{cl}
(-1)^{kr-\sum_{i=1}^r\op{ht}(\rho_i)}q^{d} C_{\bar\mu\dot\nu}^{\bar\lambda,d-r}, &\dot \nu_1\leq n-k\\
0, &\text{else}
\end{array}\right.\,, 
\end{equation} 
where $nd=|\bar\mu|+|\nu|-|\bar\lambda|$. In particular, setting $q=1$ we have the expansion
\begin{equation}\label{cylschur2schur}
s_{\bar\lambda'/d/\bar\mu'}=\sum_{\nu}\langle v^{\bar\lambda},S^*_{\nu}v_{\bar\mu}\rangle s_{\nu'}
\end{equation}
of cylindric Schur functions into Schur functions, where the sum runs over those $\nu\in\cP_k^+$ for which $|\nu|=|\bar\lambda|-|\bar\mu|+dn$.
\end{coro}
\begin{proof}
Starting point is the last equality in the Cauchy identity \eqref{Cauchy2}. Taking matrix elements in the subspace of alternating tensors and using the ring isomorphism from Corollary \ref{cor:Vminus} the assertion follows from \eqref{ribbon_red}. 
\end{proof}

The expansion formula from Corollary \ref{cor:CylSchur2Schur} shows that in general cylindric Schur function fail to be Schur positive. This was already pointed out by McNamara \cite[Thm 6.5]{mcnamara2006cylindric}. 
McNamara conjectured \cite[Conjecture 7.3]{mcnamara2006cylindric} that a general (skew) cylindric Schur function $s_{\bar\lambda/d/\bar\mu}$  has non-negative expansion coefficients in (non-skew) cylindric Schur functions of the type $s_{\bar\lambda/d/\emptyset}$ with $\bar\lambda\in\balc$ and $d\geq 0$, which he showed to be linearly independent \cite[Prop. 7.1]{mcnamara2006cylindric}.  In order to facilitate the comparison we again employ level-rank duality and exchange $k$ with $n-k$ and partitions with their conjugates.

\begin{coro}\label{cor:nscylschur} 
Let $\cP_{\bar\lambda}(d)\subset\cP_{n-k}^+$ denote the subset of all partitions $\nu\in\cP_{n-k}^+$, whose $n$-core is $\dot\nu=\bar\lambda'$ and whose $n$-weight, the number of $n$-ribbons $\rho_i$ removed to obtain $\dot\nu$, equals $d$. Then
\begin{equation}\label{cylnsSchur2Schur}
s_{\bar\lambda/d/\emptyset}=\sum_{\nu\in\cP_{\bar\lambda}(d)}(-1)^{d(n-k)-\sum_{i=1}^d\op{ht}(\rho_i)} s_{\nu'}\,,
\end{equation}
and taking the Hall inner product \eqref{Hall} in $\Lambda$ we obtain
\begin{equation}
\langle s_{\bar\lambda/d/\emptyset},s_{\bar\mu/d'/\emptyset}\rangle =\delta_{\bar\lambda\bar\mu}\delta_{dd'}|\cP_{\bar\lambda}(d)|\,,
\end{equation}
showing that the cylindric (non-skew) Schur functions are orthogonal and, hence, linearly independent.
\end{coro}
\begin{proof}
The expansion formula is immediate upon setting $\bar\mu=\emptyset$ in \eqref{cylschur2schur}. The computation of the Hall inner product then follows by observing that the $n$-core of a partition is unique and that Schur functions are an orthonormal basis in $\Lambda$ with respect to the Hall inner product. 
\end{proof} 

A proof of McNamara's conjecture was recently put forward by Lee \cite[Theorem 1]{lee2019positivity} using the connection between cylindric Schur functions and affine Stanley symmetric functions pointed out by Lam in \cite{lam2006affine}. Here we give an alternative, somewhat shorter proof of McNamara's conjecture as a direct consequence of Theorem \ref{thm:CylSchurExpansion} and show that the expansion coefficients are given by the Gromov-Witten invariants $C_{\bar\mu\bar\nu}^{\bar\lambda,d}$.

\begin{coro}\label{cor:mcnamara_conj}
Let $\bar\lambda,\bar\mu\in\balc$ and $d\geq 0$. Then we have the expansion
\begin{equation}\label{mcnamara_conj}
s_{\bar\lambda/d/\bar\mu}=\sum_{d'=0}^d\sum_{\bar\nu}
C^{\bar\lambda,d'}_{\bar\mu\bar\nu}\,s_{\bar\nu/(d-d')/\emptyset}\,,
\end{equation}
where the sum runs over all $\bar\nu\in\balc$ such that $|\bar\nu|=|\bar\lambda|+nd'-|\bar\mu|$.
\end{coro}
\begin{proof}
The proof is analogous to the one from Corollary \ref{cor:cylh2cylh}. Taking matrix elements in the subspace $\bV^-_k$ of alternating tensors in the last equality of the generalised Cauchy identity \eqref{Cauchy2} we obtain from Theorem \ref{thm:CylSchurExpansion},
\begin{eqnarray*}
\langle v^{\bar\lambda},\prod_{i\geq 0}E^*(u_i)v_{\bar\mu}\rangle &=& \sum_{d\geq 0}q^d s_{\bar\lambda'/d/\bar\mu'}(u)
=\sum_{\sigma\in\cP^+_k}\langle v^{\bar\lambda},S^*_{\sigma}v_{\bar\mu}\rangle s_{\sigma'}(u)\\
&=&\sum_{\sigma\in\cP^+_k}\langle v^{\bar\lambda},S^*_{\sigma}S^*_{\bar\mu}v_{\emptyset}\rangle s_{\sigma'}(u)
=\sum_{\bar\nu\in\balc}\langle v^{\bar\lambda},S^*_{\bar\mu}v_{\bar\nu}\rangle \,
\sum_{\sigma\in\cP^+_k}\langle v^{\bar\nu},S^*_{\sigma}v_{\emptyset}\rangle s_{\sigma'}(u)\\
&=&\sum_{\bar\nu\in\balc}\langle v^{\bar\lambda},S^*_{\bar\mu}v_{\bar\nu}\rangle \,\sum_{d''\geq 0} q^{d''} s_{\bar\nu'/d''/\emptyset}(u)\,. 
\end{eqnarray*}
Here we have used in the second line the product identity $v_{\bar\mu}=v_{\bar\mu}v_{\emptyset}=S^*_{\bar\mu}v_{\emptyset}$ in $\V^-_k$; compare with Corollary \ref{cor:Vminus}. Comparing powers in $q$ we obtain the asserted expansion upon employing once more level rank duality, $\langle v^{\bar\lambda},S^*_{\bar\mu}v_{\bar\nu}\rangle
=q^{d'}C^{\bar\lambda,d'}_{\bar\mu\bar\nu}=q^{d'}C^{\bar\lambda',d'}_{\bar\mu'\bar\nu'}$, where $d'$ is fixed by the relation stated in the corollary.
\end{proof}
By a completely analogous argument as in the proof of \eqref{DeltaCylh}, one derives the following coproduct formula for cylindric Schur functions in the Hopf algebra $\Lambda$:
\begin{coro}\label{cor:cylschurcoalg} 
Let $\bar\lambda,\bar\mu\in\balc$ and $d\in\mb{Z}_{\ge 0}$. Then
\begin{equation}\label{copcylschur}
\Delta(s_{\bar\lambda/d/\bar\mu})=\sum_{d=d_1+d_2}\sum_{\bar\nu\in\balc}s_{\bar\lambda/d_1/\bar\nu}\otimes s_{\bar\nu/d_2/\bar\mu}\,.
\end{equation}
In particular, specialising $\bar\mu=\emptyset$ it follows from \eqref{mcnamara_conj} that the subspace spanned by the cylindric non-skew Schur functions $\{s_{\bar\lambda/d/\emptyset}~|~\bar\lambda\in\balc,\;d\geq 0\}$ is a positive subcoalgebra of $\Lambda$.
\end{coro}
 
Note that when setting $d=0$ the resulting finite-dimensional subspace is spanned by the Schur functions $s_{\bar\lambda}$ with $\bar\lambda\in\balc$ and we see from the coproduct formula \eqref{copcylschur} that only (non-cylindric)  skew Schur functions can occur. We thus recover the statement from the introduction, the resulting finite-dimensional coalgebra is isomorphic to $H^*(\op{Gr}(k,n))$ (viewed as a coalgebra).

\appendix 

\section{The ring of symmetric functions}\label{AppA}
This appendix first recalls some known results about the ring of symmetric functions and then proves some formulae for weighted sums over reverse plane partitions (as well as for the subclasses of them)  that we need for the generalisation to cylindric symmetric functions in the main text and which we were unable to find in the literature.

\subsection{Symmetric functions as Hopf algebra}

Recall the definition of the ring of symmetric functions $\Lambda=\mb{C}[p_1,p_2,\ldots]$, where $p_r=\sum_{i\geq 1}x_i^r$ are the power sums in some commuting indeterminates $x_i$. While we shall work over the complex numbers, we note that the power sums $\{p_\lambda\}_{\lambda\in\cP^+}$ with $p_\lambda= p_{\lambda_1} p_{\lambda_2} \cdots$ and $\cP^+$ the set of partitions, form a $\mb{Q}$-basis. The latter basis is distinguished as it provides a link with the Frobenius character map. Namely, recall that $\Lambda=\bigoplus_{m\geq 0}\Lambda^m$ with $\Lambda^m$ denoting the set of homogeneous symmetric functions of degree $m$, is a graded ring. Given a permutation $w\in S_m$ in the symmetric group on $m$ letters, denote by $\mu(w)$ the partition fixed by the cycle type of $w$. Then the map $\op{ch}:\op{Class}(S_m)\to \Lambda^m$ that sends a class function $f$ of the symmetric group $S_m$ to a homogeneous function of degree $m$, 
\begin{equation}\label{ch}
\op{ch}(f)=\sum_{w\in S_m}\frac{f(w)}{m!}\,p_{\mu(w)}=\sum_{\mu\vdash m}\frac{f(\mu)}{\z_\mu}\,p_\mu,
\end{equation}
is known as {\em characteristic map}. Here $f(\mu)$ denotes the value of the class function on the class fixed by the partition $\mu$ and the constant 
\begin{equation}\label{z}
\z_\lambda = \prod_{i \ge 1} i^{m_i(\lambda)} m_i(\lambda)!
\end{equation}
is related to the size of the conjugacy class in $S_m$ fixed by $\lambda\vdash m$.  Here $m_i(\lambda)=\lambda_i'-\lambda'_{i+1}$ denotes the multiplicity of $i$ in $\lambda$ and $\lambda'$ the conjugate partition.

The set of class functions, which can be identified with the centre of the group algebra of $S_m$, carries a natural inner product, $\langle f,g\rangle=(1/m!)\sum_{w\in S_m} f(w)\overline{g(w)}$. The latter induces the Hall inner product on $\Lambda$ by requiring that $\op{ch}$ is an isometry. In terms of the power sums the Hall inner product then reads,
\begin{equation}\label{Hall}
\langle p_\lambda,p_\mu\rangle=\delta_{\lambda\mu}\z_\lambda\,.
\end{equation}
Once the Hall inner product is in place, one can introduce a coproduct $\Delta: \Lambda \to \Lambda \otimes \Lambda$ via the relation 
\begin{equation}\label{cop}
\langle \Delta f, g \otimes h \rangle = \langle f, gh \rangle,\qquad f,g,h\in\Lambda,
\end{equation} 
compare with \cite[Ch.I.5, Ex. 25]{macdonald1998symmetric}. The power sums are primitive elements with respect to this  coproduct, $\Delta(p_r)=p_r\otimes 1+1\otimes p_r$. Endowing $\Lambda$ in addition with an antipode $\gamma:\Lambda\to\Lambda$, $\gamma(p_r)=-p_r$ and co-unit $\veps(p_r)=0$, the ring of symmetric functions becomes a (self-dual) Hopf algebra \cite{zelevinsky1981representations}. Note that the antipode is closely related to the involutive ring automorphism $\omega: \Lambda \to \Lambda$ with $\omega(p_r)=(-1)^{r-1}p_r$, which is also an isometry of the Hall inner product; see \cite[Ch.I.2 and I.4]{macdonald1998symmetric}.

\subsection{The different bases in the ring of symmetric functions}
Besides the power sums there are also the following $\mb{Z}$-bases of $\Lambda$ which we will need for our discussion:

\subsubsection{Monomial and forgotten symmetric functions}
Set $m_\lambda=\sum_{\alpha\sim\lambda} x^\alpha$ where the notation $\alpha\sim\lambda$ means that $\alpha$ is a {\em distinct} permutation of $\lambda\in\cP^+$. These are the {\em monomial symmetric functions}. Their images under the antipode, $f_\lambda=(-1)^{|\lambda|}\gamma(m_\lambda)$ are known as the {\em forgotten symmetric functions}.

\subsubsection{Complete and elementary symmetric functions}
The dual basis of $\{m_\lambda\}_{\lambda\in\cP^+}$ with respect to the Hall inner product \eqref{Hall} is known as the {\em complete symmetric functions} $\{h_\lambda\}_{\lambda\in\cP^+}$, i.e. $\langle m_\lambda,h_\mu\rangle =\delta_{\lambda\mu}$. Their explicit definition is $h_\lambda=h_{\lambda_1}h_{\lambda_2}\cdots$, where $h_r=\sum_{1\leq i_1\leq\cdots\leq i_r}x_{i_1}x_{i_2}\cdots x_{i_r}$. Applying the antipode one obtains the {\em elementary symmetric functions}, $\gamma(h_\lambda)=(-1)^{|\lambda|}e_\lambda$. Explicitly, $e_\lambda=e_{\lambda_1}e_{\lambda_2}\cdots$ with $e_r=\sum_{1< i_1<\cdots< i_r}x_{i_1}x_{i_2}\cdots x_{i_r}$.

\subsubsection{Schur functions}
Denote by $s_\lambda$ the image of the character of the Specht module of $S_m$ indexed by a partition $\lambda\vdash m$ under the characteristic map \eqref{ch}. This is the {\em Schur function}. The latter has the known determinant relations $s_\lambda=\det(h_{\lambda_i-i+j})_{1\leq i,j\leq\ell(\lambda)}=\det(e_{\lambda'_i-i+j})_{1\leq i,j\leq\ell(\lambda')}$, where $\lambda'$ is the conjugate partition of $\lambda$ and $\ell(\lambda)$ its length, i.e. the number of nonzero parts; see \cite[Ch.I.2]{macdonald1998symmetric} for further details.

\subsection{Generalised Pieri type rules and the coproduct}
We now employ the Hopf algebra structure on $\Lambda$ to define `skew complete symmetric functions' $h_{\lambda/\mu}$ and `skew elementary symmetric functions' $e_{\lambda/\mu}$ via the coproduct expansions \eqref{skewhe} from the introduction. It follows at once from \eqref{Hall} and \eqref{cop}  that both functions are determined by the three product formulae
\begin{align} \label{pierirule} 
m_\mu m_\nu=\sum_{\lambda}f_{\mu\nu}^{\lambda }m_{\lambda},\quad
 m_\mu  h_\nu=\sum_{\lambda} \theta_{\lambda/\mu}(\nu)m_\lambda, \quad
  m_\mu e_\nu=\sum_{\lambda} \psi_{\lambda/\mu}(\nu)m_\lambda \,,
\end{align}
which uniquely fix the coefficients. The following result is immediate; compare with the discussion of skew Schur functions in \cite[Ch. I]{macdonald1998symmetric}.

\begin{lemma} \label{lem:ExpansionHE}
 The functions $h_{\lambda/\mu}$ and $e_{\lambda/\mu}$ can be expanded as
 \begin{eqnarray}
   h_{\lambda/\mu} &=& \sum_{\nu}f_{\mu\nu}^\lambda h_\nu
 =\sum_{\nu}\theta_{\lambda/\mu}(\nu)m_\nu \label{skewh} \;,  \\
  e_{\lambda/\mu} &=& \sum_{\nu}f_{\mu\nu}^\lambda e_\nu
 =\sum_{\nu}\psi_{\lambda/\mu}(\nu)m_\nu \label{skewe} \;.
 \end{eqnarray}
\end{lemma}

\begin{proof}
 Let $g \in \Lambda$ then it follows from \eqref{cop} that $\langle h_{\lambda/\mu},g \rangle = \langle h_\lambda,m_\mu g \rangle$
 and $\langle e_{\lambda/\mu},g \rangle = \langle e_\lambda, f_\mu g \rangle$. Because 
 \[
 h_{\lambda/\mu}=\sum_{\nu} \langle h_{\lambda/\mu}, m_\nu \rangle  h_\nu=\sum_{\nu} \langle h_{\lambda/\mu}, h_\nu \rangle  m_\nu
 \]
one deduces, using \eqref{pierirule}, the first and second equality of \eqref{skewh} respectively. The proof of equation \eqref{skewe} is analogous.
\end{proof}

All of the expansion coefficients in \eqref{pierirule} are non-negative integers and have combinatorial expressions, which we know recall. For the product of monomial symmetric functions the expansion coefficients
$f_{\mu\nu}^\lambda$ in \eqref{pierirule} are given by the cardinality of the set 
\begin{equation}\label{Setf}
\{(\alpha,\beta)~|~\alpha+\beta=\lambda\},
\end{equation} 
where $\alpha\sim \mu$ and $\beta \sim \nu$ are distinct permutations of respectively $\mu$ and $\nu$; see e.g. \cite{butler1993nonnegative}. This also fixes the expansion coefficients for the second and third product in \eqref{pierirule}. 

Recall that $h_\lambda=\sum_{\mu}L_{\lambda\mu}m_\mu$, where $L_{\lambda\mu}$ 
is the number of $\mathbb{N}$-matrices $A=(a_{ij})_{i,j\geq1}$ 
which have row sums $\text{row}(A)=\lambda$ and column sums $\text{col}(A)=\mu$.
Similarly $e_\lambda=\sum_{\mu} M_{\lambda \mu} m_\mu$, where $M_{\lambda  \mu}$
is the number of $(0,1)$-matrices $A=(a_{ij})_{i,j \ge 1}$ with $\text{row}(A)=\lambda$
and $\text{col}(A)=\mu$; see e.g. \cite[Ch. 7.4, Prop 7.4.1 and Ch.7.5, Prop 7.5.1]{stanley_fomin_1999}. Then it follows at once that
\begin{equation}\label{skewThetaPsi}  
\theta_{\lambda/\mu}(\nu)= \sum_{\sigma}L_{\nu \sigma}  f_{ \sigma \mu}^\lambda \qquad\text{and}\qquad
\psi_{\lambda/\mu}(\nu) = \sum_\sigma M_{\nu \sigma}  f_{ \sigma \mu}^\lambda  \;.
\end{equation}
Here we are interested in alternative expressions using reverse plane partitions as the latter suggest a natural generalisation to cylindric reverse plane partitions.

\subsection{Weighted sums over reverse plane partitions}
Given two partitions $\lambda,\mu$ with $\mu\subset \lambda$ recall that a {\em reverse plane partition} (RPP) $\pi$ of skew shape $\lambda/\mu$ is a sequence $\{\lambda^{(i)}\}_{i=0}^l$ of partitions $\lambda^{(i)}$ with %
\[
\mu=\lambda^{(0)}\subset\lambda^{(1)}\subset\cdots\subset\lambda^{(l)}=\lambda \, .
\] %
As usual we refer to the vector $\op{wt}(\pi)=(|\lambda^{(1)}/\lambda^{(0)}|,|\lambda^{(2)}/\lambda^{(1)}|,\ldots)$ as the {\em weight} of $\pi$ and denote by $x^\pi$ the monomial $x_1^{\op{wt}_1(\pi)}x_2^{\op{wt}_2(\pi)}\cdots$ in the indeterminates $x_i$. 
Alternatively, we can think of a RPP as a map $\pi:\lambda/\mu\to\mathbb{N}$ which assigns to the squares $s=(x,y)\in \lambda^{(i)}/\lambda^{(i-1)}\subset \mathbb{Z}\times\mathbb{Z}$ the integer $i$. The result is a skew tableau whose entries are non-decreasing along each row from left to right and down each column. 

Special cases of reverse plane partitions that are {\em row} or {\em column strict} will be called {\em tableaux} and denoted by $T$. A row (column) strict RPP $T$ of skew shape $\lambda/\mu$ 
is a sequence $\{\lambda^{(i)}\}_{i=0}^l$ of partitions $\lambda^{(i)}$ such that $\lambda^{(i)}/\lambda^{(i-1)}$, for $i=1,\dots,l$, is a vertical (horizontal) strip (see for example \cite[Ch.I.1]{macdonald1998symmetric}). As in the case of general RPP, we will refer to the weight of $T$ as $\op{wt}(T)$ and denote by $x^T$
the monomial $x_1^{\op{wt}_1(T)} x_2^{\op{wt}_2(T)} \cdots$ in the indeterminates $x_i$.

We now introduce weighted sums over RPP in terms of binomial coefficients. 
For this purpose we generalise the notion of skew diagrams as given for example in \cite[Ch.I.1]{macdonald1998symmetric}. For two compositions $\alpha$ and $\beta$ 
 we write $\alpha \subset \beta$ if 
$\alpha_i \le \beta_i$ for all $i$, and we refer to the set $\beta/\alpha \subset \mathbb{Z} \times \mathbb{Z}$ as a (generalised) skew diagram. 

Given any pair of partitions $\lambda,\mu$ denote by $\theta_{\lambda/\mu}$ 
the cardinality of the set
\begin{equation} \label{SetTheta}
  \{ \alpha \sim\mu ~|~ \alpha \subset \lambda \} \;,
\end{equation}
where we recall that $\alpha \sim\mu$ indicates that $\alpha$ is a distinct permutation of $\mu$.

\begin{lemma} \label{lem:BinomialTheta}
The set \eqref{SetTheta} is non-empty if and only if $\lambda/\mu$ is a skew diagram, i.e. if $\mu\subset\lambda$. In the latter case its cardinality is given by 
\begin{equation}\label{theta} 
 \theta_{\lambda / \mu} = \prod_{i \ge 1} \binom{\lambda'_i - \mu'_{i+1}}{\mu'_i - \mu'_{i+1}}\;.
\end{equation}
\end{lemma}
Here and in the following we set a binomial coefficient equal to zero if one of its arguments is negative. It then follows that the right hand side of \eqref{theta} vanishes if $\mu\not\subset\lambda$ because $\mu\subset\lambda$ if and only $\mu'\subset\lambda'$, where $\mu',\lambda'$ are the conjugate partitions.
\begin{proof}
 Assume first that $\mu \subset \lambda$ then we have that  $ \alpha=\mu $ belongs to \eqref{SetTheta}. Conversely assume
 that \eqref{SetTheta} is non-empty, then there must exist $\alpha \sim \mu$ such that $\alpha \subset \lambda$. But since $\alpha\sim\mu$ and $\mu\in\cP^+$ this is only possible if $\mu\subset\lambda$. 
 
Thus, let $\mu \subset \lambda$ and set $\ell =\ell (\lambda')=\lambda_1$, the largest part of $\lambda$, which equals the length of $\lambda'$. Because $\mu_1=\ell (\mu')\leq \ell $, the $m_\ell (\mu)=\mu'_\ell $ parts of $\mu$ equal to $\ell $ (if they exist) must be among  the first $m_\ell (\lambda)=\lambda_\ell '$ parts of $\alpha$, thus there are 
$\binom{\lambda_\ell '}{\mu_\ell '}$ possible choices with $\binom{\lambda_\ell '}{\mu_\ell '}=1$ should $m_\ell (\mu)=0$. Continuing with the (possible) second largest part $\ell -1$, the  $m_{\ell -1}(\mu)=\mu'_{\ell -1}-\mu'_\ell $
parts of $\mu$ equal to $\ell -1$ must be among the first $m_{\ell -1}(\lambda)+m_\ell (\lambda)=\lambda_{\ell -1}'$
 parts of $\alpha$, and since we have already fixed the $\mu_\ell '$ parts of $\alpha$ equal to $\ell $, there are now only ${{\lambda'_{\ell -1}-\mu'_\ell } \choose {\mu'_{\ell -1}-\mu'_\ell }} $ possibilities. Proceeding along the same vein we eventually arrive at the desired cardinality of \eqref{SetTheta} after a trivial rewriting of the following product,
\[
 \prod_{i=1}^\ell  {{\lambda_{\ell +1-i}'-\mu'_{\ell +2-i}} \choose {\mu'_{\ell +1-i}-\mu'_{\ell +2-i}}}
 = \prod_{i \ge 1} {{\lambda_i' - \mu_{i+1}'} \choose {\mu_i'-\mu_{i+1}'}} \;.
\]
\end{proof}

The following result is probably known to experts but we were unable to find it in the literature.
\begin{lemma} \label{lem:SkewComplete}
The skew complete symmetric function defined in \eqref{skewhe} is the weighted sum
 \begin{equation} \label{skewh2RPP}
  h_{\lambda / \mu} (x)= \sum_{\pi} \theta_\pi \, x^\pi \;,
  \qquad \theta_\pi=\prod_{i\ge 1} \theta_{\lambda^{(i)}/\lambda^{(i-1)}} \;,
 \end{equation}
over all reverse plane partitions $\pi$ of shape $\lambda / \mu$.
In particular, the first coefficient in \eqref{skewThetaPsi} has the alternative expression
\begin{equation} \label{Theta2RPP}
 \theta_{\lambda/\mu}(\nu)=\sum_{\pi} \theta_\pi
\end{equation}
where the sum is over all reverse plane partitions $\pi$ of shape $\lambda/\mu$
and weight $\op{wt}(\pi)=\nu$.
\end{lemma}

\begin{proof}
First we show the validity of the product expansion
 \begin{equation*}
  m_\mu h_r = \sum_{\lambda} \theta_{\lambda/\mu} m_\lambda \;,
 \end{equation*}
 where the sum runs over all partitions $\lambda$ such that $\mu \subset \lambda$ and $|\lambda / \mu|=r$.  
  The coefficient of $m_\lambda$ in $m_\mu h_r$ equals the coefficient of  $x^\lambda$ 
in the same product. If $\alpha \sim \mu$ and $\alpha \subset \lambda$  then the monomial $x^{\lambda - \alpha}$ 
appears in $h_r$ provided that $|\lambda/\mu|=r$.
Thus, the coefficient of $x^\lambda$ in $m_\mu h_r$ equals the cardinality of the set \eqref{SetTheta},
which is the definition of $\theta_{\lambda/\mu}$. Applying this result to the product $m_\mu h_\nu$ and comparing with 
the second equality in \eqref{pierirule} one arrives at \eqref{Theta2RPP}. 

Note that this also implies that $\theta_{\lambda/\mu}(\nu)=\theta_{\lambda/\mu}(\beta)$ for $\beta \sim \nu$,
  where $\theta_{\lambda/\mu}(\beta)$ for $\beta$ a composition is defined analogously to \eqref{Theta2RPP}.
Hence, after rewriting the equality $h_{\lambda/\mu} = \sum_{\nu}\theta_{\lambda/\mu}(\nu)m_\nu$ proved in Lemma \ref{lem:ExpansionHE} in terms of monomials one obtains \eqref{skewh2RPP}.
\end{proof}
We now present a similar argument for skew elementary symmetric functions.
First we generalise the notion of a vertical strip from partitions 
to compositions by saying that for two compositions $\alpha$ and $\beta$, the generalised skew diagram $\beta/\alpha$ is a generalised vertical strip if $\beta_i-\alpha_i=0,1$ for all $i$. Given any pair of partitions $\lambda,\mu$ denote by $\psi_{\lambda/\mu}$ 
the cardinality of the set
\begin{equation} \label{SetPsi}
  \{ \alpha \sim \mu ~|~  \lambda/\alpha \text{ vertical strip} \} \;.
\end{equation}

\begin{lemma} \label{lem:CardinalityPsi}
The set \eqref{SetPsi} is non-empty if and only if $\lambda/\mu$ is a vertical strip. In the latter case its cardinality equals
\begin{equation} \label{Psi}
 \psi_{\lambda / \mu} = \prod_{i \ge 1} {{\lambda'_i - \lambda'_{i+1}} \choose {\lambda'_i - \mu'_i}} \;.
\end{equation}
\end{lemma}

\begin{proof}
The proof of the first part is analogous to the one for the set \eqref{SetTheta}, so we skip this step and assume that $\lambda/\mu$ is a vertical strip.

In a similar fashion to the proof of Lemma \ref{lem:BinomialTheta} we count the number of distinct permutations $\alpha$ of $\mu$ such that
$\lambda/\alpha$ is a vertical strip. Set again $\lambda_1=\ell(\lambda')=\ell$. As before the $m_\ell(\mu)=\mu_\ell'$ parts of $\mu$ equal to $\ell$ must be among the first $m_\ell(\lambda)=\lambda_\ell'$ parts of $\alpha$ and there are $\binom{\lambda_\ell'} {\mu_\ell'}$ possible choices. Since $\lambda/\alpha$ is a vertical strip, the remaining  parts $\alpha_i$ with $\mu'_\ell\leq i\leq\lambda'_\ell$, must be equal to $\ell-1$. Hence, $\mu_{\ell-1}(\mu) - (\lambda_\ell'-\mu_\ell') =\mu_{\ell-1}'-\lambda_\ell'$ parts of $\mu$ equal to $\ell-1$ must be parts $\alpha_i$ with $\lambda_\ell'+1\leq i\leq\lambda'_\ell+m_{\ell-1}(\lambda)=\lambda'_{\ell-1}$, and there are ${{\lambda_{\ell-1}'-\lambda_\ell'} \choose {\mu_{\ell-1}'-\lambda_\ell'}}$ distinct ways of arranging these. The other $\alpha_i$, with $i$ in the range just stated, must be equal to $\ell-2$ and there are $\lambda_{\ell-1}'-\lambda_\ell'-(\mu_{\ell-1}'-\lambda_\ell')=\lambda'_{\ell-1}-\mu'_{\ell-1}$ of them. Continuing in the same manner, we eventually arrive at
\[
 \prod_{i=1}^{\ell} {{\lambda_{\ell+1-i}'-\lambda'_{\ell+2-i}} \choose {\mu_{\ell+1-i}'-\lambda'_{\ell+2-i}}}
 = \prod_{i=1}^\ell {{\lambda_{\ell+1-i}'-\lambda'_{\ell+2-i}} \choose {\lambda'_{\ell+1-i}-\mu'_{\ell+1-i}}}
 = \prod_{i \ge 1} {{\lambda_i' - \lambda_{i+1}'} \choose {\lambda_i'-\mu_i'}} \;,
\]
and this is equal to $\psi_{\lambda/\mu}$, completing the proof.
\end{proof}
We were unable to find the following result in the literature.
\begin{lemma}
The skew elementary symmetric function defined in \eqref{skewhe} is the weighted sum
 \begin{equation} \label{e2T}
  e_{\lambda / \mu} (x)= \sum_{T} \psi_T \, x^T \;,
  \qquad \psi_T=\prod_{i\ge 1} \psi_{\lambda^{(i)}/\lambda^{(i-1)}} \;,
 \end{equation}
over all row strict RPP $T$ of shape $\lambda / \mu$. In particular, the second coefficient in \eqref{skewThetaPsi} has the alternative expression
\begin{equation} \label{Psi2T}
 \psi_{\lambda/\mu}(\nu)=\sum_{T} \psi_T
\end{equation}
where the sum is over all row strict RPP $T$ of shape $\lambda/\mu$ and weight $\op{wt}(T)=\nu$.
\end{lemma}
\begin{proof}
 We first need to show that  $m_\mu e_r = \sum_{\lambda} \psi_{\lambda/\mu} m_\lambda$, where the sum runs over all partitions $\lambda$ such that $\lambda / \mu$ is a vertical strip and $|\lambda / \mu|=r$. The coefficient of $m_\lambda$ in $m_\mu e_r$ equals the coefficient of the monomial term $x^\lambda$ in the same product. If $\alpha \sim \mu$ and $\lambda / \alpha$ is a vertical strip,  then the monomial $x^{\lambda - \alpha}$ appears in $e_r$ provided $|\lambda/\mu|=r$. Thus the coefficient of $x^\lambda$ in $m_\mu e_r$ equals the cardinality of the set \eqref{SetPsi}, which by definition is $\psi_{\lambda/\mu}$. Applying this result to the product $m_\mu e_\nu$ and comparing with the third product expansion in \eqref{pierirule} the formula \eqref{Psi2T} follows. Note that this implies that $\psi_{\lambda/\mu}(\nu)=\psi_{\lambda/\mu}(\beta)$ for $\beta \sim \nu$, where $\psi_{\lambda/\mu}(\beta)$ for $\beta$ a composition is defined analogously to  \eqref{Psi2T}. Expanding $e_{\lambda/\mu} = \sum_{\nu}\psi_{\lambda/\mu}(\nu)m_\nu$, proved in Lemma \ref{lem:ExpansionHE} into monomials, using \eqref{Psi2T} one arrives at \eqref{e2T}.
\end{proof}

\begin{coro} \label{LMtableaux}
 The matrices $L$ and $M$ in \eqref{skewThetaPsi} have the following alternative expressions,
 \begin{equation}\label{LMtoThetaPsi}
  L_{\lambda \mu} = \sum_\pi \theta_\pi \qquad\text{and}\qquad 
  M_{\lambda \mu} = \sum_T \psi_T,
 \end{equation}
 where the first sum runs over all RPP $\pi$ of shape $\lambda$ and weight
 $\op{wt}(\pi)=\mu$, whereas the second sum runs over all row strict RPP $T$ of shape $\lambda$ and weight $\op{wt}(T)=\mu$.
\end{coro}

\begin{proof}
 Set $\mu=\emptyset$ in \eqref{skewThetaPsi}. Employing that $f_{\sigma \emptyset}^\lambda= \delta_{\lambda \sigma}$ we can apply  \eqref{Theta2RPP}, \eqref{Psi2T} together with $L_{\lambda \mu}=L_{\mu \lambda}$ and $M_{\lambda \mu}=M_{\mu \lambda}$ to arrive at the asserted formulae.
\end{proof}

\subsection{The expansion coefficients in terms of the symmetric group}

To motivate our construction of cylindric functions in the main text using the extended affine symmetric group $\hat \rS_k$ we reformulate here the results on the expansion coefficients in \eqref{pierirule} in terms of the symmetric group $\rS_k$. To this end we project onto $\Lambda_k=\mb{Z}[x_1,\ldots,x_k]^{\rS_k}$ by setting $x_i=0$ for $i>k$.

Recall that for a fixed weight $\mu$ we denote by $\rS_\mu\subset \rS_k$ its stabiliser group and that for any permutation $w\in \rS_k$ there exists a unique decomposition $w=w_\mu \circ w^\mu$ with $w_\mu\in \rS_\mu$ and $w^\mu$ a minimal length representative of the right coset $\rS_\mu w$. As before $\rS^\mu\subset \rS_k$ is the set of all minimal length coset representatives in $\rS_\mu\backslash \rS_k$.
\begin{lemma}
Let $\lambda,\mu,\nu\in\cP_k^+$. 
\begin{itemize}
\item[(i)] The expansion coefficient $f_{\mu\nu}^\lambda$ in \eqref{pierirule} can be expressed as the cardinality of the set
\begin{equation}\label{setf}
\{(w,w')\in \rS^\mu\times \rS^\nu~|~\mu\circ w+\nu\circ w' =\lambda\}\;.
\end{equation}
\item[(ii)] The specialisation of the weight factor $\theta_{\lambda/\mu}$ defined in \eqref{theta} to elements in $\cP_k^+$ equals the cardinality of the set
 \begin{equation} \label{setTheta}
   \{ w \in \rS^\mu~|~ \mu\circ w \le \lambda \}\,,
 \end{equation}
where $\mu\circ w\le \lambda$ is shorthand notation for $\mu_{w(i)} \le \lambda_i$ for all $i\in [k]$. 
\item[(iii)] The specialisation of the weight factor $\psi_{\lambda/\mu}$ defined in \eqref{Psi}
to elements in $\cP^+_k$ equals the cardinality
of the set 
\begin{equation} \label{setPsi}
 \{ w \in \rS^\mu ~|~  \lambda_i -(\mu \circ w)_i =0,1 \text{ for all } i \in [k]  \} \;.
\end{equation}
\end{itemize}
\end{lemma}
\begin{proof}
Noting that the the distinct permutations of a fixed weight are labelled by the minimal length representatives $w\in \rS^{\mu}$ the stated expressions are immediate consequences of \eqref{Setf}, \eqref{SetTheta} and \eqref{SetPsi}, respectively.
\end{proof}

\subsection{Adjacent column tableaux}

In this Section we will derive expansions of the symmetric functions $h_{\lambda/\mu}$ and $e_{\lambda/\mu}$ into power sums. In light of the definition \eqref{ch} these expansions amount to computing the inverse image of $h_{\lambda/\mu}$ and $e_{\lambda/\mu}$ under the characteristic map.

Consider the product expansion
\begin{equation}
 m_\mu p_\nu = \sum_{\lambda} \varphi_{\lambda/\mu} (\nu) m_\nu\,.
\end{equation}
The latter defines the coefficients $\varphi_{\lambda/\mu}(\nu)$ uniquely.
\begin{lemma}
The functions $h_{\lambda/\mu}$ and $e_{\lambda/\mu}$ have the expansions 
\begin{equation}\label{skewhe2p}
 h_{\lambda/\mu} = \sum_{\nu}\frac{\varphi_{\lambda/\mu}(\nu)}{\rm z_\nu}\, p_\nu \qquad\text{and}\qquad
 e_{\lambda/\mu} = \sum_{\nu}\frac{\varphi_{\lambda/\mu}(\nu)}{\rm z_\nu}\, \epsilon_\nu p_\nu\,,
\end{equation}
where $\epsilon_\lambda=(-1)^{|\lambda|-l(\lambda)}$.
\end{lemma}
\begin{proof}
The proof is similar to the one of Lemma \ref{lem:ExpansionHE} exploiting the known inner product formula \eqref{Hall}
 and the action $\omega(p_\lambda)=\epsilon_\lambda p_\lambda$ of the involution $\omega$ defined earlier.
\end{proof}
Recall that $p_\lambda = \sum_{\mu} R_{\lambda \mu} m_\mu$, where $R_{\lambda \mu}$
equals the number of 
 ordered set partitions $ (B_1,\dots,B_{l(\mu)})$ of $[l(\lambda)]$ such that 
$
    \mu_j = \sum_{i \in B_j} \lambda_i
$; see e.g. \cite[Ch. 7.7, Prop 7.7.1]{stanley_fomin_1999}.
 It then follows that 
 \begin{equation} \label{varphiR}
 \varphi_{\lambda/\mu}(\nu) = \sum_{\sigma} R_{\nu \sigma}  f_{ \sigma \mu}^\lambda
 \end{equation}
 We now derive an alternative expression for $\varphi_{\lambda/\mu}(\nu)$
 involving a special subclass of semi-standard tableaux or column strict RPP. Recall that in this case the RPP consists of a sequence of horizontal strips.
 
 We call a column strict RPP $T:\lambda / \mu\to\mb{N}$ an {\em adjacent column tableau} (ACT) if the boxes of each horizontal strip $T^{-1}(i)$ are located in adjacent columns of the skew diagram $\lambda/\mu$. Let $T=\{\lambda^{(i)}\}_{i=0}^l$ be an ACT then we define $\varphi_T=\prod_{i\geq 1} m_{j_i}(\lambda^{(i)})$, where  $m_{j_i}(\lambda^{(i)})$ is the multiplicity of the part $j_i$ in $\lambda^{(i)}$ which contains the rightmost box of the horizontal strip $T^{-1}(i)$. As usual we define the weight $\op{wt}(T)$ as the vector whose components are the number of boxes in each horizontal strip, i.e. $\op{wt}_i(T)=|T^{-1}(i)|$.

\begin{exam}
Consider the partitions $\lambda=(5,5,3,2)$, $\mu=(3,2,1,1)$ and
the composition $\alpha=(2,2,3,1)$. There are four adjacent column tableaux of shape $\lambda / \mu$ and weight $\op{wt}(T)=\alpha$, namely
 \[
 T=  { \Yvcentermath1 \young(\emptybox\emptybox\emptybox11,\emptybox\emptybox233,\emptybox23,\emptybox4)}, \quad 
  { \Yvcentermath1 \young(\emptybox\emptybox\emptybox11,\emptybox\emptybox234,\emptybox23,\emptybox3) }, \quad 
   { \Yvcentermath1 \young(\emptybox\emptybox\emptybox22,\emptybox\emptybox133,\emptybox13,\emptybox4) }, \quad 
  { \Yvcentermath1  \young(\emptybox\emptybox\emptybox22,\emptybox\emptybox134,\emptybox13,\emptybox3)  }\;. 
 \]
One finds the coefficients $\varphi_{T}=2,2,4,4$ (from left to right).
\end{exam}

We call an ACT of weight $\op{wt}=(0,\ldots,0,r,0,\ldots)$ an adjacent horizontal strip of length $r$ and in this case simply write $\varphi_T=\varphi_{\lambda/\mu}$.

\begin{lemma}\label{lem:varphi2ACT}
We have the equality
\begin{equation} \label{varphiLambdaMuNu}
 \varphi_{\lambda/\mu}(\nu) = \sum_{T} \varphi_T \;, 
\end{equation}
where the sum is over all ACT of shape $\lambda / \mu$ and weight $\nu$.
\end{lemma}

\begin{proof}
We first show the validity of the following product expansion
 \begin{equation} \label{plambdammu}
    m_\mu  p_r = \sum_{\lambda} \varphi_{\lambda/\mu} m_\lambda 
   \end{equation}
 where the sum runs over all partitions $\lambda$ such that $\lambda/\mu$ is an adjacent horizontal strip of length $|\lambda/\mu|=r$.

Since each monomial appearing in the product $m_\mu p_r $ is of the form $ x^\alpha x_l^r$ for some $l>0$ and $\alpha \sim \mu$,
 the monomial $x^\lambda$ appears in $m_\mu p_r$ only if $\lambda$ is obtained from $\mu$ by removing a part equal to $a-1$, for some $a \ge 1$, and then adding a part equal to $a-1+r$. In other words, $\lambda$ must be the partition $\mu(a,r)$ obtained by
 adding one box per column in the diagram of $\mu$, starting at column $a$ and ending at column $a+r-1$ for a total of $r$ boxes.
This is equivalent to removing a part equal to $a-1$ from $\mu$ (or removing no parts if $a=1$) and adding a part equal to $a-1+r$.
 
Therefore, the monomials in $m_\mu p_r$ equal to $x^{\lambda}$ must be of the form
$$
 x^{(\mu_1,\dots,\mu_{i-l-1},\mu_j,\mu_{i-l},\dots,\mu_{j-1},\mu_{j+1},\dots)} \, x_{i-l}^r 
$$
for $l=0,\dots,m_{a-1+r}(\mu)$, where $i$ and $j$ are respectively the smallest indices for which $\mu_i<a-1+r$ and
$\mu_j=a-1$. This implies that 
$
 m_\mu  p_r
 = \sum_{a} (m_{a-1+r}(\mu)+1) m_{\mu(a,r)}
$, where the sum is over all $a \ge 1$ such that $m_{a-1}(\mu) \neq 0$.  But according to the definition of $\varphi_{\lambda/\mu}$ this
equals \eqref{plambdammu}.
Equation \eqref{varphiLambdaMuNu} now follows by repeated application of \eqref{plambdammu} to the product $m_\mu p_\nu$,
which also implies that $\varphi_{\lambda/\mu}(\nu)=\varphi_{\lambda/\mu}(\beta)$ for $\beta \sim \nu$.
\end{proof}

\begin{coro} \label{TransitionPMcol}
The matrix $R$ in \eqref{varphiR} has the following alternative expression
 \begin{equation}
  R_{\lambda \mu} = \sum_{T} \varphi_T
 \end{equation}
where the sum is over all $ACT$ of shape $\mu$ and weight $\lambda$.
\end{coro}
Note the difference between this and Corollary \ref{LMtableaux}, since in general $R_{\lambda \mu} \neq R_{\mu \lambda}$.

\section{The irreducible representations of $\GS$}
As already explained in the main text the finite-dimensional irreducible representations of $\GS$ are induced by the irreducibles of the normal subgroup $\mc{N}$ and their associated stabiliser groups in $\GS$ \cite{osima1954representations}. It will be convenient to identify the irreducible representations of $\cN$ with weights in the set \eqref{alcove} via $\lambda:y^\mu\mapsto\zeta^{(\lambda,\mu)}$, where $\zeta$ is a fixed primitive root of unity of order $n$. Given $\lambda\in\alc$ its associated stabiliser group in $\GS$ is defined as
\begin{equation*}
G_\lambda=\{g\in \GS~|~(g\cdot y ^{\mu%
}\cdot g^{-1})^{\lambda}= y^{(\mu, \lambda)%
},~\forall y ^{\mu}\in \mathcal{N}\}\;.
\end{equation*}%
The latter is canonically isomorphic to the Young subgroup of the weight $\lambda$, i.e. 
$\rS_{\lambda}=S%
_{m_{n}(\lambda)}\times \cdots \times \rS_{m_{1}(\lambda)}\subset \rS_{k}$ where $m_i(\lambda)$ is the multiplicity of the part $i$ in $\lambda$; see \cite{osima1954representations} for details. %
Thus, the irreducible representations of the stabiliser group $G_\lambda$ are of the form $\bigotimes_{i=1}^{n}\mathbb{S}_{\lambda ^{(i)}}$, where $\lambda^{(i)}\vdash m_i(\lambda)$ are partitions and $\mathbb{S}_{\lambda ^{(i)}}$ denotes the corresponding Specht module of the symmetric group $\rS_{m_i(\lambda)}$. The basis elements in  $\bigotimes_{i=1}^{n}\mathbb{S}_{\lambda ^{(i)}}$ are tensor products of Young tableaux,  $T=T^{(1)}\otimes \cdots \otimes T^{(n)}$, which define a map $T%
:\boldsymbol{\lambda} \rightarrow \lbrack k]$ with $\bs{\lambda}=(\lambda^{(1)},\ldots,\lambda^{(n)})$ such that the numbers assigned to the boxes in
each Young diagram $\lambda ^{(i)}$ are strictly increasing in rows (left to right) and columns (top to bottom). We call such a map $T$ a $n$-tableau and $\bs{\lambda}$ is called a $n$-multipartition. Conversely, given a $n$-tableau $T:\boldsymbol{\lambda} \rightarrow \lbrack k]$, let $\lambda(T)\in\alc$ be the unique intersection point of $\alc$ with the $\rS_k$-orbit of the weight $p=p(T)\in\cP_k$ defined by setting $p_j(T)=i$ if $T^{-1}(j)\in\lambda^{(i)}$. Explicitly, the $\cN$-action on a $n$-tableau $T$ then reads
\begin{equation} \label{actionzeta}
y ^{\mu}\cdot T= \zeta^{(\mu, p(T))} T,\qquad\mu\in\cP_k
\end{equation}%
For completeness, we also recall the definition of the action of the elementary transpositions $\sigma_i$ on $n$-tableaux; compare with \cite{pushkarev1999representation}. Introduce the numbers
\begin{equation}
t_{i}=\left\{ 
\begin{array}{cc}
\frac{1}{c_{T}(i+1)-c_{T}(i)}~, & \text{if }%
T^{-1}(i),T^{-1}(i+1)\in \lambda ^{(r)}, r\in[n] \\ 
0~, & \text{else}%
\end{array}%
\right. \;,
\end{equation}
where $c_{T}(i)=b-a$ is the\emph{\ content} of the box $(a,b)\in\lambda^{(r)}$ in $%
T$ containing the entry $i$. Denote by $T_{(i,i+1)}$ the $n$-tableau which 
is obtained from $T$ by swapping the entries $i$ and $i+1$ if the result is another $n$-tableau,
otherwise set $T_{(i,i+1)}=0$. Then
\begin{equation}
T\cdot\sigma_i=t_{i}T+\sqrt{1-t_{i}^{2}}~%
T_{(i,i+1)}\;.
\end{equation}%

Thus, in summary the irreducible representations of $\GS$ are in bijection with $n$-multipartitions $\boldsymbol{\lambda}=(\lambda^{(1)},\ldots ,\lambda ^{(n)})$ satisfying 
\begin{equation}
|\boldsymbol{\lambda}|=\sum_{i=1}^n|\lambda^{(i)}|=k\;.
\end{equation}
We denote the resulting module by $\mL(\boldsymbol{\lambda})$. These are all irreducible modules of 
$\GS$ and, furthermore, they have dimension
\begin{equation}\label{dimL}
\dim \mL(\boldsymbol{\lambda})=\frac{k!}{%
|\lambda^{(1)}|!\cdots |\lambda^{(n)}|!}\,\prod_{i=1}^{n}f_{\lambda^{(i)}},
\end{equation}%
where $f_\mu$ is the number of standard tableaux of shape $\mu$,
$$f_{\mu}=\frac{|\mu|!}{\prod_{(i,j)\in\mu}h_\mu(i,j)},\qquad h_\mu(i,j)=\mu_i+\mu'_j-i-j+1\,.$$
If $\mu=\emptyset$ we set $f_\emptyset=1$.

\end{document}